\documentclass[twoside, 11pt]{amsart}

\usepackage{fullpage}
\usepackage{amscd}
\usepackage{amsmath}
\usepackage{amssymb}
\usepackage{amsthm}
\usepackage{color}
\usepackage{graphicx}
\usepackage{hyperref}
\usepackage[noabbrev,capitalize]{cleveref}
\usepackage{mathrsfs}
\usepackage{mathtools}
\usepackage{pict2e}
\usepackage{stackrel}
\usepackage[T1]{fontenc}
\usepackage{tikz-cd}
\usepackage[utf8]{inputenc}
\usepackage{wasysym}
\usepackage[all]{xy}
\usepackage{xcolor}
\usepackage{a4wide}

\let\oldtocsection=\tocsection
\let\oldtocsubsection=\tocsubsection

\renewcommand{\tocsection}[2]{\hspace{0em}\vspace{0.2em}\rule{0pt}{14pt}\oldtocsection{#1}{#2}\bf}
\renewcommand{\tocsubsection}[2]{\hspace{2em}\oldtocsubsection{#1}{#2}}

\newcommand{\changelocaltocdepth}[1]{%
	\addtocontents{toc}{\protect\setcounter{tocdepth}{#1}}%
	\setcounter{tocdepth}{#1}%
	}

\usepackage{tikz}
\usetikzlibrary{shapes,fit,positioning,calc,matrix}
\tikzset{
  optree/.style={scale=.5,thick,grow'=up,level distance=10mm,inner sep=1pt},
  comp/.style={draw=none,circle,fill,line width=0,inner sep=0pt},
  dot/.style={draw,circle,fill,inner sep=0pt,minimum width=3pt},
  circ/.style={draw,circle,inner sep=1pt,minimum width=4mm},
  emptycirc/.style={draw,circle,inner sep=1pt,minimum width=2mm},
  root/.style={level distance=10mm,inner sep=1pt},
  leaf/.style={draw=none,circle,fill,line width=0,inner sep=0pt},
  nodot/.style={draw,circle,inner sep=1pt},
}

\definecolor{Chocolat}{rgb}{0.36, 0.2, 0.09}
\definecolor{BleuTresFonce}{rgb}{0.215, 0.215, 0.36}
\hypersetup{colorlinks, citecolor=BleuTresFonce, linkcolor=Chocolat}

\usepackage[lf]{Baskervaldx} % lining figures
\usepackage[bigdelims,vvarbb]{newtxmath} % math italic letters from Nimbus Roman
\usepackage[cal=boondoxo]{mathalfa} % mathcal from STIX, unslanted a bit

\allowdisplaybreaks

\theoremstyle{plain}

\newtheorem{lemma}{Lemma}[section]
\newtheorem{theorem}[lemma]{Theorem}

\newtheorem{corollary}[lemma]{Corollary}
\newtheorem{proposition}[lemma]{Proposition}

\newtheorem*{convention}{\sc Convention}

\theoremstyle{definition}

\newtheorem{definition}[lemma]{Definition}
\newtheorem{remark}[lemma]{\sc Remark}
\newtheorem{example}[lemma]{\sc Example}
\newtheorem{notation}[lemma]{\sc Notation}
\newtheorem{assumption}[lemma]{\sc Assumption}

\def\Hbul{H^\bullet}
\def\dE{d_\mathrm{E}}
\def\DE{\mathrm{d}_\mathrm{E}}
\def\ad{\mathrm{ad}}
\def\Co{\mathbb{C}}
\newcommand{\NNs}{\overline{\NN^2}}
\newcommand{\K}{\mathbb{k}}
\newcommand{\ZZ}{\mathbb{Z}}

\newcommand{\Ga}{\Gamma}
\newcommand{\Bij}{\mathsf{{Bij}}}
\def\M{\mathcal{M}}
\def\Gra{\mathrm{CGra}}

\def\EGra{\mathrm{EGra}}
\def\S{\mathcal{S}}
\def\Lie{\mathrm{Lie}}
\def\T{\mathrm{T}}
\def\Com{\mathrm{Com}}
\def\grt{\mathfrak{grt}}
\def\GRT{\mathrm{GRT}}
\def\oM{\overline{\mathcal{M}}}
\def\CC{\mathrm{C}}
\def\PP{\mathrm{P}}
\def\BCH{\mathrm{BCH}}
\def\aa{\mathfrak{a}}
\def\SDL{\S\Delta\Lie}
\def\SL{\S\Lie}
\def\F{\mathcal{F}}
\def\Aut{\mathrm{Aut}}
\def\MC{\mathrm{MC}}

\def\calMC{\mathcal{MC}}

\def\1{\mathbb{1}}

\newcommand{\G}{\mathscr{G}}
\newcommand{\NN}{\mathbb{N}}
\def\TTT{\mathcal{T}}
\newcommand{\Sy}{\mathbb{S}}
\newcommand{\Syg}{\widehat{\mathbb{S}}}
\DeclareMathOperator{\id}{id}
\newcommand{\ac}{{\scriptstyle \text{\rm !`}}}
\def\B{\mathrm{B}}
\def\P{\mathcal{P}}
\def\whP{\widehat{\mathcal{P}}}
\def\C{\mathcal{C}}
\DeclareMathOperator{\End}{End}
\DeclareMathOperator{\EEnd}{\mathcal{E}\!\mathcal{n}\!\mathcal{d}}
\DeclareMathOperator{\Hom}{Hom}
\DeclareMathOperator{\Tw}{Tw}
\def\Smod{\Syg\textsf{-}\mathsf{Mod}}

\def\g{\mathfrak{g}}
\def\gh{\mathfrak{g}[[\hbar]]}
\def\gth{\mathfrak{g}^\hbar}

\def\d{\mathrm{d}}
\def\ba{\bar{\mathfrak{a}}}
\def\b{\mathfrak{b}}

\def\bb{\bar{\mathfrak{b}}}
\def\R{\mathfrak{R}}
\def\m{\mathfrak{m}}
\def\GGra{\mathrm{TautGra}}
\def\Mrk{\mathrm{M}}

\def\Frob{\mathcal{F}\!\mathcal{r}\!\mathcal{o}\!\mathcal{b}}

\def\od{\mathrm{odd}}
\def\ev{\mathrm{even}}
\def\Z2{\ZZ/2\ZZ}
\def\Khbar{\K[[\hbar]]}
\def\giv{\mathfrak{giv}}
\def\GIV{\mathrm{GIV}}
\def\givgrt{\mathfrak{ggrt}}
\def\GIVGRT{\mathrm{GGRT}}
\def\BR{\mathrm{BR}}

\author{Vladimir Dotsenko}
\address{Institut de Recherche Math\'ematique Avanc\'ee, UMR 7501, Universit\'e de Strasbourg et CNRS, 7 rue Ren\'e-Descartes, 67000 Strasbourg, France}
\email{vdotsenko@unistra.fr}

\author{Sergey Shadrin}
\address{Korteweg-de Vries Institute for Mathematics, University of Amsterdam, P. O. Box 94248, 1090 GE Amsterdam, The Netherlands}
\email{s.shadrin@uva.nl}

\author{Arkady Vaintrob}
\address{Department of Mathematics, University of Oregon, OR 97403, Eugene, USA}
\email{vaintrob@uoregon.edu}

\author{Bruno Vallette}
\address{Universit\'e Sorbonne Paris Nord, Laboratoire de G\'eom\'etrie, Analyse et Applications, LAGA, CNRS, UMR 7539, F-93430, Villetaneuse, France}
\email{vallette@math.univ-paris13.fr}

\title{Deformation theory of Cohomological Field Theories}
\date{
\today
}

\keywords{Deformation theory, moduli spaces of stable curves, cohomological field theories, modular operads, homotopical algebra, graph complexes, Grothendieck--Teichm\"uller groups, Givental group.}

\thanks{2020 \emph{Mathematics Subject Classification.} Primary 18M85; Secondary 18G85, 18M70, 53D55, 14H10, 53D45.
\newline
B.V. was supported by the Institut Universitaire de France. 
S.S was supported by the Netherlands Organization for Scientific Research.
During the work on the final version of the paper, V.D. and B.V. were supported by the French national research agency project ANR-20-CE40-0016.}

\begin{document}

\begin{abstract}
 
 We develop the deformation theory of cohomological field theories (CohFTs), which is done as a special case of a general deformation theory of morphisms of modular operads. 
 This leads us to introduce two new natural extensions of the notion of a CohFT: homotopy 
 (necessary algebraic toolkit to develop chain-level Gromov--Witten invariants) and quantum (with examples found in the works of Buryak and Rossi on integrable systems). 
 The universal group of symmetries of morphisms of modular operads, based on Kontsevich's graph complex, is shown to be trivial. Using the tautological rings on moduli spaces of curves, we introduce a natural enrichment of Kontsevich's graph complex. This leads to universal groups of nontrivial symmetries of both homotopy and quantum CohFTs, which, in the latter case, is shown to contain both the prounipotent Grothendieck--Teichm\"uller group and the Givental group.
\end{abstract}

\maketitle

\setcounter{tocdepth}{2}

\tableofcontents

\section*{Introduction}
\changelocaltocdepth{1}

\subsection*{Moduli spaces of curves and related symmetries}
In his celebrated
\emph{Esquisse d'un programme} \cite{Grothendieck97}, Grothendieck observed that the absolute
Galois group $\mathrm{Gal}\big(\overline{\mathbb{Q}}/\mathbb{Q}\big)$ acts on
the (geometric) fundamental groupoids of the moduli stacks $\mathcal{M}_{g,n}$ of genus $g$ curves
with $n$ marked points and proposed to study this group  via its action on the \emph{Teichm\"uller tower},
the collection of all these groupoids for all $g$ and $n$ connected by natural operations induced
by gluing and  forgetting marked points. Bely\u{\i}'s theorem~\cite{Belyi79} already shows that the
geometric fundamental group of $\mathcal{M}_{0,4}$ contains a faithful action of the absolute Galois group.
This motivated Drinfeld to introduce in~\cite{Drinfeld90} several flavors of
\emph{Grothendieck--Teichm\"uller} groups whose profinite version contains the absolute Galois
group and is conjecturally isomorphic to it.

\medskip

Since the \'etale (resp.\ Betti) fundamental group of $\mathcal{M}_{0,4}$ is the profinite (resp.\ prounipotent) completion of the free group on two generators,
Drinfeld's Grothendieck--Teichm\"uller groups can be defined in purely algebraic terms.
They also play a prominent role in deformation theory, quantum groups and braided monoidal categories.
One of the sources of inspiration for the present work was to bring
Drinfeld's algebraic definitions closer to Grothendieck's original plan to study an arithmetic object via geometrical methods related to the moduli stacks of curves and their structural operations.

\medskip

On the other hand, the Deligne--Mumford--Knudsen stacks  $\overline{\mathcal{M}}_{g,n}$  of stable marked curves compactifying
$\mathcal{M}_{g,n}$ are also connected by similar operations of gluing and forgetting marked points.
This collection of spaces forms the basis for the notion of a  \emph{cohomological field theory}
(CohFT) which was introduced by Kontsevich and Manin~\cite{KontsevichManin94, Manin99} in order to
formalize the algebraic structures of quantum cohomology and
Gromov--Witten invariants. Since then they have become the standard tool in the study of the mirror symmetry and related areas (singularity theory, symplectic geometry, theory of integrable hierarchies, to name just a few).
Inspired by the operations on $\overline{\mathcal{M}}_{g,n}$ Getzler and Kapranov introduced
in~\cite{GetzlerKapranov95} a general notion of a \emph{modular operad}. In
this language, a CohFT is an algebra over the modular operad $H_*(\overline{\mathcal{M}}_{g,n})$~.

\medskip

In his work on formal properties of Gromov--Witten invariants, Givental~\cite{Givental01}
introduced a remarkable universal group which, as was shown by Teleman in~\cite{Teleman2012},
naturally acts on the set of CohFTs. In a different direction, Pandharipande--Zvonkine described, in a
recent paper~\cite{PZ18},  an explicit method to construct non-trivial examples of CohFTs by
deforming topological field theories using minimal cohomology classes on  $\overline{\mathcal{M}}_{g,n}$~.
\medskip

One of the overarching goals of the current work was to find and develop a deformation-theoretic
context which would allow us to unite under one umbrella all these different universal symmetries of objects
governed by the natural operations induced by the spaces  $\overline{\mathcal{M}}_{g,n}$~.
Accomplishing this is, what we believe, our main result.

\subsection*{Homotopy CohFT structures}
In this paper, we develop the deformation theory of morphisms of modular operads following the general operadic methods described in~\cite{MerkulovVallette09I, MerkulovVallette09II, LodayVallette12}.
We start with the definition of a modular operad as an algebra over a certain monad of graphs; this classical approach will play a crucial role latter on.
On a more innovative side, we also encode the category of modular operads with a certain
 operad $\mathcal{O}$ colored in \emph{groupoids}.

\medskip

The Koszul duality for operads has recently been extended to groupoid-colored operads by Ward in
\cite{Ward19}, the goal being to obtain a notion of a homotopy modular operad and an associated homotopy transfer theorem, and to introduce Massey products in graph homology.
Notice that we do not follow the approach of (the first version of) \cite{Ward19} exactly.
Instead, we use the notion of the Koszul dual groupoid-colored \emph{cooperad},
which allows us to avoid the finiteness assumptions required in \emph{op.\ cit.}.
This relates to the principle that ``operads encode algebras and cooperads encode homotopy algebras''.
(After the present work first appeared on the ArXiv, Ward updated his paper to deal with cooperads this time.)
With our approach, we recover in a more conceptual way various important constructions related to
modular operads including the Feynman transform and its (homotopy) inverse. This also allows us to clarify various degrees and signs appearing in these constructions.

\medskip

We apply the calculus of modular operads in a different way in order to introduce a notion of \emph{cohomological field theories up to homotopy}, or \emph{homotopy CohFTs} for short, together with their deformation theory. Note that if we restrict ourselves to the genus zero case, homotopy CohFTs can be described in a very economical way, since the corresponding operad has the Koszul property. Since this property fails for higher genera, one has to consider bigger resolutions. Here we choose to apply the functorial bar-cobar resolution (Feynman transforms), available for any modular operad, which gives a new homotopy invariant generalization of CohFTs with homotopic properties expected to be satisfied by chain-level Gromov--Witten theories.

\medskip

While fully constructed examples of homotopy CohFTs, for instance, chain-level Gromov--Witten theories, are not yet available in the literature (and this will presumably require a lot of hard
analytic work), the algebraic toolkit can be prepared in advance and may suggest natural ways to pack the answer and potentially provide help for their constructions. Notice that some first steps in the direction of chain-level Gromov--Witten invariants are performed  in \cite{seidel2019formal, MannRobalo}.

\medskip

It seems that it is unavoidable to work on the chain level in the Gromov--Witten theory once one has to think about the choices involved in the open-closed and real ramifications of the theory, as well as in the B-model categorical constructions of Gromov--Witten-like classes. But even without these natural sources, it would be just unnatural to ignore the homotopy structures behind CohFTs --- as an analogy (which is in fact the simplest pieces of the structure that we study, in some cases) think of the relationships between the product structure on cohomology and the higher Massey products.

\subsection*{Quantum homotopy CohFTs and the Buryak--Rossi functor}
Our formalism allows to study infinitesimal deformations, formal deformations, and obstructions to
deformations for homotopy CohFTs using traditional methods of deformation theory.
As an application we give a conceptual explanation and an extension of the recent construction, due
to Pandharipande--Zvonkine~\cite{PZ18}, of a CohFT obtained by deforming a topological quantum field  theory in the
direction of a minimal cohomology class on $\overline{\mathcal{M}}_{g,n}$.
Interestingly enough, these deformations, a priori constructed only infinitesimally, happen to be global ones.

\medskip

An interesting feature of our description of homotopy CohFTs as solutions of the Maurer--Cartan equation in some
differential graded (dg) Lie algebra 
is that it leads to yet another notion --- of a \emph{quantum homotopy CohFT} ---
which occurs naturally and is technically necessary for various steps in our construction.
It is obtained by replacing our Maurer--Cartan equation over the ground field $\K$ by an
$\hbar$-deformation over $\K[[\hbar]]$,
the so-called  \emph{quantum master equation} in the physics literature.
Surprisingly, examples of these quantum CohFT structures do occur naturally in the theory of integrable hierarchies associated to CohFTs, namely they are the core ingredients of the Buryak--Rossi program of quantization of such integrable systems~\cite{BuryakRossi2016}.
Specifically, we observe that the
quantization procedure of Buryak--Rossi, that basically involves multiplication of a CohFT by the
properly parametrized Hodge class, gives rise to a
functor,  that we call the \emph{Buryak--Rossi functor},
between the Deligne groupoids of classical and quantum homotopy CohFT structures.

\subsection*{Universal deformation group}
The deformation complex of homotopy CohFTs can be considered within a more general framework of morphisms of modular operads. Then it carries two fundamental algebraic structures:
the minimal one, needed only for writing the Maurer--Cartan equation, and the maximal one, encoding all natural
operations on it. This gives rise to two operads, one embedded into the other.

\medskip

One of our main objects of study in this paper is
the deformation complex of this embedding of operads.
This chain complex, in the quantum case, turns out to be isomorphic to the
 $\K[[\hbar]]$-extension of the Kontsevich graph complex~\cite{Kontsevich97,Willwacher15}, see \cite{MW14}.
The even homology of this deformation complex carries a structure of a prounipotent
group  which we call the \emph{universal deformation group of a morphism of quantum modular operads}.
We show that this group contains the prounipotent Grothendieck--Teichm\"uller group and that it acts
naturally
by universal and explicit formulas on the moduli space of gauge
equivalence classes of quantum Maurer--Cartan elements.
Our approach is essentially the same that was used by Merkulov and Willwacher in~\cite{MW14}, but we
present the results in detail using our new conceptual interpretation and
the tools of the pre-Lie deformation theory developed in~\cite{DSV16}.

\medskip

Back in the classical case,
we face the following problem: the corresponding deformation complex is almost acyclic
and so the deformation theory is not very interesting.
(This issue was independently observed in~\cite{KhorWillZiv2017}, where the authors
were specifically looking for an extension of the differential in the Kontsevich graph complex that would make it acyclic.)
However, the situation becomes much more interesting, if we consider some additional structure.

\subsection*{Givental--Grothendieck--Teichm\"uller group}

The maximal algebraic structure, encoding all natural
operations that use only the structure of morphisms of modular operads on the deformation complexes
of classical and quantum homotopy CohFTs can be extended further once we use the specifics of the
Deligne--Mumford--Knudsen modular operad. Namely, we use that it is a Hopf modular operad and that
is contains a distinguished and easy to describe suboperad of tautological classes~\cite{Faber}.
This allows us to extend the algebraic structures on
the deformation complexes of classical and quantum homotopy CohFTs
with the action of tautological classes combined with the non-stable parts of the endomorphism modular operad of the target vector space (or, more generally, one can think of non-stable components of the target modular operad). This gives rise to an embedding of the minimal operad needed to control the Maurer--Cartan equations into a much larger operad.
The deformation complex of this embedding is a huge extension of the Kontsevich and
Merkulov--Willwacher graph complexes. It has
non-trivial homology both in the classical and in the quantum case.
In particular, we obtained
the following results:
\begin{itemize}
\item In the classical case, we are able to compute the full homology and it matches
  the Givental group / Lie algebra action on CohFTs constructed by Teleman~\cite{Teleman2012}.
  So, quite
 surprisingly,
our construction recovers
  the full Givental group now acting in a natural way on \emph{homotopy} CohFTs.

\item In the quantum case,
while the complete homology computation is currently out of reach, the resulting group is huge.
To demonstrate that how interesting and non-trivial this group is, we show that it contains
both the prounipotent Grothendieck--Teichm\"uller group and the Givental group acting in a natural way on quantum homotopy CohFTs.
\end{itemize}

\medskip

We call the latter group the \emph{Givental--Grothendieck--Teichm\"uller group}.
Both
a full description of its structure and the consequences of its action on quantum homotopy CohFTs
(for instance, for the Buryak--Rossi theory of quantization of integrable hierarchies of
topological type) are very interesting
questions for future research.

\changelocaltocdepth{2}
\medskip

\paragraph{\bf Layout.} The first section recalls the notion of a modular operad together their homological constructions. In the second section, we develop the deformation theory of morphisms of modular operads and study the particular case of the deformation theory of (homotopy) CohFTs. The third section deals with the universal deformation group of morphisms of modular operads, its relationship with the prounipotent Grothendieck--Teichm\"uller group, and its action of the moduli space of gauge equivalent classes of Maurer--Cartan elements.
The last section introduces a new graph operad enriched with tautological classes whose deformation complex is shown to act on quantum homotopy CohFTs and to contain the prounipotent Grothendieck--Teichm\"uller group and the Givental group.

\medskip

\paragraph{\bf Conventions.}
We denote by $\K$ a fixed ground field of characteristic $0$.
We mainly work over the
category of differential graded (dg) vector spaces ($\ZZ$-graded by default)
with the usual monoidal structure including the Koszul sign rule.
At three points of the paper (\cref{subsec:DefCohFT},  \cref{subsec:CycleAct}, and \cref{sec:GGT}), we
switch from the $\ZZ$-grading to the induced $\Z2$-grading.

We use the homological degree convention, for which differentials have degree $-1$.
To accommodate that in the cases of cochain complexes and their cohomology, we place elements of cohomological degree $k$ in the opposite homological degree $-k$."
We denote by $s$
a one-dimensional vector space concentrated in degree $1$ (or abusing notation, a fixed non-zero
element in it), so tensoring with it produces the shift functor.
For a dg vector space $V$ we denote by $V^{\odot n}:=V^{\otimes n}/{\Sy_n}$ its $n$th symmetric power.

\medskip
\paragraph{\bf Acknowledgments.}
We express our appreciation to Mohammed Abouzaid, Alexander Alexandrov, Michael Batanin, Alexandr Buryak, Gabriel Drummond-Cole, Johan Leray, Sergei Mer\-ku\-lov, Dan Petersen, John Terilla, Alexander Voronov, Ben Ward, and Dimitri Zvonkine for useful discussions and references. Special thanks are due to Thomas Willwacher for numerous insightful remarks on a preliminary version of the paper.

\section{Modular operads}

In this section, we recall the necessary material on modular operads. We choose to
present them as algebras over
a groupoid-colored operad $\mathcal{O}$ as in~\cite{Ward19} mainly because it has a quadratic presentation which is Koszul. This allows to apply all the results of the general operad calculus~\cite{LodayVallette12} \emph{mutatis mutandis}.
There exist other tools for modeling various
types of operads and in particular modular operads, like
Feynman categories~\cite{KaufmannWard17} or operadic categories~\cite{BataninMarkl18} for instance.
The former theory has been shown to be equivalent to the theory of colored operads
in~\cite{Caviglia15, BKW18} and the latter has not yet been extended to colored operads.
Another reason we present here the calculus of modular operads in this way is  to have as few prerequisites
as possible. This  should help non-experts to follow the rest of the text.
 We do not  claim much
novelty here except in the presentation and in the interpretation of the Feynman transforms (direct
and  homotopy inverse) of~\cite{GetzlerKapranov95} as the bar and cobar constructions associated to
the universal Koszul morphism $\kappa : \mathcal{O}^{\ac} \to \mathcal{O}$. We conclude this section with the new notion of a \emph{homotopy cohomological field theory}, which is modeled by a cofibrant replacement of the modular operad
$H_\bullet\big(\overline{\mathcal{M}}_{g,n}\big)$; so the general theory ensures that this generalization carries the required homotopical properties.

\subsection{Definitions}
\label{subsec:ModOp}
We recall the notion of a modular operad from~\cite{GetzlerKapranov98}.
Let $(\mathsf{C}, \otimes)$ be a symmetric monoidal category, with colimits commuting with the monoidal product.
We consider the groupoid $\Syg$
that has
\begin{description}
\item[\sc Objects] the set $\NNs\coloneqq\left\{(g,n)\in \NN^2 \mid 2g+n>2\right\}
=\NN^2\backslash \left\{
(0,0), (0, 1), (0,2), (1,0)
\right\}$ \ and
\item[\sc Morphisms]
  $\Hom_{\Syg}((g,n), (g',n')) \coloneq \Sy_n$,  \ if $(g',n') =(g,n)$, \ and $\emptyset$  otherwise.
\end{description}

\begin{definition}[Stable $\Sy$-module]
A \emph{stable $\Sy$-module} is a module over the groupoid $\Syg$, that is
a collection
$\left\{\P_g(n)\right\}_{(g,n)\in \NNs}$ of elements of $\mathsf{C}$
such that each $\P_g(n)$ has an action of the symmetric group $\Sy_n$.
The corresponding category is  denoted by $\Smod$.
\end{definition}

We will also consider the
category $\widehat{\Bij}$ whose objects are pairs $(g, X)$, where $X$ is a finite set satisfying
$2g+|X|>2$,
with morphisms $(g_1,X_1)\to (g_2,X_2)$
given by bijections $ \mathrm{Bij}(X_1,X_2)$ when $g_1=g_2$~.
The groupoid $\Syg$ is the skeletal category of $\widehat{\Bij}$
and
thus the data
 of a stable $\Sy$-module is equivalent to the data of a $\widehat{\Bij}$-module under the formula
\[\P_g(X)\coloneqq \Big(\prod_{f\in \mathrm{Bij}(\underline{n},\, X)} \P_g(n)\Big)\big/\sim\ ,
\]
where $|X|=n$~,
\ $\underline{n}\coloneq \{1, \ldots, n\}$,  and $(f, \mu)\sim (g, \nu)$ when $\nu={g}^{-1}f\cdot\mu$.
We will be using both the coordinate-free and skeletal descriptions interchangeably, choosing the more appropriate one every time.

\medskip

We will be working with unoriented graphs which may have multiple edges, leaves (legs) and tadpoles (loops).
We refer the reader to~\cite[Sections~2.5-2.16]{GetzlerKapranov98} for complete definitions and more details.

\begin{definition}[Stable graphs]
  A \emph{labeled graph} is a connected graph $\gamma$
in which every vertex $\nu\in \text{vert}(\gamma)$ is labeled by a non-negative integer $g(\nu)$
called the \emph{internal genus} of $\nu$.
 The \emph{total genus} of a labeled graph $\gamma$ is the sum of its topological genus
 $b_1(\gamma)$~, the first Betti number of its geometric realization, with all the internal genera
$$g(\gamma)\coloneq b_1(\gamma)+\sum_{\nu\in \text{vert}(\gamma)}g(\nu).$$
A \emph{stable graph} is a labeled graph such that every vertex $\nu$ satisfies the stability condition
$$2g(\nu)+n(\nu)> 2~,$$
where $n(\nu)$ is the arity of $\nu$, that is the number of edges or leaves incident to it.
The set of isomorphism classes of stable graphs of total genus $g$ with $n$ leaves
is denoted by $\Ga_{g}(n)$\ .
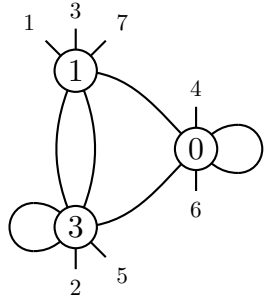
\[\vcenter{\hbox{\begin{tikzpicture}[scale=0.8]

	\draw[thick] (-1, 1.3) to [out=0,in=135] (1,0);
	\draw[thick] (-1, -1.3) to [out=0,in=225] (1,0);
	\draw[thick] (-1, 1.3) to [out=245,in=115] (-1,-1.3);
	\draw[thick] (-1, 1.3) to [out=295,in=65] (-1,-1.3);

	\draw[thick]  (1.7,-0.4) arc [radius=0.4, start angle=270, end angle= 450];
	\draw[thick] (1, 0) to [out=315,in=180] (1.7,-0.4);
	\draw[thick] (1, 0) to [out=45,in=180] (1.7,0.4);

	\draw[thick]  (-1.7,-1.7) arc [radius=0.4, start angle=270, end angle= 90];
	\draw[thick] (-1, -1.3) to [out=225,in=0] (-1.7,-1.7);
	\draw[thick] (-1, -1.3) to [out=135,in=0] (-1.7,-0.9);

	\draw[thick] (-1, 1.3) -- (-1, 2) node[above]  {\scalebox{0.8}{$3$}};
	\draw[thick] (-1, 1.3) -- (-1.5, 1.8) node[above left]  {\scalebox{0.8}{$1$}};
	\draw[thick] (-1, 1.3) -- (-0.5, 1.8) node[above right]  {\scalebox{0.8}{$7$}};
	\draw[thick] (-1, -1.3) -- (-1, -2) node[below]  {\scalebox{0.8}{$2$}};
	\draw[thick] (-1, -1.3) -- (-0.5, -1.8) node[below right]  {\scalebox{0.8}{$5$}};
	\draw[thick] (1, 0) -- (1, 0.7) node[above]  {\scalebox{0.8}{$4$}};
	\draw[thick] (1, 0) -- (1, -0.7) node[below]  {\scalebox{0.8}{$6$}};

	\draw[fill=white, thick] (1,0) circle [radius=10pt];
	\draw[fill=white, thick] (-1,1.3) circle [radius=10pt];
	\draw[fill=white, thick] (-1,-1.3) circle [radius=10pt];

	\node at (1,0) {\scalebox{1.1}{$0$}};
	\node at (-1,1.3) {\scalebox{1.1}{$1$}};
	\node at (-1,-1.3) {\scalebox{1.1}{$3$}};
	\end{tikzpicture}}}
\]
\end{definition}

For $d\in \ZZ$, we consider the endofunctor
\[
\mathbb{G}_d : \Smod \to \Smod
\]
defined by
\[
  \left(\mathbb{G}_d(\P) \right)_g(n)\coloneqq \coprod_{\gamma\in \Ga_g(n)} \gamma(\P)\ ,
\]
where
\begin{equation}\label{eq:monadG}
\gamma(\P)\coloneqq\bigotimes_{v\in \text{vert}(\gamma)} \P_{g(\nu)}(n(v))
\end{equation}
and where each edge of $\gamma$ has degree $d$.
The operation of forgetting the nesting of graphs in
$\mathbb{G}_d\left(\mathbb{G}_d(\P)\right)$, produces elements of $\mathbb{G}_d(\P)$ and thus
induces a monad structure on $\mathbb{G}_d$.  We call this monad the \emph{monad of graphs}
with edges of degree $d$. (We will only need cases of $d=0$ and $d=1$.)

\begin{definition}[Modular operad]
A \emph{modular operad} is an algebra over the monad $\mathbb{G}_0$ of graphs.
\end{definition}

The monad of graphs admits a homogeneous
quadratic presentation and thus the notion of a modular operad can equivalently be defined as
a stable $\mathbb{S}$-module $\mathcal P$ endowed with the following structure maps.

\begin{definition}[Partial compositions and contraction maps]\label{def:PCdef}\leavevmode
\emph{Partial composition maps} on a stable $\Sy$-module $\P$ are equivariant maps of the form
\begin{equation*}
\circ_i^j \ : \ \mathcal{P}_g(n)\otimes \mathcal{P}_{g'}(n') \to  \mathcal{P}_{g+g'}(n+n'-2)\ , \quad \text{for} \ 1\leqslant i \leqslant n \ \text{and} \ 1\leqslant j \leqslant n'\ .
\end{equation*}
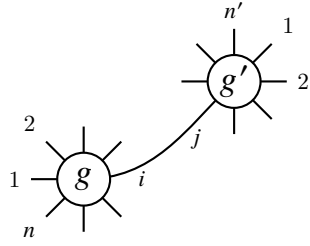
\[\vcenter{\hbox{\begin{tikzpicture}[scale=1]
	\draw[thick] (-1, -1.3) to [out=0,in=225] (1,0);

	\draw[thick] (-1, -1.3) -- (-1, -2) ;
	\draw[thick] (-1, -1.3) -- (-0.5, -1.8) ;
	\draw[thick] (-1, -1.3) -- (-1, -0.6) ;
	\draw[thick] (-1, -1.3) -- (-0.5, -0.8) ;
	\draw[thick] (-1, -1.3) -- (-1.7, -1.3) node[left]  {\scalebox{0.8}{$1$}};
	\draw[thick] (-1, -1.3) -- (-1.5, -0.8) node[above left]  {\scalebox{0.8}{$2$}};
	\draw[thick] (-1, -1.3) -- (-1.5, -1.8) node[below left]  {\scalebox{0.8}{$n$}};

	\draw[thick] (1, 0) -- (1, 0.7) node[above]  {\scalebox{0.8}{$n'$}};
	\draw[thick] (1, 0) -- (1, -0.7) ;
	\draw[thick] (1, 0) -- (0.3, 0) ;
	\draw[thick] (1, 0) -- (1.7, 0) node[right]  {\scalebox{0.8}{$2$}};
	\draw[thick] (1, 0) -- (1.5, 0.5) node[above right]  {\scalebox{0.8}{$1$}};
	\draw[thick] (1, 0) -- (0.5, 0.5) ;
	\draw[thick] (1, 0) -- (1.5, -0.5) ;

	\draw[fill=white, thick] (1,0) circle [radius=10pt];
	\draw[fill=white, thick] (-1,-1.3) circle [radius=10pt];

	\node at (1,0) {\scalebox{1.2}{$g'$}};
	\node at (-1,-1.3) {\scalebox{1.2}{$g$}};

	\node [right] at (-0.4,-1.3) {\scalebox{0.8}{$i$}};
	\node [below] at (0.5,-0.5) {\scalebox{0.8}{$j$}};
		\end{tikzpicture}}}
\]
\emph{Contraction maps} are equivariant maps of the form
\begin{equation*}
\xi_{ij} \ : \  \mathcal{P}_g(n) \to  \mathcal{P}_{g+1}(n-2), \quad \text{for} \ 1\leqslant i\neq j \leqslant n    \ .
\end{equation*}
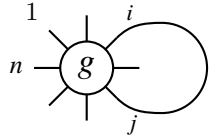
\[\vcenter{\hbox{\begin{tikzpicture}[scale=1]
	\draw[thick] (0, 0) -- (0, 0.7) ;
	\draw[thick] (0, 0) -- (0, -0.7) ;
	\draw[thick] (0, 0) -- (-0.7, 0) node[left]  {\scalebox{0.9}{$n$}};
	\draw[thick] (0, 0) -- (0.7, 0) ;
	\draw[thick] (0, 0) -- (0.4, 0.4);
	\draw[thick] (0, 0) -- (-0.5, 0.5) node[above left]  {\scalebox{0.9}{$1$}};
	\draw[thick] (0, 0) -- (0.4, -0.4);
	\draw[thick] (0, 0) -- (-0.5, -0.5);

	\draw[fill=white, thick] (0,0) circle [radius=10pt];

	\node at (0,0) {\scalebox{1.2}{$g$}};

	\draw[thick]  (1,-0.6) arc [radius=0.6, start angle=270, end angle= 450];

	\draw[thick] (0.4, 0.4) to [out=45,in=180] (1,0.6);
	\draw[thick] (0.4, -0.4) to [out=315,in=180] (1,-0.6);

	\node [above right] at (0.4,0.5) {\scalebox{0.8}{$i$}};
	\node [below right] at (0.4,-0.5) {\scalebox{0.8}{$j$}};
	\end{tikzpicture}}}
\]
\end{definition}

The respective equivariance properties are explicitly given as follows. For any $i\in \underline{n}$ and $j\in \underline{n'}$, we consider the bijection
$
\Phi_{i,j}\, :\, \underline{n} \backslash \{i\} \, \sqcup\, \underline{n'} \backslash \{j\} \to \underline{n+n'-2}
$
corresponding to the total order given by
$
\{1, \ldots, i-1\}, \{j+1, \ldots, n'\}, \{1, \ldots, j-1\}, \{i+1, \ldots, n\}
$\ .
For any permutation $\sigma\in \Sy_n$, we denote by $\sigma_i$ the  induced bijection
$\sigma_i \, : \, \underline{n} \backslash \{i\} \to \underline{n} \backslash \{\sigma(i)\}$\ .
The equivariance property for the partial composition maps amounts to
\[\mu^\sigma \circ_{\sigma(i)}^{\tau(j)} \nu^\tau=\left(\mu \circ_i^j \nu\right)^{\omega}\ ,\]
where the permutation $\omega\in \Sy_{n+n'-2}$ is equal to
$
\omega=\Phi_{\sigma(i), \tau(j)}\circ (\sigma_i\sqcup \tau_j) \circ \Phi^{-1}_{i,j}
$\ .

For any $i,j\in \underline{n}$, we denote by
$
	\Psi_{i,j}\, :\, \underline{n} \backslash \{i,j\} \to \underline{n-2}
$
the order preserving  bijection.
For any permutation $\sigma\in \Sy_n$, we denote by $\sigma_{ij}$ the  induced bijection
$
\sigma_{ij} \, : \, \underline{n} \backslash \{i,j\} \to \underline{n} \backslash \{\sigma(i), \sigma(j)\}
$\ .
The equivariance property for the contraction maps amounts to
\[\xi_{\sigma(i)\sigma(j)} \big(\mu^\sigma\big)=\left(\xi_{ij}(\mu)\right)^{\chi}\ ,\]
where the permutation $\chi\in \Sy_{n-2}$ is equal to
$
\chi=\Psi_{\sigma(i), \sigma(j)}\circ \sigma_{ij} \circ \Psi^{-1}_{i,j}
$\ .

\begin{proposition}[{\cite[Theorem~3.7]{GetzlerKapranov98}}]\label{prop:DefEquivModOp}
A modular operad structure on a stable $\Sy$-module $\P$ is equivalent to the data of partial
composition  maps  $\circ_i^j$ and contraction maps $\xi_{ij}$ satisfying the following three
sets
of relations, for every $\mu \in \mathcal{P}_g(X)$, $\nu\in \mathcal{P}_{g'}(Y)$, and $\omega\in \mathcal{P}_{g''}(Z)$:
\begin{equation}\label{Rel1}
\big(\mu \circ_i^j \nu\big)\circ_k^l\omega = \left\{
\begin{array}{ll}
 \mu \circ_i^j \big(\nu\circ_k^l\omega\big)\ , &\text{when}\ k\in Y\ ,  \\
\rule{0pt}{12pt} \big(\mu \circ_k^l\omega\big)\circ_{i}^j\nu\ , &\text{when}\  k\in X\ , \\
\end{array}\right. \
\text{for any}\  i\in X, j\in Y,  \text{and}\  l\in Z\ ,
\end{equation}
\begin{equation}\label{Rel2}
\xi_{ij} \xi_{kl} \, \mu =\xi_{kl} \xi_{ij} \, \mu   \ ,
\quad \text{for any distinct}\ i,j,k,l \in X\ ,
\end{equation}
\begin{equation}\label{Rel3}
\xi_{ij}\big(\mu \circ_k^l \nu\big) = \left\{
\begin{array}{ll}
\xi_{ij}(\mu) \circ_k^l \nu\ , &\text{when}\ i,j\in X\ ,  \\
\rule{0pt}{12pt}\mu \circ_k^l \xi_{ij}(\nu)\ , &\text{when}\ i,j\in Y\ ,  \\
\rule{0pt}{12pt}\xi_{kl}\big(\mu \circ_i^j \nu\big)\ , &\text{when}\ i\in X \ \text{and}\ j\in Y\ ,  \\
\rule{0pt}{12pt}\xi_{kl}\big(\nu \circ_i^j \mu\big)\ , &\text{when}\ i\in Y \ \text{and}\ j\in X\ ,
\end{array}\right.\
\text{for any}\ k\in X\ \text{and}\  l\in Y\ .
\end{equation}

\begin{remark}
  In the category of (differential) graded vector spaces, the
second of the relations~\eqref{Rel1}
  receives the non-trivial sign
$\big(\mu \circ_i^j \nu\big)\circ_k^l\omega =(-1)^{|\nu||\omega|}\big(\mu \circ_k^l\omega\big)\circ_{i}^j\nu$ and
the last relation in~\eqref{Rel3} receives the sign
$\xi_{ij}\big(\mu \circ_k^l \nu\big) =(-1)^{|\mu||\nu|}\xi_{kl}\big(\nu \circ_i^j \mu\big)$\ .
\end{remark}

\end{proposition}

\begin{example}
\leavevmode
\begin{enumerate}
\item
The collection of the Deligne--Mumford--Knudsen moduli spaces
$\oM\coloneqq\{\overline{\mathcal{M}}_{g,n}\}_{(g,n)\in \NNs}$
of stable curves with marked points forms a topological modular operad where the structure maps
are given by gluing curves at marked points, see~\cite[Section~6]{GetzlerKapranov98}.

\item The  homology groups $H_\bullet\big(\oM\big)$ of the Deligne--Mumford--Knudsen modular operad form a modular operad in the category of graded vector spaces.

\item In a similar fashion, we can construct modular operads $\overline{\M}^f$ and $H_\bullet\big(\overline{\M}^f\big)$, where $\overline{\M}^f_{(g,n)}$ is the Kimura--Stasheff--Voronov real compactification~\cite{KSV95} of the moduli space of genus $g$ curves with $n$ framed marked points (a point with a choice of a unit tangent vector).

\item
Consider the stable $\Sy$-module $\Frob$, where   $\Frob_g(n)\coloneq \K \chi_{g,n}$ is a
one-dimensional space with trivial $\Sy_n$ action concentrated in degree $0$.
It forms a modular operad once equipped by the following  partial composition maps and contraction maps
\begin{align*}
& \chi_{g,n}\circ_i^j \chi_{g',n'}\coloneq \chi_{g+g',n+n'-2}\ , \\
& \xi_{ij}\left(\chi_{g,n}\right)\coloneq \chi_{g+1,n-2}\ .
\end{align*}

\item 
Let $(A, d_A, \langle\ ,\,  \rangle)$ be a \emph{differential graded (dg) symmetric vector space},
i.e.\  a dg vector space $(A,d_A)$ with a symmetric bilinear form of degree  $0$ (of even degree in the
$\Z2$-graded case) compatible with the differential.
The \emph{endomorphism modular operad}  $\EEnd_A$  of
 $(A, d_A, \allowbreak \langle\ ,\,  \rangle)$
is the  stable $\Sy$-module \[\EEnd_A(g,n)\coloneqq A^{\otimes n}\]
with the partial composition and contraction maps given by
\[\circ_i^j(a_1\otimes \cdots \otimes a_m, b_1\otimes \cdots \otimes b_n)\coloneqq \pm \,\langle a_i, b_j\rangle \,
a_1 \otimes \cdots \otimes a_{i-1}\otimes b_{j+1} \otimes \cdots \otimes b_n \otimes b_1 \otimes \cdots \otimes b_{j-1}\otimes
a_{i+1}\otimes\cdots \otimes a_m \]
and
\[\xi_{ij}(a_1\otimes \cdots \otimes a_n)\coloneqq\pm \,\langle a_i, a_j\rangle \,
a_1 \otimes \cdots \otimes a_{i-1}\otimes a_{i+1} \otimes \cdots \otimes a_{j-1}\otimes a_{j+1} \otimes \cdots\otimes
a_n\ ,
\]
where the signs come automatically from the permutation of terms: one first permutes the terms, then applies the pairing. The previous example is a particular case of this one (when $A$ is a one-dimensional space concentrated in homological degree zero)
\end{enumerate}
\end{example}

\begin{remark}
Modular operads have no unit elements
with respect to the partial composition maps since
they would have to live in $\P_0(2)$, an unstable component not considered here.
\end{remark}

    A \emph{morphism of modular operads} is a map of stable $\Sy$-modules which commutes
with the respective partial composition and contraction maps.

\begin{definition}[Algebra over a modular operad]
An \emph{algebra structure over a modular operad $\P$} on a
dg symmetric vector space  $(A, d_A, \langle\ ,\, \rangle)$
is given by the datum of a morphism of modular operads $\P \to \EEnd_A$.
\end{definition}

\begin{example}
Algebras over the modular operad $\Frob$ correspond to commutative Frobenius algebras,
or equivalently, two-dimensional topological quantum field theories.
\end{example} 

\begin{definition}[Cohomological field theory]
A \emph{(non-unital) cohomological field theory (CohFT)}
  is a ($\ZZ$-graded)
dg symmetric vector space  $(A, d_A, \langle\ ,\, \rangle)$
equipped with a $\Z2$-graded
algebra structure over
the homology $H_\bullet\big(\oM\big)$ of the Deligne--Mumford--Knudsen modular operad.
\end{definition}

In other words, we consider the induced
$\Z2$-grading
on both modular operads $H_\bullet\big(\oM\big)$ and $\EEnd_A$
and an even morphism between them.
\subsection{Bar and cobar constructions}
One conceptual way to produce a homotopically meaningful notion of bar-cobar adjunction from a category of algebraic structures is as follows: encode it with a colored operad $\mathcal{O}$, coin a quadratic presentation of $\mathcal{O}$, compute its Koszul dual colored cooperad $\mathcal{O}^{\ac}$ and then apply the operadic calculus of the Koszul duality theory, as explained for instance in~\cite[Chapter~11]{LodayVallette12}.
\[
\vcenter{\hbox{
\begin{tikzcd}
\mathcal{O}
\arrow[d, dotted, "\textsf{alg}"]
 &\arrow[l, "\kappa"'] \mathcal{O}^{\ac}
 \arrow[d, dotted, "\textsf{coalg}"]
 \\
\Omega : \mathcal{O}\textsf{-}\mathsf{alg}
\arrow[r, harpoon, shift left=0.9ex, "\bot"']
&
\arrow[l, harpoon,  shift left=0.9ex]
\mathcal{O}^{\ac}\textsf{-}\mathsf{coalg} : \mathrm{B}
\end{tikzcd}
}}
\]
The example of the category of operads is treated in detail in~\cite{DehlingVallette15}.

\medskip

In order to treat the category of modular operads in this way, one needs to consider operads colored by \emph{groupoids} \cite{BaezDolan98, Petersen13} and not only by sets \cite{BoardmanVogt73}: this step is necessary in order to encode the equivariance property of the partial composition maps and the contraction maps, made explicit in \cref{def:PCdef} below.
The Koszul duality for groupoid-colored operads was developed by Ward in~\cite{Ward19}.

\begin{definition}[The groupoid-colored operad $\G$]\label{def:QLCPresentation}
We define the $\Syg$-colored operad $\G=\TTT(E)/(R)$ by the following presentation. The generators are given by
\begin{align*}
  &
  E\big((g+1,n-2);(g,n)\big) := \left\{\Sy_{n-2}\times
    \begin{aligned}\begin{tikzpicture}[optree]
    \node{}
      child { node[circ]{$\xi_{ij}$}
              child { edge from parent node[right,near end]{\tiny$(g,n)$} }
        edge from parent node[right,near start]{\tiny$(g+1,n-2)$} } ;
   \end{tikzpicture}\end{aligned}\ , \
   i\neq j \in \{1, \ldots, n\}
  \right\}\ ,
\end{align*}
with regular left action of $\Sy_{n-2}$ and right action given, in the notation of \cref{def:PCdef} below, by
\[\left(\id, \xi_{\sigma(i)\sigma(j)}\right)^\sigma=\left(\chi, \xi_{ij}\right)\ ,\]
 and
\begin{multline*}
 E\big((g+g', n+n'-2); (g,n), (g',n')\big) \coloneq \left\{
\Sy_{n+n'-2}\times
   \begin{aligned}\begin{tikzpicture}[optree]
    \node{}
      child { node[circ]{$\circ^j_i$}
        child { node[label=above:$1$]{} edge from parent node[left,near end]{\tiny$(g,n)$} }
        child { node[label=above:$2$]{} edge from parent node[right,near end]{\tiny$(g',n')$} }
        edge from parent node[right,near start]{\tiny$(g+g', n+n'-2)$} } ;
   \end{tikzpicture}\end{aligned} , \
   1 \leqslant i \leqslant n,    1 \leqslant j \leqslant n';
   \right.\\ \left.
   \Sy_{n+n'-2}\times
   \begin{aligned}\begin{tikzpicture}[optree]
    \node{}
      child { node[circ]{$\circ^j_i$}
        child { node[label=above:$2$]{} edge from parent node[left,near end]{\tiny$(g,n)$} }
        child { node[label=above:$1$]{} edge from parent node[right,near end]{\tiny$(g',n')$} }
        edge from parent node[right,near start]{\tiny$(g+g', n+n'-2)$} } ;
   \end{tikzpicture}\end{aligned} , \
   1 \leqslant i \leqslant n',    1 \leqslant j \leqslant n\right\}\ ,
\end{multline*}
with regular $\Sy_2$-action,  that is the action of the transposition $(12)$ sends
\[
\vcenter{\hbox{\begin{tikzpicture}[optree]
    \node{}
      child { node[circ]{$\circ^j_i$}
        child { node[label=above:$1$]{} edge from parent node[left,near end]{\tiny$(g,n)$} }
        child { node[label=above:$2$]{} edge from parent node[right,near end]{\tiny$(g',n')$} }
        edge from parent node[right,near start]{\tiny$(g+g', n+n'-2)$} } ;
   \end{tikzpicture}}}
 \ \text{to} \
\vcenter{\hbox{\begin{tikzpicture}[optree]
    \node{}
      child { node[circ]{$\circ^j_i$}
        child { node[label=above:$2$]{} edge from parent node[left,near end]{\tiny$(g',n')$} }
        child { node[label=above:$1$]{} edge from parent node[right,near end]{\tiny$(g,n)$} }
        edge from parent node[right,near start]{\tiny$(g+g', n+n'-2)$} } ;
   \end{tikzpicture}}}\ ,\]
and with regular left $\Sy_{n+n'-2}$ action and right  action  given, in the notation of  \cref{def:PCdef} below, by
\[\left(\id, \circ_{\sigma(i)}^{\sigma(j)}\right)^{(\sigma,\tau)}=\left(\omega, \circ_i^j\right)\ .\]
   The quadratic relations $R$ are the ones given in  \eqref{Rel1}, \eqref{Rel2} and \eqref{Rel3}.
\end{definition}

\cref{prop:DefEquivModOp} implies that the category of $\G$-algebras is the category of modular
operads. This shows for instance that the free modular operad $ \mathbb{G}_0(\M)$ on a stable
$\Sy$-module $\M$ is equivalently given by $\G(\M)$\ .

\begin{remark}
The notion of an \emph{operad colored with a small category} was introduced by Baez--Dolan in \cite{BaezDolan98},
where is was shown that the category of (symmetric) operads can be encoded this way, see also \cite{Petersen13}.
Equivalent notions of \emph{symmetric substitudes}, \emph{regular patterns}, and \emph{Feynman categories} were respectively and gradually introduced first by Day--Street \cite{DayStreet03}, then by Getzler \cite{Getzler09}, and finally by Kaufmann--Ward \cite{KaufmannWard17} inspired by \cite{BorisovManin08}.
In these latter three references, the authors even mention that the category of modular operads can be encoded using their new notion. We refer the reader to \cite{Caviglia15, BKW18} for precise statements relating these four equivalent  theories.
The description of the groupoid-colored operad $\G$ encoding modular first appeared in \cite{Ward19}.
\end{remark}

\begin{lemma}[{\cite[Corollary~3.9]{Ward19}}, see also {\cite[Theorem~12.10]{BataninMarkl18}}]\label{lem:G!}
The Koszul dual colored operad $\G^!$ admits for generators
\[  s^{-1}E\big((g+1,n-2);(g,n)\big)^* \quad \text{and} \quad E\big((g+g', n+n'-2);(g,n), (g',n')\big)^*\otimes \mathrm{sgn}_{\Sy_2}\]
satisfying strictly the relations \eqref{Rel1} and the
first three
relations of \eqref{Rel3} and satisfying up to an extra minus sign the relations \eqref{Rel2} and the fourth relation \eqref{Rel3}.
\end{lemma}

\begin{proof}
The Koszul dual operad is defined by
$\G^!\coloneqq \left(\G^{\ac}\right)^*\otimes \End_{\K s^{-1}}$, see~\cite[Section~7.2.3]{LodayVallette12}, and it admits for presentation
$\G^!\cong \TTT(s^{-1}\End_{\K s^{-1}}\otimes E^*)/(R^\perp)$
by~\cite[Proposition~7.2.1]{LodayVallette12}. The rest of the proof is a direct computation.
\end{proof}

By definition of the Koszul dual colored operad, a $\left(\G^{\ac}\right)^*$-algebra structure on a stable $\Sy$-module $\P$ is equivalent to a $\G^!$-algebra structure on the desuspension $s^{-1}\P$.

\begin{definition}[Shifted modular operad]
A \emph{shifted modular operad} is a $\left(\G^{\ac}\right)^*$-algebra, that is a stable $\Sy$-module $\P$ endowed with degree $-1$ partial composition maps $\circ_i^j$ and degree $-1$ contraction maps $\xi_{ij}$ satisfying strictly the relations \eqref{Rel1} and the third relation of \eqref{Rel3}
and satisfying up to an extra minus sign the relations \eqref{Rel2} and the three other relations of \eqref{Rel3}.
\end{definition}

\begin{definition}[Shifted modular cooperad]
A \emph{shifted modular cooperad} is a coalgebra over the Koszul dual cooperad $\G^{\ac}$, that is a stable $\Sy$-module $\C$ endowed with degree $-1$ \emph{partial decomposition maps}
\begin{equation*}
\delta_i^j \ : \ \C_{g}(X)\to
\bigoplus_{\substack{g'+g''=g\\  X'\sqcup X''=X}}
\C_{g'}(X'\cup\{i\})\otimes \C_{g''}(X''\cup \{j\})\ , \quad \text{for} \  X'\cap X''=\emptyset
\end{equation*}
and degree $-1$ \emph{expansion maps}
\begin{equation*}
\chi_{ij} \ : \  \C_{g+1}\big(X\backslash\{i,j\}\big) \to \C_g(X), \quad \text{for} \ i\neq j \in X
\end{equation*}
satisfying  
the following signed versions of the relations~\eqref{Rel1}--\eqref{Rel3}, for any distinct elements $i,j,k,l$:
\begin{equation}\label{Rel1'}
\big( \delta_i^j \otimes \id \big) \delta_k^l
= \left\{
\begin{array}{ll}
\rule{0pt}{8pt} -\big(\id \otimes \delta_k^l \big) \delta_i^j\ , &\text{when}\  j,k \ \text{lie in the same set}\ , \\
\rule{0pt}{12pt}
-\big(23)\big(\delta_k^l \otimes \id \big) \delta_i^j \ , &\text{when}\  i,k \ \text{lie in the same set}\ , \\
\end{array}\right.
\end{equation}
\begin{equation}\label{Rel2'}
\chi_{kl} \chi_{ij} =- \chi_{ij} \chi_{kl} \ ,
\end{equation}
\begin{equation}\label{Rel3'}
\delta_k^l \chi_{ij}
= \left\{
\begin{array}{cl}
\rule{0pt}{8pt}-(\chi_{ij}\otimes \id)\delta_k^l\ , &\text{when}\  i,j,k \ \text{belong to the same set}\ , \\
\rule{0pt}{12pt} -(\id \otimes \chi_{ij})\delta_k^l\ , &\text{when}\  i,j,l \ \text{belong to the same set}\ , \\
\rule{0pt}{12pt}\delta_i^j \chi_{kl}\ , &\text{when}\  i,k \ \text{belong to the same set and}\ j,l\ \text{belong to the same set}\ , \\
\rule{0pt}{12pt} -(12)\delta_i^j \chi_{kl}\ , &\text{when}\  i,l \ \text{belong to the same set and}\ j,k\ \text{belong to the same set}\ . \\
\end{array}\right.
\end{equation}
\end{definition}

The category of shifted modular cooperads is the category of coalgebras over a comonad of graphs
$\mathbb{G}^c_1 : \Smod \to \Smod$, dual to the monad
$\mathbb{G}_1$ of graphs with edges of degree $1$.
This implies for instance that the cofree
shifted modular cooperad on a stable $\Sy$-module $\M$ is given equivalently by $\G^{\ac}(\M)\cong \mathbb{G}^c_1(\M)$.

\medskip

The canonical Koszul morphism $\kappa : \G^{\ac} \to \G$ gives rise to the following adjunction

$$\xymatrix@C=30pt{{\B \ : \ \mathsf{modular} \ \mathsf{operads}  \ }  \ar@_{->}@<-1ex>[r]  \ar@^{}[r]|(0.45){\perp}   & \ar@_{->}@<-1ex>[l]  {\ \mathsf{shifted} \ \mathsf{modular} \ \mathsf{cooperads}  \ : \ \Omega}}\ .   $$

\begin{definition}[Bar construction]\label{def:BarCons}
The \emph{bar construction} of a modular operad $\big(\P, d_\P, \circ^j_i, \xi_{ij}\big)$ is
the quasi-cofree shifted modular cooperad
\[\B \P\coloneqq \big(\G^{\ac}(\P), d_1+d_2 \big)\ ,\]
where $d_1$ is the unique coderivation which extends the internal differential $d_\P$ and
where $d_2$ is the unique coderivation which extends the partial composition maps $\circ_i^j$
and the contraction maps $\xi_{ij}$.
\end{definition}

Explicitly, the bar construction is spanned by stable graphs  with edges of degree $1$ and with vertices
labeled
by elements of $\P$. The differential $d_1$ amounts to applying the internal differential
$d_\P$ to every vertex one by one and the differential $d_2$ amounts to contracting the
internal edges one by one and applying the partial composition maps or the contraction maps.

\begin{definition}[Cobar construction]
The \emph{cobar construction} of a shifted modular cooperad $\big(\C, d_\C,\allowbreak
\delta^j_i, \allowbreak  \chi_{ij}\big)$ is the quasi-free  modular operad
\[\Omega \C\coloneqq \big(
\G(\C), d_1+d_2
\big)\ ,\]
where $d_1$ is the unique derivation which extends the internal differential $d_\C$ and
where $d_2$ is the unique derivation which extends the partial decomposition maps $\delta_i^j$
and the expansion maps $\chi_{ij}$.
\end{definition}

Explicitly, the cobar construction is spanned by stable graphs with edges of degree $0$ and
with vertices
labeled
by elements of $\C$. The differential $d_1$ amounts to
applying the internal differential $d_\C$ to every vertex, and the differential $d_2$ amounts
to decomposing every vertex into two, applying one by one partial decomposition maps, and
creating one new loop by applying to every vertex one by one the expansion maps.

\begin{remark}
In the original paper~\cite{GetzlerKapranov98} on modular operads, these homological constructions were not introduced in this way. Since the authors stuck with the modular operads world, that is did not consider any version of modular cooperads; they just introduced a notion of \emph{twisted modular operad}. The notion of twisted modular operad associated to the dualizing cocycle given by the determinant on edges coincides with the present notion of shifted modular operad. Then, they introduced a fundamental functor from \emph{finite dimensional} modular operads to shifted modular operads called the \emph{Feynman transform} $\mathcal{F} \P$. This latter one is nothing but the linear dual of the present bar construction: $\mathcal{F} \P\cong (\B \P)^*$. What is called in~\cite[Section~5.3]{GetzlerKapranov98} the \emph{homotopy inverse Feynman transform} is nothing but the cobar construction $\Omega (\P^*)$ of the linear dual of a shifted modular operad. Notice that the present approach allows us to drop the  assumption of finite dimensionality and to work in full generality.
\end{remark}

\begin{theorem}[{\cite[Theorem~5.4]{GetzlerKapranov98}}, see also {\cite[Theorem~3.10]{Ward19}}]\label{thm:Koszul}
The colored operad $\G$ is Koszul.
\end{theorem}

\begin{proof}
Under the above identification between the (homotopy inverse) Feynman transform and the (co)bar construction, this statement is equivalent to~\cite[Theorem~5.4]{GetzlerKapranov98}, generalized in~\cite[Theorem~7.4.3]{KaufmannWard17}, using the characterization given in~\cite[Theorem~11.3.3]{LodayVallette12}.  Another more direct proof is given at~\cite[Theorem~3.10]{Ward19}.
\end{proof}

\begin{corollary}
The counit of the bar-cobar adjunction provides us with a functorial cofibrant replacement for modular operads:
\[\Omega \B \P \xrightarrow{\sim} \P\ .\]
\end{corollary}

\begin{proof}
The statement that the counit of the bar-cobar adjunction is always a quasi-isomorphism is \cref{thm:Koszul}.
The cofibrant property can be read in the model category structure on modular operads
introduced by Berger--Moerdjik~\cite{BergerMoerdijk07}, see also~\cite{BataninBerger17} and~\cite[Chapter~8]{KaufmannWard17}.
\end{proof}

\begin{definition}[Homotopy cohomological field theory]
A \emph{homotopy cohomological field theory} ($\text{\it CohFT}_\infty$) is
a dg symmetric vector space  $(A, d_A, \langle\ ,\, \rangle)$ ($\ZZ$-graded)
equipped with a $\Z2$-graded
algebra structure over the bar-cobar resolution
$\Omega \B  H_\bullet\big(\oM\big)$ of the homology of the Deligne--Mumford--Knudsen modular operad.
\end{definition}

\begin{remark}
A cyclic operad is a modular operad concentrated in genus $0$ and a Koszul duality theory has been
proposed on this level in~\cite{GetzlerKapranov95}. Its main idea amounts to using a quadratic
presentation of a cyclic operad in order to
obtain a
functorial bar-cobar resolution. For instance, the cyclic operad
$\mathrm{HyperCom}\coloneqq H_\bullet\big(\overline{\mathcal{M}}_0\big) \coloneqq
\{H_\bullet\big(\overline{\mathcal{M}}_{0,n}\big)\}$
is Koszul.
It admits a Koszul dual shifted cyclic cooperad  i.e.\ a shifted modular cooperad concentrated in
genus $0$, which is up to operadic suspension isomorphic to
$\mathrm{HyperCom}^{\ac}\cong
H^\bullet\big({\mathcal{M}}_0\big)$.
Thus, the minimal way  to define the notion of a \emph{homotopy tree level cohomological field theory}
in the terminology of~\cite{Manin99}, that is a ``homotopy cohomological field theory concentrated
in genus $0$'', is via the cyclic cobar construction $\Omega \mathrm{HyperCom}^{\ac}$.
The canonical embedding
$\mathrm{HyperCom}^{\ac}\hookrightarrow \B H_\bullet\big(\overline{\mathcal{M}}_{0}\big)$
induces a morphism of modular operads
$\Omega \mathrm{HyperCom}^{\ac} \to \Omega \B H_\bullet\big(\overline{\mathcal{M}}_0\big)$
and thus produces a canonical and functorial homotopy tree level cohomological field theory on any
homotopy cohomological field theory.

\medskip

In higher genera, there exists so far no Koszul duality theory for modular operads, so we do not
know how to reduce the bar-cobar resolution
$\Omega \B  H_\bullet\big(\oM\big)$
used here to define the notion of a homotopy cohomological field theory.
The arguments of~\cite[Section~1]{AlmPetersen17} and especially Proposition 1.11
 show that the shifted modular cooperad structure on $H^{\bullet}\left(\mathcal{M}\right)$ given by the residue maps cannot be the ``Koszul dual'' of  $H_\bullet\big(\oM\big)$: the cobar construction of $H^{\bullet}\left(\mathcal{M}\right)$ is not a resolution of $H_\bullet\big(\oM\big)$.
Since the Koszul dual operad is not readily available (and is not known to be Koszul), one can use instead the
general operadic method, see for instance~\cite{DrummondColeVallette13}. It amounts to first computing the homology groups of $\B  H_\bullet\big(\oM\big)$: they are actually given by $H^{\bullet}\left(\mathcal{M}\right)$, see~\cite[Proposition~6.11]{GetzlerKapranov95}. In the second step, one endows this homology with
 a homotopy shifted modular cooperad structure via the homotopy transfer theorem; this would produce the minimal model for the modular operad $H_\bullet\big(\oM\big)$.
The only drawback with this latter point does not lie in the formula for the homotopy transfer theorem, which is produced (dually) in~\cite[Section~3.6]{Ward19}, but in  the choice of an explicit contraction to use.
\end{remark}

\section{Deformation theory}\label{sec:Def}
Following the general method used for instance in~\cite[Section~12.2]{LodayVallette12} and in~\cite{MerkulovVallette09I, MerkulovVallette09II}, we introduce the suitable  deformation complex for morphisms of modular operads. We use it to control the deformation theory of cohomological field theories. This allows us to recover  conceptually and to extend recent constructions of CohFTs due to Pandharipande--Zvonkine~\cite{PZ18}.
By considering the quantum master equation instead of the master equation, we introduce  a new notion of \emph{quantum homotopy CohFT}. The program of quantization of Dubrovin--Zhang integrable hierarchies proposed
by Buryak--Rossi~\cite{BuryakRossi2016} actually provides us with examples of such a structure.

\subsection{Deformation complex of morphisms of modular operads}\label{subsec:DefCom}

The following algebraic notion is prompted by our main example. Notice that we use the homological degree convention that might seem at bit unusual at first sight.

\begin{definition}[Shifted $\Delta$-Lie algebra]
A \emph{shifted dg  $\Delta$-Lie algebra}
is a quadruple $(\g, \d,\Delta, \{\ , \,\})$
consisting of a dg vector space $(\g,\d)$, a degree $-1$ linear operator
$\Delta : \g \to \g$, and a degree $-1$ symmetric product $\{\ , \,\} : \g^{\odot 2} \to \g$
satisfying the following relations:
\[\Delta^2=0 \ , \qquad  \d \Delta+\Delta \d =0\ ,
\quad \d \ \text{and} \ \Delta\ \ \text{are derivations of }\ \ \{\ ,\,\}\ ,
\]
and the shifted Jacobi identity for  \ $\{\ ,\,\}$
i.e.\
\[ \{\{x,y\},z\}+
(-1)^{(|x|+1)(|y|+|z|)}
\{\{y,z\},x\}+
(-1)^{(|z|+1)(|x|+|y|)}
\{\{z,x\},y\}=0 \
 \]
for homogeneous elements $x, y, z$.
\end{definition}
\begin{remark}
When considering the total derivation $\d +\Delta$, one gets a shifted dg Lie algebra,
that is a dg Lie algebra structure on the desuspension $s^{-1}\g$ of $\g$~.
\end{remark}

\begin{remark}
This notion was already considered in the literature, sometimes with different degree conventions, see \cite{KaufmannWard17, VoronovQuantDefII} for instance.
\end{remark}

\begin{definition}[Master equation]
The \emph{master equation} is the equation
\begin{equation}\label{eq:MaEq}
\d \alpha + \Delta \alpha + {\textstyle \frac12} \{\alpha, \alpha\}=0
 \end{equation}
in a shifted $\Delta$-Lie algebra.  We only consider solutions of degree $|\alpha|=0$; their set is denoted by $\MC(\g)$.
\end{definition}

\begin{remark}
The master equation in a shifted $\Delta$-Lie algebra is equivalent to the Maurer--Cartan equation
$d\alpha + {\textstyle \frac12} \{\alpha, \alpha\}=0$ in the associated shifted
dg Lie algebra, where $d=\d+\Delta$\ .
\end{remark}

Let us now introduce our main example.

\begin{definition}[Totalisation]
The  \emph{totalisation} of a stable $\Sy$-module $\P$ is the graded vector space defined by
\[
\whP\coloneqq \prod_{(g,n) \in \NNs} \P_g(n)^{\Sy_n}
\ .\]
\end{definition}

\begin{remark}
Since we are working here in characteristic $0$, we could have equivalently considered coinvariants instead of invariants in the definition of the totalisation.
\end{remark}

\begin{lemma}\label{lem:ModtoDeltaLie}
 The  assignment
\[
  \big(\P, d_\P, \circ^j_i, \xi_{ij}\big)
  \mapsto
 \big(\whP, \d, \Delta, \{\ , \,\}, \F
\big)\ ,
\]
defines a functor 
\[
\mathsf{shifted} \ \mathsf{modular} \ \mathsf{operads}
\to
\mathsf{complete} \ \mathsf{shifted} \ \Delta\textsf{-}\mathsf{Lie}\ \mathsf{algebras}~, \]
where the differential $\d$ is induced by the differential $d_\P$,
the operator $\Delta$ is induced by the contraction maps
\[\Delta(\mu)\coloneqq \sum_{i,j}
\xi_{ij}(\mu)\ ,\]
 the shifted Lie bracket is induced by partial composition maps
\[\{\mu,\nu\}\coloneq
\sum_{i,j}\circ^j_i (\mu\otimes \nu)\ ,
\]
and the term $\F_N$ of the decreasing filtration
\[\whP=\F_0=\F_1 \supset \F_2 \supset \cdots \F_N \supset \F_{N+1}\supset \cdots \ ,\]
consists of elements of \ $\whP$ in which the components with
$n+2g-2<N$ vanish.
\end{lemma}

\begin{proof}
Since the partial composition maps $\circ_i^j$ and the contraction maps $\xi_{ij}$ have degree $-1$,
the operator $\Delta$ and the bracket $\{\, ,\}$ also have degree $-1$.
Therefore the relation~\eqref{Rel2'} forces the operator $\Delta$ to square to zero and~\eqref{Rel1'} implies the shifted Jacobi relation.
The composition $\Delta\circ\{\ , \})$ can be rewritten using the relation \eqref{Rel3'}.
The first two cases give respectively $-\{\Delta,\, \}$ and $-\{\ , \Delta\}$. Due to the difference of signs in the last two cases, they cancel and we finally get the shifted derivation relation between $\Delta$ and $\{\, , \}$. Any morphism of shifted modular operads induces a morphism of shifted $\Delta$-Lie algebras.

By definition, this filtration defines a complete topology; we refer the reader to~\cite[Chapter~1]{DotsenkoShadrinVallette18} for the respective definitions.
The stability condition of the underlying stable $\Sy$-module of a shifted modular operad ensures that
$\whP=\F_0=\F_1$\ .
The various operations  preserve this filtration: $\d, \Delta : \F_N \to \F_N$ and
$\{\ , \,\} : \F_N\odot \F_M \to \F_{N+M}$\ .
\end{proof}

\begin{definition}[Totalisation of a modular operad]
We call the above complete shifted $\Delta$-Lie
algebra $\big(\whP, \d, \Delta, \{\ , \,\}, \F \big)$
the \emph{totalisation} of the shifted modular operad $\P$ and
(slightly abusing notation) will denote it by $\whP$.
\end{definition}

\begin{example}\label{ex:TPoly}
Let $V$ be
a graded vector space with finite dimensional components. The vector space
$A\coloneq V\oplus s V^*$ is equipped with the degree $-1$
pairing defined by $\langle s f, x \rangle\coloneq f(x)$\
which makes it an odd symplectic manifold.
The associated endomorphism
operad $\EEnd_A$ is a shifted modular operad
and its totalisation is the shifted $\Delta$-Lie algebra
\[
\widehat{\EEnd_A}\cong \prod_{(g,n) \in \NNs} A^{\odot n}\subset
\K[[V\oplus s V^*]][[t]]
\ ,\]
which is a sub-algebra of the quantized version of the Weyl algebra.
Notice that in the cohomological
situation (like for example in~\cite{Kontsevich97, Willwacher15}), one considers the odd
symplectic manifold structure on $V\oplus s^{-1} V^*$
induced by a
degree $+1$ pairing. In this case, one gets a usual (unshifted) modular operad of endomorphisms,
where the composition and contractions maps have degree $+1$. Its totalisation is an
unshifted $\Delta$-Lie algebra which is actually a Batalin--Vilkovisky subalgebra of the quantized algebra of polyvector fields on $V$.
\end{example}

\begin{lemma}
The assignment
\[  \left(\big(\C, d_\C, \delta^j_i, \chi_{ij}\big), \big(\P, d_\P, \circ^j_i, \xi_{ij}\big)\right) \mapsto
\Hom \left(\C, \P\right)\coloneq \left(
\left\{\Hom\left(
\C_g(n), \P_g(n)\right)\right\}_{g,n}, \partial,\bigcirc^j_i, \Xi_{ij}
\right)
\]
defines a functor
\[
(\mathsf{shifted} \ \mathsf{modular} \ \mathsf{cooperads})^{\mathsf{op}} \times
\mathsf{modular} \ \mathsf{operads}
\to
\mathsf{shifted} \ \mathsf{modular} \ \mathsf{operads}~,
\]
where
$\Sy_n$  acts on $\Hom\left(\C_g(n), \P_g(n)\right)$  by conjugation, the  differential $\partial$
is given by \[\partial(f)\coloneq d_\P \circ f - (-1)^{|f|}f \circ d_\C\ ,\]
the partial compositions maps are given by
\[\bigcirc_{i}^j(f\otimes g)\coloneq (-1)^{|f|+|g|}
\circ^j_i\circ (f\otimes g)\circ \delta^j_i\ ,
\]
 and the contraction maps are given by
 \[\Xi_{ij}(f)\coloneq(-1)^{|f|} \, \xi_{ij}\circ f\circ \chi_{ij}\ .\]
\end{lemma}

\begin{proof}
This immediately follows from the defining relations.
\end{proof}

\begin{definition}[Convolution shifted modular operad]
We call the shifted modular operad $\Hom(\C, \P)$ described above the \emph{convolution shifted modular operad}.
\end{definition}

\begin{definition}[Convolution algebra]\label{def:ConDeltaLie}
The \emph{convolution shifted $\Delta$-Lie algebra} associated to
a shifted modular cooperad $\C$ and a modular operad $\P$ is
the totalisation $\widehat{\Hom}(\C, \P)$ of the convolution shifted modular operad $\Hom(\C, \P)$\ , i.e.\
\[\widehat{\Hom}(\C, \P)\coloneq\big(\widehat{\Hom}_{\Syg} \left(\C, \P\right), \d, \Delta, \{\ , \,\}\big)\ ,\]
where \[\widehat{\Hom}_{\Syg} \left(\C, \P\right)\coloneq\prod_{(g,n)\in \NNs} \Hom_{\Sy_n}\left(
\C_g(n), \P_g(n)\right)\ .\]
\end{definition}

Notice that
all operations of $\widehat{\Hom}(\C, \P)$ preserve the arity. The differential and the bracket also preserve the genus
 and the operator $\Delta$ increases the genus by $1$.

\begin{remark}
The arity- and genus-wise linear dual $\Frob^*$ of the Frobenius modular operad produces a modular cooperad which satisfies the following universal property: any shifted modular operad $\P$ is canonically isomorphic to the shifted convolution modular operad $\P\cong \Hom \left(\Frob^*, \P\right)$ and its totalisation  is canonically isomorphic to the convolution algebra
$\whP\cong \widehat{\Hom}_{\Syg} \left(\Frob^*, \P\right)$\ .
So there is no restriction to consider only convolution operads and algebras.
\end{remark}

\begin{definition}[Twisting morphism]
Solutions of degree $0$ to the master equation \eqref{eq:MaEq} in the convolution algebra $\widehat{\Hom}(\C, \P)$
 are called  \emph{twisting morphisms} from $\C$ to $\P$; their set is denoted by $\Tw(\C, \P)$.
\end{definition}

\begin{proposition}[{\cite[Theorem~1]{Barannikov07}}]\label{prop:MorphTw}
The set of morphisms of modular operads $\Omega  \C \to \P$ is in a natural bijection with the set of twisting morphisms in $\widehat{\Hom}(\C, \P)$:
\[
\Hom_{\mathsf{mod}\,  \mathsf{op}}(\Omega \C, \P)\cong \Tw(\C, \P)\ .
\]
\end{proposition}

\begin{proof}
The proof on the level of operads given in~\cite[Proposition~11.3.1]{LodayVallette12} works \emph{mutatis mutandis} for groupoid-colored operads. This result falls into the general pattern of~\cite[Theorem~7.5.3]{KaufmannWard17}.
\end{proof}

\subsection{Deformation theory of cohomological field theories}\label{subsec:DefCohFT}
We are  interested in the case of the deformation theory of (homotopy) cohomological field theories, that we now treat in detail.

\begin{convention}
From now until the end of \cref{subsec:PZ}, we will be working only with $\Z2$ gradings in order to be consistent with the standard conventions about CohFTs.
\end{convention}

\begin{proposition}\label{lem:DefComp}
Let $(A, d_A, \langle\ ,\, \rangle)$
be a dg symmetric vector space.
The $\Z2$-graded
convolution algebra associated to $\P=\EEnd_A$ and $\C=\B H_\bullet\big(\oM\big)$  is isomorphic to
the $\Delta$-Lie algebra
\[\g_A\coloneq
  \left(
\widehat{\mathbb{G}}_{\od}\big(\Hbul\big(\oM\big)\big)(A), \mathrm{d}, \Delta, \{\ ,\,\}
\right)\ ,
\]
where
\[
\widehat{\mathbb{G}}_{\od}\big(\Hbul\big(\oM\big)\big)(A)\coloneqq
\prod_{(g,n)\in\NNs} \ \prod_{\gamma\in \Ga_g(n)} \left( \gamma\big(\Hbul\big(\oM\big)\big) \otimes A^{\otimes n} \right)~,
\]
with $\gamma\big(\Hbul\big(\oM\big)\big)$  is given by~\eqref{eq:monadG} and all edges of $\gamma$
of odd degree. In other words,  $\g_A$ is spanned by isomorphism classes of stable graphs with vertices labeled by cohomology classes of $\oM$ and leaves labeled by elements of $A$, e.g.\
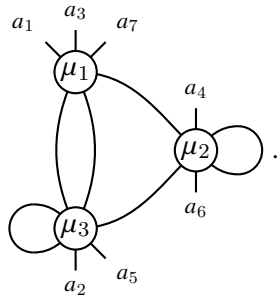
  \[\vcenter{\hbox{\begin{tikzpicture}[scale=0.8]

	\draw[thick] (-1, 1.3) to [out=0,in=135] (1,0);
	\draw[thick] (-1, -1.3) to [out=0,in=225] (1,0);
	\draw[thick] (-1, 1.3) to [out=245,in=115] (-1,-1.3);
	\draw[thick] (-1, 1.3) to [out=295,in=65] (-1,-1.3);

	\draw[thick]  (1.7,-0.4) arc [radius=0.4, start angle=270, end angle= 450];
	\draw[thick] (1, 0) to [out=315,in=180] (1.7,-0.4);
	\draw[thick] (1, 0) to [out=45,in=180] (1.7,0.4);

	\draw[thick]  (-1.7,-1.7) arc [radius=0.4, start angle=270, end angle= 90];
	\draw[thick] (-1, -1.3) to [out=225,in=0] (-1.7,-1.7);
	\draw[thick] (-1, -1.3) to [out=135,in=0] (-1.7,-0.9);

	\draw[thick] (-1, 1.3) -- (-1, 2) node[above]  {\scalebox{0.8}{$a_3$}};
	\draw[thick] (-1, 1.3) -- (-1.5, 1.8) node[above left]  {\scalebox{0.8}{$a_1$}};
	\draw[thick] (-1, 1.3) -- (-0.5, 1.8) node[above right]  {\scalebox{0.8}{$a_7$}};
	\draw[thick] (-1, -1.3) -- (-1, -2) node[below]  {\scalebox{0.8}{$a_2$}};
	\draw[thick] (-1, -1.3) -- (-0.5, -1.8) node[below right]  {\scalebox{0.8}{$a_5$}};
	\draw[thick] (1, 0) -- (1, 0.7) node[above]  {\scalebox{0.8}{$a_4$}};
	\draw[thick] (1, 0) -- (1, -0.7) node[below]  {\scalebox{0.8}{$a_6$}};

	\draw[fill=white, thick] (1,0) circle [radius=10pt];
	\draw[fill=white, thick] (-1,1.3) circle [radius=10pt];
	\draw[fill=white, thick] (-1,-1.3) circle [radius=10pt];

	\node at (1,0) {\scalebox{1}{$\mu_2$}};
	\node at (-1,1.3) {\scalebox{1}{$\mu_1$}};
	\node at (-1,-1.3) {\scalebox{1}{$\mu_3$}};
	\end{tikzpicture}}}~.
\]

The differential $\d=\d_A-\d_1-\d_2$ is the difference between the extension $\d_A$ of the
differential of $A$ and the sum of the differential $\d_1$ producing
loops under the
expansion maps of the modular cooperad
$\Hbul\big(\oM\big)$ applied at each vertex and the differential $\d_2$ splitting vertices into two along a new edge under the partial decomposition maps of $\Hbul\big(\oM\big)$.
The operator $\Delta$ amounts to considering any pair of elements $a_i, a_j \in A$ labeling pairs of leaves of a graph, compute their pairing $\langle a_i, a_j \rangle\in \mathbb{k}$ and create a tadpole at their place.
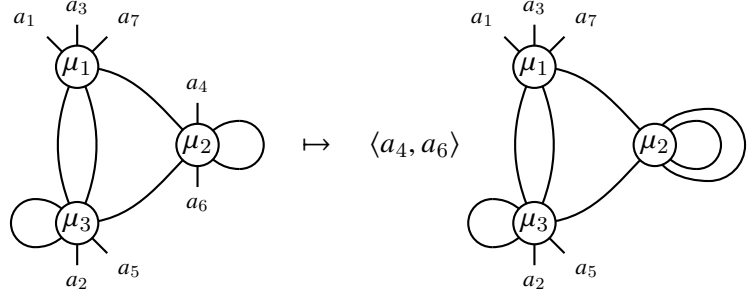
\[
\vcenter{\hbox{\begin{tikzpicture}[scale=0.8]
	\draw[thick] (-1, 1.3) to [out=0,in=135] (1,0);
	\draw[thick] (-1, -1.3) to [out=0,in=225] (1,0);
	\draw[thick] (-1, 1.3) to [out=245,in=115] (-1,-1.3);
	\draw[thick] (-1, 1.3) to [out=295,in=65] (-1,-1.3);

	\draw[thick]  (1.7,-0.4) arc [radius=0.4, start angle=270, end angle= 450];
	\draw[thick] (1, 0) to [out=315,in=180] (1.7,-0.4);
	\draw[thick] (1, 0) to [out=45,in=180] (1.7,0.4);

	\draw[thick]  (-1.7,-1.7) arc [radius=0.4, start angle=270, end angle= 90];
	\draw[thick] (-1, -1.3) to [out=225,in=0] (-1.7,-1.7);
	\draw[thick] (-1, -1.3) to [out=135,in=0] (-1.7,-0.9);

	\draw[thick] (-1, 1.3) -- (-1, 2) node[above]  {\scalebox{0.8}{$a_3$}};
	\draw[thick] (-1, 1.3) -- (-1.5, 1.8) node[above left]  {\scalebox{0.8}{$a_1$}};
	\draw[thick] (-1, 1.3) -- (-0.5, 1.8) node[above right]  {\scalebox{0.8}{$a_7$}};
	\draw[thick] (-1, -1.3) -- (-1, -2) node[below]  {\scalebox{0.8}{$a_2$}};
	\draw[thick] (-1, -1.3) -- (-0.5, -1.8) node[below right]  {\scalebox{0.8}{$a_5$}};
	\draw[thick] (1, 0) -- (1, 0.7) node[above]  {\scalebox{0.8}{$a_4$}};
	\draw[thick] (1, 0) -- (1, -0.7) node[below]  {\scalebox{0.8}{$a_6$}};

	\draw[fill=white, thick] (1,0) circle [radius=10pt];
	\draw[fill=white, thick] (-1,1.3) circle [radius=10pt];
	\draw[fill=white, thick] (-1,-1.3) circle [radius=10pt];

	\node at (1,0) {\scalebox{1}{$\mu_2$}};
	\node at (-1,1.3) {\scalebox{1}{$\mu_1$}};
	\node at (-1,-1.3) {\scalebox{1}{$\mu_3$}};
	\end{tikzpicture}}}
\quad  \mapsto\quad \langle a_4, a_6 \rangle
  \vcenter{\hbox{\begin{tikzpicture}[scale=0.8]
	\draw[thick] (-1, 1.3) to [out=0,in=135] (1,0);
	\draw[thick] (-1, -1.3) to [out=0,in=225] (1,0);
	\draw[thick] (-1, 1.3) to [out=245,in=115] (-1,-1.3);
	\draw[thick] (-1, 1.3) to [out=295,in=65] (-1,-1.3);

	\draw[thick]  (1.7,-0.4) arc [radius=0.4, start angle=270, end angle= 450];
	\draw[thick] (1, 0) to [out=315,in=180] (1.7,-0.4);
	\draw[thick] (1, 0) to [out=45,in=180] (1.7,0.4);

	\draw[thick]  (1.9,-0.6) arc [radius=0.6, start angle=270, end angle= 450];
	\draw[thick] (1, 0) to [out=270,in=180] (1.9,-0.6);
	\draw[thick] (1, 0) to [out=90,in=180] (1.9,0.6);

	\draw[thick]  (-1.7,-1.7) arc [radius=0.4, start angle=270, end angle= 90];
	\draw[thick] (-1, -1.3) to [out=225,in=0] (-1.7,-1.7);
	\draw[thick] (-1, -1.3) to [out=135,in=0] (-1.7,-0.9);

	\draw[thick] (-1, 1.3) -- (-1, 2) node[above]  {\scalebox{0.8}{$a_3$}};
	\draw[thick] (-1, 1.3) -- (-1.5, 1.8) node[above left]  {\scalebox{0.8}{$a_1$}};
	\draw[thick] (-1, 1.3) -- (-0.5, 1.8) node[above right]  {\scalebox{0.8}{$a_7$}};
	\draw[thick] (-1, -1.3) -- (-1, -2) node[below]  {\scalebox{0.8}{$a_2$}};
	\draw[thick] (-1, -1.3) -- (-0.5, -1.8) node[below right]  {\scalebox{0.8}{$a_5$}};

	\draw[fill=white, thick] (1,0) circle [radius=10pt];
	\draw[fill=white, thick] (-1,1.3) circle [radius=10pt];
	\draw[fill=white, thick] (-1,-1.3) circle [radius=10pt];

	\node at (1,0) {\scalebox{1}{$\mu_2$}};
	\node at (-1,1.3) {\scalebox{1}{$\mu_1$}};
	\node at (-1,-1.3) {\scalebox{1}{$\mu_3$}};
	\end{tikzpicture}}}
	\]
The bracket $\{\gamma ,\zeta\}$ of two such graphs $\gamma$ and $\zeta$ is equal to the sum over the pairs of elements of $A$, one $a_i$ from $\gamma$ and one $a_j$ from $\zeta$, of the pairing $\langle a_i, a_j \rangle$ and the graph created by the grafting the two associated edges together.
\end{proposition}

\begin{proof}
This follows directly form the description of the bar construction given below \cref{def:BarCons}.
So the edges should receive degree $-1$, that is an odd degree in the $\Z2$-grading.  The minus
signs in the differential $\d=\d_A-\d_1-\d_2$
come
pfrom the fact that we consider the linear dual of the bar construction.
\end{proof}

\begin{remark}\label{rem:FunctGA}
The assignment $A \mapsto \g_A$ is a functor from the
category
of dg symmetric vector spaces
to the category of shifted dg  $\Delta$-Lie algebras.
\end{remark}

The underlying space $\widehat{\mathbb{G}}_{\od}\big(\Hbul\big(\oM\big)\big)(A)$ of graphs admits several gradings: the arity (number of elements of $A$), the total genus (the sum of the genus of the graph with the genera of the elements labeling the vertices), the weight (number of vertices), the homological degree (the total degree of the various labeling elements) which is here either even or odd, etc. The master equation becomes
\[ \d_A\alpha - \d_1 \alpha - \d_2 \alpha +\Delta \alpha + {\textstyle \frac12} \{\alpha, \alpha\}=0\ ,\]
where $\d_A$, $\d_2$, and $\{\ ,\,\}$ preserve the genus and where $\d_2$ and $\Delta$ increases it by $1$. \cref{prop:MorphTw} shows that solutions to this master equation correspond bijectively to homotopy cohomological field theories structures on $A$.

\begin{corollary}\label{cor:CohFTTFTtree}\leavevmode
\begin{enumerate}
\item The solutions $\alpha$ concentrated in weight $1$ without tadpoles, that is $\alpha\in \Hbul\big(\oM\big)(A)$, coincide with cohomological field theories without unit.

\item If moreover, $\alpha$ in concentrated in $H^0\big(\oM\big)(A)$, one gets a topological field theory.

\item The solutions concentrated in weight $1$ and total genus $0$, i.e.\ $\alpha\in \Hbul\big(\oM_0\big)(A)$ are \emph{tree level} cohomological field theories, in the terminology of~\cite{Manin99}.
\end{enumerate}
\end{corollary}

\begin{proof}\leavevmode
\begin{enumerate}

\item Note that the master equation splits into two according to the weight: the equation
$-\d_1 \alpha +\Delta \alpha =0$, which lives in weight $1$, and the equation
$-\d_2 \alpha  + \frac12 \{\alpha, \alpha\}=0$, which lives in weight $2$. These are the two
relations defining the notion of a cohomological field theory without unit: the former is
Relation~(4.8) of~\cite[Chapter~III]{Manin99} and the latter is Relation~(4.7) of \emph{op.\ cit.}.

\item[(2)-(3)] These are straightforward from the definitions.
\end{enumerate}
\end{proof}

Recall that any shifted dg Lie algebra can be twisted by any
Maurer--Cartan element $\varphi$.
In the present case, it gives the shifted dg  Lie algebra
\[
\g^\varphi_A\coloneq
 \left(
\widehat{\mathbb{G}}_{\od}\big(\Hbul\big(\oM\big)\big)(A), d^\varphi, \{\ ,\,\}
\right)\ ,
\]
where $d^\varphi\coloneq \d +\Delta +\{\varphi, -\}$\ .

\begin{definition}[Homology groups of a (homotopy) CohFT]
The \emph{homology groups} of a (homotopy) CohFT $(A, \varphi)$ are defined
as the homology
$H_\bullet\left(\, \g_A^\varphi\right)$ with respect to the twisted differential $d^\varphi$\ .
\end{definition}

One can now adapt all the general methods of deformation theories controlled by dg Lie algebras, see~\cite[Section~12.2]{LodayVallette12} for instance, in order to study the deformation properties of homotopy CohFTs with their homology groups.

\begin{definition}[Deformations]
For any local algebra $\R\cong\K \oplus \m$ with residue field $\K$ and for any homotopy cohomological field theory $(A, \varphi)$, with $\varphi\in \MC\left(\g_A\right)$, an \emph{$\R$-deformation} of $\varphi$ is an $\R$-linear homotopy cohomological field theory $\Phi$ on $A\otimes \R$ which is equal to $\varphi$ modulo $\m$. We denote the set of such deformations by $\mathrm{Def}_\varphi(\R)$.
\end{definition}

An \emph{infinitesimal deformation} is a deformation over the algebra
$\K[\varepsilon]/(\varepsilon^2)$
of dual numbers.
A \emph{formal deformation} is a deformation over the algebra
$\K[[\hbar]]$
of formal power series.

\begin{lemma}\label{lem:DEF}
The set of $\R$-deformations of a homotopy CohFT $\varphi$ is in natural bijection with the following set of solutions to the Maurer--Cartan equation:
\[\mathrm{Def}_\varphi(\R)\cong \MC\left(
\g_A^\varphi\otimes \m
\right)\ .\]
\end{lemma}

\begin{proof}
The $\R$-linear convolution algebra associated to $A\otimes \R$ is isomorphic to $\g \otimes \R$.
Thus a homotopy cohomological field theory structure $\Phi$ on $A\otimes \R$ which is equal to $\varphi$ modulo $\m$ amounts to a Maurer--Cartan element of $\g_A \otimes \R$ whose restriction to $\g$ is equal to $\varphi$. Under the isomorphism
$\g_A\otimes \R\cong \g_A \oplus \g_A \otimes \m$, the element $\Phi=\varphi+\bar{\Phi}$ satisfies the master equation of $\g\otimes \R$ if and only if $\bar{\Phi}$ satisfies
\[
d \left(\bar{\Phi}\right) + \left\{\varphi, \bar{\Phi}\right\}+{\textstyle \frac12} \left\{\bar{\Phi}, \bar{\Phi}\right\}= d^\varphi\left(\bar{\Phi}\right) + {\textstyle \frac12} \left\{\bar{\Phi}, \bar{\Phi}\right\}=0
\ , \]
which is the Maurer--Cartan equation in $\g_A^\varphi\otimes \m$\ .
\end{proof}

Let us first consider the case of CohFTs on a dg vector space $A$ with
the trivial differential $\d_A=0$.
We say that an element of the twisted convolution algebra
$\g_A=\mathbb{G}_{\od}\big(\Hbul\big(\oM\big)\big)(A)$ has \emph{syzygy degree} $n$
if it belongs to the subspace $\g_A^{[n]}$ spanned by labeled graphs with $n$ edges.
With respect to the syzygy degree, the space $\g_A$ forms a cohomological non-negatively graded
chain complex.
Maurer--Cartan elements $\varphi$ of syzygy degree $0$ define CohFT structures on $A$.

\medskip

An infinitesimal deformation of a CohFT $\varphi$,
called a ``first order deformation'' in~\cite{PZ18},
is an element $\lambda\in (\g_A)^{[0]}_\ev$ satisfying
\[\d\lambda+\Delta \lambda+\{\varphi, \lambda\}=0\ .
\]
This equation is equivalent to the two equations
$\d_1 \lambda =\Delta \lambda$ of weight 1 and
$\d_2 \lambda=   \{\varphi, \lambda\}$ of weight 2,
which are respectively Equation~(iiq) and Equation~(iir) in~\cite[Section~2.1]{PZ18}.

\begin{proposition}[Infinitesimal deformations of CohFTs]
For any cohomological field theory $\varphi$ on a graded symmetric vector space $(A, \langle\ ,\, \rangle)$, there is a  canonical bijection
\[
\mathrm{Def}_\varphi\left(\K[\varepsilon]/\left(\varepsilon^2\right)\right)\cong Z^{[0]}\left(\,\g_A^\varphi\right)\ .
\]
\end{proposition}

\begin{proof}
This is a direct application of \cref{lem:DEF}.
\end{proof}

\begin{proposition}[Obstructions to formal deformations of CohFTs]\label{prop:Obs}
Let  $\varphi$ be a  cohomological field theory on a graded vector symmetric space $(A, \langle\ ,\, \rangle)$.
When $H^{[1]}\left(\,\g_A^\varphi\right)=0$, any even cocycle of $Z^{[0]}\left(\,\g_A^\varphi\right)$ extends to a formal deformation of  the CohFT $\varphi$.
\end{proposition}

\begin{proof}
A series $\varphi+\sum_{n\geqslant 1} \varphi_n t^n$ is a formal deformation of $\varphi$ if and only if
\begin{equation}\label{eq:FormalDef}
\underbrace{d \left(\varphi_n\right) + \left\{\varphi_n, \varphi\right\}}_{=d^\varphi\left(\varphi_n\right)}+
\underbrace{{\textstyle \frac12}\sum_{k=1}^{n-1} \left\{\varphi_k, \varphi_{n-k}\right\}}_{\coloneq \Phi_{n}}=0 \ ,
\end{equation}
for any $n\geqslant 1$. Let $\varphi_1\in Z^{[0]}\left(\,\g_A^\varphi\right)$ be an even cocycle, so it satisfies \cref{eq:FormalDef} for $n=1$ since $\Phi_{1}=0$. Suppose now that there exist elements $\varphi_1, \ldots, \varphi_{n-1}\in (\g_A)^{[0]}$ which satisfy \cref{eq:FormalDef} up to $n-1$. A direct computation shows that $\Phi_{n}$ is an even cocycle, that is $d^\varphi\left(\Phi_{n}\right)=0$, which concludes the proof.
\end{proof}

Let us now consider the case when the underlying differential $\d_A$ of $A$ is not necessarily
trivial. In this case, each syzygy degree summand $(\g_A)^{[n]}$ is a chain complex with respect to
$\d_A$\ .

\begin{proposition}[Obstructions for homotopy CohFTs]
Let $(A, d_A, \langle\ ,\, \rangle)$ be a dg symmetric vector space  with  $\varphi\in (\g_A)^{[0]}_\ev$ such that $\d_A(\varphi)=0$. When $H_\od\left((\g_A)^{[n]}, \d_A\right)=0$, for any $n\geqslant 1$, then $\varphi$ extends to a homotopy CohFT.
\end{proposition}

\begin{proof}
The method is similar to the one of \cref{prop:Obs}.
A homotopy CohFT is a Maurer--Cartan of the form $\sum_{n\geqslant 0} \varphi^{[n]}$, where $\varphi^{[n]}\in (\g_A)^{[n]}_\ev$; equivalently they satisfy the equation
\begin{equation}\label{eq:HoCohFT}
\d_A\left(\varphi^{[n]}\right)+
\underbrace{(-\d_1-\d_2+\Delta)\left(\varphi^{[n-1]}\right)+
{\textstyle \frac12}\sum_{k=1}^{n-1} \left\{\varphi^{[k]}, \varphi^{[n-1-k]}\right\}}_{\Phi^{[n]}\coloneq}=0
\ ,
\end{equation}
for any $n\geqslant 0$. When $\varphi$ is an even cycle of $(\g_A)^{[0]}$ with respect to $\d_A$, the element $\varphi^{[0]}\coloneq \varphi$ satisfies \cref{eq:HoCohFT} for $n=0$.
Suppose now that there exist elements $\varphi^{[k]}\in (\g_A)^{[k]}_\ev$, for $0\leqslant k \leqslant n-1$,  which satisfy \cref{eq:HoCohFT} up to $n-1$. A direct computation shows that $\Phi^{[n]}$ is an even cycle with respect to $\d_A$, that is $\d_A\left(\Phi^{[n]}\right)=0$, which concludes the proof.
\end{proof}

In order to get homological interpretation of the deformation theory of homotopy CohFTs, we go back by considering the convolution algebra $\g$ with its underlying homological degree.

\begin{proposition}[Infinitesimal deformations of homotopy CohFTs]
For any homotopy cohomological field theory $\varphi$ on a dg symmetric vector space $(A, d_A, \langle\ ,\, \rangle)$, there is a canonical bijection
\[\mathrm{Def}_\varphi\left(\K[\varepsilon]/\left(\varepsilon^2\right)\right)\cong Z_\ev\left(\g_A^\varphi\right)\ .\]
\end{proposition}

\begin{proof}
This is a direct application of \cref{lem:DEF}.
\end{proof}

\begin{proposition}[Obstructions for formal deformations of homotopy CohFTs]
Let $\varphi$ be a homotopy cohomological field theory on a chain complex $(A, d_A, \langle\ ,\, \rangle)$.
When $H_\od\left(\g_A^\varphi\right)=0$, any even cycle of $Z_\ev\left(\,{\g}_A^\varphi\right)$ extends to a formal deformation of the homotopy CohFT $\varphi$.
\end{proposition}

\begin{proof}
This proof is similar to the one of \cref{prop:Obs}.
\end{proof}

To go even further in the deformation properties of homotopy CohFTs,
one would need to introduce a suitable notion of morphism and $\infty$-morphism in the homotopy case, see~\cite[Section~12.2.10]{LodayVallette12} for more details in the case of operads.

\subsection{Pandharipande--Zvonkine's construction}\label{subsec:PZ}
In this section, we cast Pandharipande--Zvon\-kine's construction of cohomological field theories with possibly non-tautological classes~\cite{PZ18} in the general context of deformation of (homotopy) CohFTs. This method allows us to introduce a wider class of examples.

\medskip

Motivated by the Frobenius algebra structure on the cohomology $H^\bullet(X, \Co)$ of a genus $m$ curve, we consider the following $\ZZ$-graded vector space
\[
A_{-1}\coloneqq \K\{ b_1, \ldots, b_m\}\ , \quad
A_{0}\coloneqq \K\{ a, d\}\ , \quad
A_{1}\coloneqq \K\{ c_1, \ldots, c_m\}\ ,
\]
endowed with the following non-degenerate symmetric pairing
\[
\langle b_i, c_i \rangle \coloneqq 1 \quad \text{and} \quad
\langle a, d \rangle \coloneqq 1 \ .
\]

We consider the following even degree element in $\widehat{\mathbb{G}}_\od\big(H^\bullet\big(\oM\big)\big)(A)$:
\begin{equation}\label{eqn:alpha}
\alpha\coloneqq \sum_{n\geqslant 3}
\frac{1}{(n-1)!}
	\vcenter{\hbox{\begin{tikzpicture}[scale=1]
	\draw[thick] (0:0) -- (0:0.7);
	\draw[thick] (45:0) -- (45:0.7);
	\draw[thick] (90:0) -- (90:0.7) ;
	\draw[thick] (135:0) -- (135:0.7);
	\draw[thick] (180:0) -- (180:0.7);
	\draw[thick] (225:0) -- (225:0.7);
	\draw[thick] (270:0) -- (270:0.7);
	\draw[thick] (315:0) -- (315:0.7);
	\draw[fill=white, thick] (0,0) circle [radius=11pt];
	\node at (0,0) {\scalebox{1}{$\mathbb{1}_{0,n}$}};
	\node at (0:0.9) {\scalebox{0.9}{$d$}};
	\node at (45:0.9) {\scalebox{0.9}{$d$}};
	\node at (90:0.9) {\scalebox{0.9}{$a$}};
	\node at (135:0.9) {\scalebox{0.9}{$d$}};
	\node at (180:0.9) {\scalebox{0.9}{$d$}};
	\end{tikzpicture}}}
+\sum_{\substack{n\geqslant 3\\ i=1, \ldots, m}}
\frac{1}{(n-2)!}
\vcenter{\hbox{\begin{tikzpicture}[scale=1]
	\draw[thick] (0:0) -- (0:0.7);
	\draw[thick] (45:0) -- (45:0.7);
	\draw[thick] (90:0) -- (90:0.7) ;
	\draw[thick] (135:0) -- (135:0.7);
	\draw[thick] (180:0) -- (180:0.7);
	\draw[thick] (225:0) -- (225:0.7);
	\draw[thick] (270:0) -- (270:0.7);
	\draw[thick] (315:0) -- (315:0.7);
	\draw[fill=white, thick] (0,0) circle [radius=11pt];
	\node at (0,0) {\scalebox{1}{$\mathbb{1}_{0,n}$}};
	\node at (0:0.9) {\scalebox{0.9}{$d$}};
	\node at (45:0.9) {\scalebox{0.9}{$b_i$}};
	\node at (90:0.9) {\scalebox{0.9}{$c_i$}};
	\node at (135:0.9) {\scalebox{0.9}{$d$}};
	\node at (180:0.9) {\scalebox{0.9}{$d$}};
	\end{tikzpicture}}}
+(2-2m)\sum_{n\geqslant 1} \frac{1}{n!}
\vcenter{\hbox{\begin{tikzpicture}[scale=1]
	\draw[thick] (0:0) -- (0:0.7);
	\draw[thick] (45:0) -- (45:0.7);
	\draw[thick] (90:0) -- (90:0.7) ;
	\draw[thick] (135:0) -- (135:0.7);
	\draw[thick] (180:0) -- (180:0.7);
	\draw[thick] (225:0) -- (225:0.7);
	\draw[thick] (270:0) -- (270:0.7);
	\draw[thick] (315:0) -- (315:0.7);
	\draw[fill=white, thick] (0,0) circle [radius=11pt];
	\node at (0,0) {\scalebox{1}{$\mathbb{1}_{1,n}$}};
	\node at (0:0.9) {\scalebox{0.9}{$d$}};
	\node at (45:0.9) {\scalebox{0.9}{$d$}};
	\node at (90:0.9) {\scalebox{0.9}{$d$}};
	\node at (135:0.9) {\scalebox{0.9}{$d$}};
	\node at (180:0.9) {\scalebox{0.9}{$d$}};
	\end{tikzpicture}}}\ ,
\end{equation}
where $\mathbb{1}_{g,n}$ stands for the generator of $H^0\big(\overline{\mathcal{M}}_{g,n}\big)$\ .

\begin{remark}
There is slight discrepancy between the present presentation and~\cite{PZ18}, since
in \emph{op.\ cit.} the authors consider rather $\alpha\in H^\bullet\big(\oM\big)(A^*)$\ .
\end{remark}

\begin{lemma}[{\cite[Proposition~8]{PZ18}}]\label{lem:PZalpha}
The element $\alpha$ is solution to the master equation \eqref{eq:MaEq}:
\[ \d \alpha + \Delta \alpha + {\textstyle \frac12} \{\alpha, \alpha\}=0\ .\]
\end{lemma}

In other words, the element $\alpha$ defines a topological field theory and thus a
CohFT.

\begin{proof}
Even though a proof of this fact is given in \emph{op.\ cit.}, we give here another one
which uses the algebraic structure of
our
deformation theory. Let us illustrate various cancellations in
$\d \alpha + \Delta \alpha + {\textstyle \frac12} \{\alpha, \alpha\}$:
\begin{align*}
\d_A \alpha=&0\ ,\\
-\d_1\alpha=&-\sum_{n\geqslant 3} {\textstyle \frac{2-2m}{2(n-2)!}}
\vcenter{\hbox{\begin{tikzpicture}[scale=0.9]
	\draw[thick] (90:0) -- (90:0.7) ;
	\draw[thick] (135:0) -- (135:0.7);
	\draw[thick] (180:0) -- (180:0.7) ;
	\draw[thick] (225:0) -- (225:0.7);
	\draw[thick] (270:0) -- (270:0.7);
	\draw[thick] (0.7,-0.4) arc [radius=0.4, start angle=270, end angle= 450];
	\draw[thick] (0, 0) to [out=315,in=180] (0.7,-0.4);
	\draw[thick] (0, 0) to [out=45,in=180] (0.7,0.4);
	\draw[fill=white, thick] (0,0) circle [radius=11pt];
	\node at (0,0) {\scalebox{1}{$\mathbb{1}_{0,n}$}};
	\node at (270:0.9) {\scalebox{0.9}{$d$}};
	\node at (90:0.9) {\scalebox{0.9}{$d$}};
	\node at (135:0.9) {\scalebox{0.9}{$d$}};
	\node at (225:0.9) {\scalebox{0.9}{$d$}};
	\end{tikzpicture}}}\ ,\\
-\d_2\alpha=&
-\sum_{k,l\geqslant 3} {\textstyle \frac{1}{(k-2)!(l-1)!}}\,
\vcenter{\hbox{\begin{tikzpicture}[scale=0.9]
	\draw[thick] (0:0) -- (0:2.7);
	\draw[thick] (45:0) -- (45:0.7);
	\draw[thick] (90:0) -- (90:0.7) ;
	\draw[thick] (135:0) -- (135:0.7);
	\draw[thick] (180:0) -- (180:0.7);
	\draw[thick] (225:0) -- (225:0.7);
	\draw[thick] (270:0) -- (270:0.7);
	\draw[thick] (315:0) -- (315:0.7);
	\draw[thick] (2,0) -- (2.5,0.5);
	\draw[thick] (2,0) -- (2,0.7);
	\draw[thick] (2,0) -- (1.5,0.5);
	\draw[thick] (2,0) -- (1.3,0);
	\draw[thick] (2,0) -- (1.5,-0.5);
	\draw[thick] (2,0) -- (2,-0.7);
	\draw[thick] (2,0) -- (2.5,-0.5);
	\draw[fill=white, thick] (0,0) circle [radius=11pt];
	\draw[fill=white, thick] (2,0) circle [radius=11pt];
	\node at (0,0) {\scalebox{1}{$\mathbb{1}_{0,k}$}};
	\node at (2,0) {\scalebox{1}{$\mathbb{1}_{0,l}$}};
	\node at (45:0.9) {\scalebox{0.9}{$d$}};
	\node at (315:0.9) {\scalebox{0.9}{$d$}};
	\node at (90:0.9) {\scalebox{0.9}{$a$}};
	\node at (135:0.9) {\scalebox{0.9}{$d$}};
	\node at (1.33,0.67) {\scalebox{0.9}{$d$}};
	\node at (1.36,-0.64) {\scalebox{0.9}{$d$}};
	\node at (2,-0.9) {\scalebox{0.9}{$d$}};
	\node at (2,0.9) {\scalebox{0.9}{$d$}};
\end{tikzpicture}}}
-\sum_{\substack{k,l\geqslant 3\\  i=1, \ldots, m}}
{\textstyle\frac{1}{(k-2)!(l-2)!}}\,
\vcenter{\hbox{\begin{tikzpicture}[scale=0.9]
	\draw[thick] (0:0) -- (0:2.7);
	\draw[thick] (45:0) -- (45:0.7);
	\draw[thick] (90:0) -- (90:0.7) ;
	\draw[thick] (135:0) -- (135:0.7);
	\draw[thick] (180:0) -- (180:0.7);
	\draw[thick] (225:0) -- (225:0.7);
	\draw[thick] (270:0) -- (270:0.7);
	\draw[thick] (315:0) -- (315:0.7);
	\draw[thick] (2,0) -- (2.5,0.5);
	\draw[thick] (2,0) -- (2,0.7);
	\draw[thick] (2,0) -- (1.5,0.5);
	\draw[thick] (2,0) -- (1.3,0);
	\draw[thick] (2,0) -- (1.5,-0.5);
	\draw[thick] (2,0) -- (2,-0.7);
	\draw[thick] (2,0) -- (2.5,-0.5);
	\draw[fill=white, thick] (0,0) circle [radius=11pt];
	\draw[fill=white, thick] (2,0) circle [radius=11pt];
	\node at (0,0) {\scalebox{1}{$\mathbb{1}_{0,k}$}};
	\node at (2,0) {\scalebox{1}{$\mathbb{1}_{0,l}$}};
	\node at (45:0.9) {\scalebox{0.9}{$d$}};
	\node at (315:0.9) {\scalebox{0.9}{$d$}};
	\node at (90:0.9) {\scalebox{0.9}{$c_i$}};
	\node at (135:0.9) {\scalebox{0.9}{$d$}};
	\node at (1.33,0.67) {\scalebox{0.9}{$d$}};
	\node at (1.36,-0.64) {\scalebox{0.9}{$d$}};
	\node at (2,0.9) {\scalebox{0.9}{$b_i$}};
	\node at (2.67,0.67) {\scalebox{0.9}{$d$}};
\end{tikzpicture}}}\\
&-\sum_{\substack{k,l\geqslant 3\\  i=1, \ldots, m}}
{\textstyle\frac{1}{(k-3)!(l-1)!}}\,
\vcenter{\hbox{\begin{tikzpicture}[scale=0.9]
	\draw[thick] (0:0) -- (0:2.7);
	\draw[thick] (45:0) -- (45:0.7);
	\draw[thick] (90:0) -- (90:0.7) ;
	\draw[thick] (135:0) -- (135:0.7);
	\draw[thick] (180:0) -- (180:0.7);
	\draw[thick] (225:0) -- (225:0.7);
	\draw[thick] (270:0) -- (270:0.7);
	\draw[thick] (315:0) -- (315:0.7);
	\draw[thick] (2,0) -- (2.5,0.5);
	\draw[thick] (2,0) -- (2,0.7);
	\draw[thick] (2,0) -- (1.5,0.5);
	\draw[thick] (2,0) -- (1.3,0);
	\draw[thick] (2,0) -- (1.5,-0.5);
	\draw[thick] (2,0) -- (2,-0.7);
	\draw[thick] (2,0) -- (2.5,-0.5);
	\draw[fill=white, thick] (0,0) circle [radius=11pt];
	\draw[fill=white, thick] (2,0) circle [radius=11pt];
	\node at (0,0) {\scalebox{1}{$\mathbb{1}_{0,k}$}};
	\node at (2,0) {\scalebox{1}{$\mathbb{1}_{0,l}$}};
	\node at (45:0.9) {\scalebox{0.9}{$d$}};
	\node at (315:0.9) {\scalebox{0.9}{$d$}};
	\node at (90:0.9) {\scalebox{0.9}{$c_i$}};
	\node at (135:0.9) {\scalebox{0.9}{$b_i$}};
	\node at (180:0.9) {\scalebox{0.9}{$d$}};
	\node at (1.33,0.67) {\scalebox{0.9}{$d$}};
	\node at (1.36,-0.64) {\scalebox{0.9}{$d$}};
	\node at (2,0.9) {\scalebox{0.9}{$d$}};
	\node at (2,-0.9) {\scalebox{0.9}{$d$}};
\end{tikzpicture}}}
-\sum_{\substack{k\geqslant 3,\\ l\geqslant 1}}
{\textstyle\frac{2-2m}{(k-1)!(l-1)!}}\,
\vcenter{\hbox{\begin{tikzpicture}[scale=0.9]
	\draw[thick] (0:0) -- (0:2.7);
	\draw[thick] (45:0) -- (45:0.7);
	\draw[thick] (90:0) -- (90:0.7) ;
	\draw[thick] (135:0) -- (135:0.7);
	\draw[thick] (180:0) -- (180:0.7);
	\draw[thick] (225:0) -- (225:0.7);
	\draw[thick] (270:0) -- (270:0.7);
	\draw[thick] (315:0) -- (315:0.7);
	\draw[thick] (2,0) -- (2.5,0.5);
	\draw[thick] (2,0) -- (2,0.7);
	\draw[thick] (2,0) -- (1.5,0.5);
	\draw[thick] (2,0) -- (1.3,0);
	\draw[thick] (2,0) -- (1.5,-0.5);
	\draw[thick] (2,0) -- (2,-0.7);
	\draw[thick] (2,0) -- (2.5,-0.5);
	\draw[fill=white, thick] (0,0) circle [radius=11pt];
	\draw[fill=white, thick] (2,0) circle [radius=11pt];
	\node at (0,0) {\scalebox{1}{$\mathbb{1}_{0,k}$}};
	\node at (2,0) {\scalebox{1}{$\mathbb{1}_{1,l}$}};
	\node at (45:0.9) {\scalebox{0.9}{$d$}};
	\node at (315:0.9) {\scalebox{0.9}{$d$}};
	\node at (90:0.9) {\scalebox{0.9}{$d$}};
	\node at (270:0.9) {\scalebox{0.9}{$d$}};
	\node at (1.33,0.67) {\scalebox{0.9}{$d$}};
	\node at (1.36,-0.64) {\scalebox{0.9}{$d$}};
	\node at (2,-0.9) {\scalebox{0.9}{$d$}};
	\node at (2,0.9) {\scalebox{0.9}{$d$}};
\end{tikzpicture}}}\ ,\\
\Delta\alpha=&
\sum_{n\geqslant 3}
{\textstyle\frac{1}{(n-2)!}}
\vcenter{\hbox{\begin{tikzpicture}[scale=0.9]
	\draw[thick] (90:0) -- (90:0.7) ;
	\draw[thick] (135:0) -- (135:0.7);
	\draw[thick] (180:0) -- (180:0.7) ;
	\draw[thick] (225:0) -- (225:0.7);
	\draw[thick] (270:0) -- (270:0.7);
	\draw[thick] (0.7,-0.4) arc [radius=0.4, start angle=270, end angle= 450];
	\draw[thick] (0, 0) to [out=315,in=180] (0.7,-0.4);
	\draw[thick] (0, 0) to [out=45,in=180] (0.7,0.4);
	\draw[fill=white, thick] (0,0) circle [radius=11pt];
	\node at (0,0) {\scalebox{1}{$\mathbb{1}_{0,n}$}};
	\node at (270:0.9) {\scalebox{0.9}{$d$}};
	\node at (90:0.9) {\scalebox{0.9}{$d$}};
	\node at (135:0.9) {\scalebox{0.9}{$d$}};
	\node at (225:0.9) {\scalebox{0.9}{$d$}};
	\end{tikzpicture}}}
-\sum_{n\geqslant 3}
{\textstyle\frac{m}{(n-2)!}}
\vcenter{\hbox{\begin{tikzpicture}[scale=0.9]
	\draw[thick] (90:0) -- (90:0.7) ;
	\draw[thick] (135:0) -- (135:0.7);
	\draw[thick] (180:0) -- (180:0.7) ;
	\draw[thick] (225:0) -- (225:0.7);
	\draw[thick] (270:0) -- (270:0.7);
	\draw[thick] (0.7,-0.4) arc [radius=0.4, start angle=270, end angle= 450];
	\draw[thick] (0, 0) to [out=315,in=180] (0.7,-0.4);
	\draw[thick] (0, 0) to [out=45,in=180] (0.7,0.4);
	\draw[fill=white, thick] (0,0) circle [radius=11pt];
	\node at (0,0) {\scalebox{1}{$\mathbb{1}_{0,n}$}};
	\node at (270:0.9) {\scalebox{0.9}{$d$}};
	\node at (90:0.9) {\scalebox{0.9}{$d$}};
	\node at (135:0.9) {\scalebox{0.9}{$d$}};
	\node at (225:0.9) {\scalebox{0.9}{$d$}};
	\end{tikzpicture}}}\ ,\\
\frac{1}{2}[\alpha,\alpha]=&
\sum_{k,l\geqslant 3}
{\textstyle\frac{1}{(k-2)!(l-1)!}}\,
\vcenter{\hbox{\begin{tikzpicture}[scale=0.9]
	\draw[thick] (0:0) -- (0:2.7);
	\draw[thick] (45:0) -- (45:0.7);
	\draw[thick] (90:0) -- (90:0.7) ;
	\draw[thick] (135:0) -- (135:0.7);
	\draw[thick] (180:0) -- (180:0.7);
	\draw[thick] (225:0) -- (225:0.7);
	\draw[thick] (270:0) -- (270:0.7);
	\draw[thick] (315:0) -- (315:0.7);
	\draw[thick] (2,0) -- (2.5,0.5);
	\draw[thick] (2,0) -- (2,0.7);
	\draw[thick] (2,0) -- (1.5,0.5);
	\draw[thick] (2,0) -- (1.3,0);
	\draw[thick] (2,0) -- (1.5,-0.5);
	\draw[thick] (2,0) -- (2,-0.7);
	\draw[thick] (2,0) -- (2.5,-0.5);
	\draw[fill=white, thick] (0,0) circle [radius=11pt];
	\draw[fill=white, thick] (2,0) circle [radius=11pt];
	\node at (0,0) {\scalebox{1}{$\mathbb{1}_{0,k}$}};
	\node at (2,0) {\scalebox{1}{$\mathbb{1}_{0,l}$}};
	\node at (45:0.9) {\scalebox{0.9}{$d$}};
	\node at (315:0.9) {\scalebox{0.9}{$d$}};
	\node at (90:0.9) {\scalebox{0.9}{$a$}};
	\node at (135:0.9) {\scalebox{0.9}{$d$}};
	\node at (1.33,0.67) {\scalebox{0.9}{$d$}};
	\node at (1.36,-0.64) {\scalebox{0.9}{$d$}};
	\node at (2,-0.9) {\scalebox{0.9}{$d$}};
	\node at (2,0.9) {\scalebox{0.9}{$d$}};
\end{tikzpicture}}}
+
\sum_{\substack{k,l\geqslant 3\\  i=1, \ldots, m}}
{\textstyle\frac{1}{(k-3)!(l-1)!}}\,
\vcenter{\hbox{\begin{tikzpicture}[scale=0.9]
	\draw[thick] (0:0) -- (0:2.7);
	\draw[thick] (45:0) -- (45:0.7);
	\draw[thick] (90:0) -- (90:0.7) ;
	\draw[thick] (135:0) -- (135:0.7);
	\draw[thick] (180:0) -- (180:0.7);
	\draw[thick] (225:0) -- (225:0.7);
	\draw[thick] (270:0) -- (270:0.7);
	\draw[thick] (315:0) -- (315:0.7);
	\draw[thick] (2,0) -- (2.5,0.5);
	\draw[thick] (2,0) -- (2,0.7);
	\draw[thick] (2,0) -- (1.5,0.5);
	\draw[thick] (2,0) -- (1.3,0);
	\draw[thick] (2,0) -- (1.5,-0.5);
	\draw[thick] (2,0) -- (2,-0.7);
	\draw[thick] (2,0) -- (2.5,-0.5);
	\draw[fill=white, thick] (0,0) circle [radius=11pt];
	\draw[fill=white, thick] (2,0) circle [radius=11pt];
	\node at (0,0) {\scalebox{1}{$\mathbb{1}_{0,k}$}};
	\node at (2,0) {\scalebox{1}{$\mathbb{1}_{0,l}$}};
	\node at (45:0.9) {\scalebox{0.9}{$d$}};
	\node at (315:0.9) {\scalebox{0.9}{$d$}};
	\node at (90:0.9) {\scalebox{0.9}{$c_i$}};
	\node at (135:0.9) {\scalebox{0.9}{$b_i$}};
	\node at (180:0.9) {\scalebox{0.9}{$d$}};
	\node at (1.33,0.67) {\scalebox{0.9}{$d$}};
	\node at (1.36,-0.64) {\scalebox{0.9}{$d$}};
	\node at (2,0.9) {\scalebox{0.9}{$d$}};
	\node at (2,-0.9) {\scalebox{0.9}{$d$}};
\end{tikzpicture}}}\\
&\sum_{\substack{k,l\geqslant 3\\  i=1, \ldots, m}}
{\textstyle\frac{1}{(k-2)!(l-2)!}}\,
\vcenter{\hbox{\begin{tikzpicture}[scale=0.9]
	\draw[thick] (0:0) -- (0:2.7);
	\draw[thick] (45:0) -- (45:0.7);
	\draw[thick] (90:0) -- (90:0.7) ;
	\draw[thick] (135:0) -- (135:0.7);
	\draw[thick] (180:0) -- (180:0.7);
	\draw[thick] (225:0) -- (225:0.7);
	\draw[thick] (270:0) -- (270:0.7);
	\draw[thick] (315:0) -- (315:0.7);
	\draw[thick] (2,0) -- (2.5,0.5);
	\draw[thick] (2,0) -- (2,0.7);
	\draw[thick] (2,0) -- (1.5,0.5);
	\draw[thick] (2,0) -- (1.3,0);
	\draw[thick] (2,0) -- (1.5,-0.5);
	\draw[thick] (2,0) -- (2,-0.7);
	\draw[thick] (2,0) -- (2.5,-0.5);
	\draw[fill=white, thick] (0,0) circle [radius=11pt];
	\draw[fill=white, thick] (2,0) circle [radius=11pt];
	\node at (0,0) {\scalebox{1}{$\mathbb{1}_{0,k}$}};
	\node at (2,0) {\scalebox{1}{$\mathbb{1}_{0,l}$}};
	\node at (45:0.9) {\scalebox{0.9}{$d$}};
	\node at (315:0.9) {\scalebox{0.9}{$d$}};
	\node at (90:0.9) {\scalebox{0.9}{$c_i$}};
	\node at (135:0.9) {\scalebox{0.9}{$d$}};
	\node at (1.33,0.67) {\scalebox{0.9}{$d$}};
	\node at (1.36,-0.64) {\scalebox{0.9}{$d$}};
	\node at (2,0.9) {\scalebox{0.9}{$b_i$}};
	\node at (2.67,0.67) {\scalebox{0.9}{$d$}};
\end{tikzpicture}}}
+\sum_{\substack{k\geqslant 3,\\ l\geqslant 1}}
{\textstyle\frac{2-2m}{(k-1)!(l-1)!}}\,
\vcenter{\hbox{\begin{tikzpicture}[scale=0.9]
	\draw[thick] (0:0) -- (0:2.7);
	\draw[thick] (45:0) -- (45:0.7);
	\draw[thick] (90:0) -- (90:0.7) ;
	\draw[thick] (135:0) -- (135:0.7);
	\draw[thick] (180:0) -- (180:0.7);
	\draw[thick] (225:0) -- (225:0.7);
	\draw[thick] (270:0) -- (270:0.7);
	\draw[thick] (315:0) -- (315:0.7);
	\draw[thick] (2,0) -- (2.5,0.5);
	\draw[thick] (2,0) -- (2,0.7);
	\draw[thick] (2,0) -- (1.5,0.5);
	\draw[thick] (2,0) -- (1.3,0);
	\draw[thick] (2,0) -- (1.5,-0.5);
	\draw[thick] (2,0) -- (2,-0.7);
	\draw[thick] (2,0) -- (2.5,-0.5);
	\draw[fill=white, thick] (0,0) circle [radius=11pt];
	\draw[fill=white, thick] (2,0) circle [radius=11pt];
	\node at (0,0) {\scalebox{1}{$\mathbb{1}_{0,k}$}};
	\node at (2,0) {\scalebox{1}{$\mathbb{1}_{1,l}$}};
	\node at (45:0.9) {\scalebox{0.9}{$d$}};
	\node at (315:0.9) {\scalebox{0.9}{$d$}};
	\node at (90:0.9) {\scalebox{0.9}{$d$}};
	\node at (270:0.9) {\scalebox{0.9}{$d$}};
	\node at (1.33,0.67) {\scalebox{0.9}{$d$}};
	\node at (1.36,-0.64) {\scalebox{0.9}{$d$}};
	\node at (2,-0.9) {\scalebox{0.9}{$d$}};
	\node at (2,0.9) {\scalebox{0.9}{$d$}};
\end{tikzpicture}}}\ .
\end{align*}
\end{proof}

Associated to any cohomology class $\Lambda\in H^\bullet\big(\overline{\mathcal{M}}_{h,m}\big)$ whose degree has the same parity as $m$, we consider the following degree $0$ element of $\widehat{\mathbb{G}}_\od\big(H^\bullet\big(\oM\big)\big)(A)$:
\begin{equation}\label{eq:alphapluslambda}
\lambda\coloneqq \sum_{k\geqslant m}
\frac{1}{(k-m)!}
	\vcenter{\hbox{\begin{tikzpicture}[scale=1]
	\draw[thick] (0:0) -- (0:0.7);
	\draw[thick] (45:0) -- (45:0.7);
	\draw[thick] (90:0) -- (90:0.7) ;
	\draw[thick] (135:0) -- (135:0.7);
	\draw[thick] (180:0) -- (180:0.7);
	\draw[thick] (225:0) -- (225:0.7);
	\draw[thick] (270:0) -- (270:0.7);
	\draw[thick] (315:0) -- (315:0.7);
	\draw[fill=white, thick] (0,0) circle [radius=11pt];
	\node at (0,0) {\scalebox{1}{$\Lambda_k$}};
	\node at (0:0.9) {\scalebox{0.9}{$d$}};
	\node at (45:0.95) {\scalebox{0.9}{$c_1$}};
	\node at (135:0.95) {\scalebox{0.9}{$c_m$}};
	\node at (180:0.9) {\scalebox{0.9}{$d$}};
	\end{tikzpicture}}}\ ,
\end{equation}
where the notation $\Lambda_k\coloneqq p_k^*(\Lambda)$ stands for the cohomology class defined by the map
 $p_k : \overline{\mathcal{M}}_{h,k} \to \overline{\mathcal{M}}_{h,m}$ which forgets $k-m$ marked points.

\begin{definition}[Minimal class]
A cohomology class $\Lambda$ is called \emph{minimal} when it is a \emph{primitive element} in the modular cooperad
$H^\bullet\big(\overline{\mathcal{M}}\big)$: this means that its image under
any expansion map $\chi_{ij} : H^\bullet\big(\overline{\mathcal{M}}_{g+1,n-2}\big) \to H^\bullet\big(\overline{\mathcal{M}}_{g,n}\big)$ is trivial and that its image under
any partial decomposition map
$\delta_{i}^j : H^\bullet\big(\overline{\mathcal{M}}_{g+g',n+n'-2}\big) \to
H^\bullet\big(\overline{\mathcal{M}}_{g,n}\big)
\otimes H^\bullet\big(\overline{\mathcal{M}}_{g',n'}\big)$ is trivial.
\end{definition}

\begin{theorem}[{\cite[Theorem~5]{PZ18}}]\label{thm:PZalphalambda}
If  $\Lambda\in H^\bullet\big(\overline{\mathcal{M}}_{h,m}\big)$ is a is minimal
cohomology class, the element $\alpha+\lambda$ is a
solution to the master equation \eqref{eq:MaEq}:
 \[\ \d (\alpha+\lambda) + \Delta (\alpha+\lambda) + {\textstyle \frac12}
   \{\alpha+\lambda, \alpha+\lambda\}=0\ .
 \]
\end{theorem}

In other words, the element $\alpha+\lambda$ defines a CohFT.

\begin{proof}
Since $\alpha$ is a CohFT, we have
\begin{align*}
\d (\alpha+\lambda) + \Delta (\alpha+\lambda) + {\textstyle \frac12} \{\alpha+\lambda, \alpha+\lambda\}=
\underbrace{\d \alpha + \Delta \alpha + {\textstyle \frac12} \{\alpha, \alpha\}}_{=0}+
\d \lambda + \Delta \lambda +  \{\alpha, \lambda\}+ {\textstyle \frac12} \{\lambda, \lambda\}\ .
\end{align*}
Since the sub-space $ \K\{ a, b_1, \ldots, b_m\}$ is isotropic, we have
$\Delta \lambda=0$ and $\{\lambda, \lambda\}=0$. It remains to show that $-\d \lambda = \d_1 \lambda +\d_2 \lambda=\{\alpha, \lambda\}$ to conclude the proof.
Let us first compute $\{\alpha, \lambda\}=$
\begin{align*}
&
\sum_{\substack{n\geqslant 3\\ k\geqslant m}}
\left\{
{\textstyle \frac{1}{(n-1)!}}
	\vcenter{\hbox{\begin{tikzpicture}[scale=0.9]
	\draw[thick] (0:0) -- (0:0.7);
	\draw[thick] (45:0) -- (45:0.7);
	\draw[thick] (90:0) -- (90:0.7) ;
	\draw[thick] (135:0) -- (135:0.7);
	\draw[thick] (180:0) -- (180:0.7);
	\draw[thick] (225:0) -- (225:0.7);
	\draw[thick] (270:0) -- (270:0.7);
	\draw[thick] (315:0) -- (315:0.7);
	\draw[fill=white, thick] (0,0) circle [radius=11pt];
	\node at (0,0) {\scalebox{1}{$\mathbb{1}_{0,n}$}};
	\node at (0:0.9) {\scalebox{0.9}{$d$}};
	\node at (45:0.9) {\scalebox{0.9}{$d$}};
	\node at (90:0.9) {\scalebox{0.9}{$a$}};
	\node at (135:0.9) {\scalebox{0.9}{$d$}};
	\node at (180:0.9) {\scalebox{0.9}{$d$}};
	\end{tikzpicture}}}
+\sum_{i=1, \ldots, m}
{\textstyle \frac{1}{(n-2)!}}
\vcenter{\hbox{\begin{tikzpicture}[scale=0.9]
	\draw[thick] (0:0) -- (0:0.7);
	\draw[thick] (45:0) -- (45:0.7);
	\draw[thick] (90:0) -- (90:0.7) ;
	\draw[thick] (135:0) -- (135:0.7);
	\draw[thick] (180:0) -- (180:0.7);
	\draw[thick] (225:0) -- (225:0.7);
	\draw[thick] (270:0) -- (270:0.7);
	\draw[thick] (315:0) -- (315:0.7);
	\draw[fill=white, thick] (0,0) circle [radius=11pt];
	\node at (0,0) {\scalebox{1}{$\mathbb{1}_{0,n}$}};
	\node at (0:0.9) {\scalebox{0.9}{$d$}};
	\node at (45:0.9) {\scalebox{0.9}{$b_i$}};
	\node at (90:0.9) {\scalebox{0.9}{$c_i$}};
	\node at (135:0.9) {\scalebox{0.9}{$d$}};
	\node at (180:0.9) {\scalebox{0.9}{$d$}};
	\end{tikzpicture}}}
+{\textstyle { \frac{2-2m}{n!}}}
\vcenter{\hbox{\begin{tikzpicture}[scale=0.9]
	\draw[thick] (0:0) -- (0:0.7);
	\draw[thick] (45:0) -- (45:0.7);
	\draw[thick] (90:0) -- (90:0.7) ;
	\draw[thick] (135:0) -- (135:0.7);
	\draw[thick] (180:0) -- (180:0.7);
	\draw[thick] (225:0) -- (225:0.7);
	\draw[thick] (270:0) -- (270:0.7);
	\draw[thick] (315:0) -- (315:0.7);
	\draw[fill=white, thick] (0,0) circle [radius=11pt];
	\node at (0,0) {\scalebox{1}{$\mathbb{1}_{1,n}$}};
	\node at (0:0.9) {\scalebox{0.9}{$d$}};
	\node at (45:0.9) {\scalebox{0.9}{$d$}};
	\node at (90:0.9) {\scalebox{0.9}{$d$}};
	\node at (135:0.9) {\scalebox{0.9}{$d$}};
	\node at (180:0.9) {\scalebox{0.9}{$d$}};
\end{tikzpicture}}},
{\textstyle \frac{1}{(k-m)!}}
	\vcenter{\hbox{\begin{tikzpicture}[scale=0.9]
	\draw[thick] (0:0) -- (0:0.7);
	\draw[thick] (45:0) -- (45:0.7);
	\draw[thick] (90:0) -- (90:0.7) ;
	\draw[thick] (135:0) -- (135:0.7);
	\draw[thick] (180:0) -- (180:0.7);
	\draw[thick] (225:0) -- (225:0.7);
	\draw[thick] (270:0) -- (270:0.7);
	\draw[thick] (315:0) -- (315:0.7);
	\draw[fill=white, thick] (0,0) circle [radius=11pt];
	\node at (0,0) {\scalebox{1}{$\Lambda_k$}};
	\node at (0:0.9) {\scalebox{0.9}{$d$}};
	\node at (45:0.95) {\scalebox{0.9}{$c_1$}};
	\node at (135:0.95) {\scalebox{0.9}{$c_m$}};
	\node at (180:0.9) {\scalebox{0.9}{$d$}};
	\end{tikzpicture}}}
\right\}\\
&=
\sum_{\substack{n\geqslant 3\\ k\geqslant m}}
{\textstyle \frac{1}{(n-1)!(k-m-1)!}}\
\vcenter{\hbox{\begin{tikzpicture}[scale=0.9]
	\draw[thick] (0:0) -- (0:2.7);
	\draw[thick] (45:0) -- (45:0.7);
	\draw[thick] (90:0) -- (90:0.7) ;
	\draw[thick] (135:0) -- (135:0.7);
	\draw[thick] (180:0) -- (180:0.7);
	\draw[thick] (225:0) -- (225:0.7);
	\draw[thick] (270:0) -- (270:0.7);
	\draw[thick] (315:0) -- (315:0.7);
	\draw[thick] (2,0) -- (2.5,0.5);
	\draw[thick] (2,0) -- (2,0.7);
	\draw[thick] (2,0) -- (1.5,0.5);
	\draw[thick] (2,0) -- (1.3,0);
	\draw[thick] (2,0) -- (1.5,-0.5);
	\draw[thick] (2,0) -- (2,-0.7);
	\draw[thick] (2,0) -- (2.5,-0.5);
	\draw[fill=white, thick] (0,0) circle [radius=11pt];
	\draw[fill=white, thick] (2,0) circle [radius=11pt];
	\node at (0,0) {\scalebox{1}{$\mathbb{1}_{0,n}$}};
	\node at (2,0) {\scalebox{1}{$\Lambda_k$}};
	\node at (45:0.9) {\scalebox{0.9}{$d$}};
	\node at (315:0.9) {\scalebox{0.9}{$d$}};
	\node at (90:0.9) {\scalebox{0.9}{$d$}};
	\node at (270:0.9) {\scalebox{0.9}{$d$}};
	\node at (1.33,0.67) {\scalebox{0.9}{$c_m$}};
	\node at (2.67,0.67) {\scalebox{0.9}{$c_1$}};
	\node at (2.9,0) {\scalebox{0.9}{$d$}};
	\node at (1.36,-0.64) {\scalebox{0.9}{$d$}};
\end{tikzpicture}}}
+\sum_{{\substack{
n\geqslant 3\\
i=1, \ldots, m\\
k\geqslant m}}}
{\textstyle \frac{1}{(n-2)!(k-m)!}}\
\vcenter{\hbox{\begin{tikzpicture}[scale=1]
	\draw[thick] (0:0) -- (0:2.7);
	\draw[thick] (45:0) -- (45:0.7);
	\draw[thick] (90:0) -- (90:0.7) ;
	\draw[thick] (135:0) -- (135:0.7);
	\draw[thick] (180:0) -- (180:0.7);
	\draw[thick] (225:0) -- (225:0.7);
	\draw[thick] (270:0) -- (270:0.7);
	\draw[thick] (315:0) -- (315:0.7);
	\draw[thick] (2,0) -- (2.5,0.5);
	\draw[thick] (2,0) -- (2,0.7);
	\draw[thick] (2,0) -- (1.5,0.5);
	\draw[thick] (2,0) -- (1.3,0);
	\draw[thick] (2,0) -- (1.5,-0.5);
	\draw[thick] (2,0) -- (2,-0.7);
	\draw[thick] (2,0) -- (2.5,-0.5);
	\draw[fill=white, thick] (0,0) circle [radius=11pt];
	\draw[fill=white, thick] (2,0) circle [radius=11pt];
	\node at (0,0) {\scalebox{1}{$\mathbb{1}_{0,n}$}};
	\node at (2,0) {\scalebox{1}{$\Lambda_k$}};
	\node at (45:0.9) {\scalebox{0.9}{$c_i$}};
	\node at (315:0.9) {\scalebox{0.9}{$d$}};
	\node at (90:0.9) {\scalebox{0.9}{$d$}};
	\node at (270:0.9) {\scalebox{0.9}{$d$}};
	\node at (135:0.9) {\scalebox{0.9}{$d$}};
	\node at (1.33,0.67) {\scalebox{0.9}{$c_{i-1}$}};
	\node at (2,-0.9) {\scalebox{0.9}{$c_m$}};
	\node at (2,0.9) {\scalebox{0.9}{$c_1$}};
	\node at (2.67,0.67) {\scalebox{0.9}{$d$}};
	\node at (2.67,-0.67) {\scalebox{0.9}{$d$}};
	\node at (1.36,-0.64) {\scalebox{0.9}{$c_{i+1}$}};
\end{tikzpicture}}}\ .
\end{align*}
The minimality of $\Lambda$ implies that the restriction of $\Lambda_k$ to a
boundary
divisor in $\overline{\mathcal{M}}_{h,k}$ is non-trivial if and only if the image of this divisor under the projection $p_k$ is the whole space  $\overline{\mathcal{M}}_{h,m}$. It is the case only if the generic point of the divisor is represented by a two-component curve with one component of genus $0$, with at most one point labeled by $c_i$, for $i=1, \dots, m$, on that component (cf.~\cite[Section~1.5]{PZ18}).
This immediately implies that
 $\d_1(\lambda)=0$ and that
\[
\d_2(\lambda)=\sum_{\substack{n\geqslant 3\\ k\geqslant m}}
{\textstyle \frac{1}{(n-1)!(k-m-1)!}}\
\vcenter{\hbox{\begin{tikzpicture}[scale=0.9]
	\draw[thick] (0:0) -- (0:2.7);
	\draw[thick] (45:0) -- (45:0.7);
	\draw[thick] (90:0) -- (90:0.7) ;
	\draw[thick] (135:0) -- (135:0.7);
	\draw[thick] (180:0) -- (180:0.7);
	\draw[thick] (225:0) -- (225:0.7);
	\draw[thick] (270:0) -- (270:0.7);
	\draw[thick] (315:0) -- (315:0.7);
	\draw[thick] (2,0) -- (2.5,0.5);
	\draw[thick] (2,0) -- (2,0.7);
	\draw[thick] (2,0) -- (1.5,0.5);
	\draw[thick] (2,0) -- (1.3,0);
	\draw[thick] (2,0) -- (1.5,-0.5);
	\draw[thick] (2,0) -- (2,-0.7);
	\draw[thick] (2,0) -- (2.5,-0.5);
	\draw[fill=white, thick] (0,0) circle [radius=11pt];
	\draw[fill=white, thick] (2,0) circle [radius=11pt];
	\node at (0,0) {\scalebox{1}{$\mathbb{1}_{0,n}$}};
	\node at (2,0) {\scalebox{1}{$\Lambda_k$}};
	\node at (45:0.9) {\scalebox{0.9}{$d$}};
	\node at (315:0.9) {\scalebox{0.9}{$d$}};
	\node at (90:0.9) {\scalebox{0.9}{$d$}};
	\node at (270:0.9) {\scalebox{0.9}{$d$}};
	\node at (1.33,0.67) {\scalebox{0.9}{$c_m$}};
	\node at (2.67,0.67) {\scalebox{0.9}{$c_1$}};
	\node at (2.9,0) {\scalebox{0.9}{$d$}};
	\node at (1.36,-0.64) {\scalebox{0.9}{$d$}};
\end{tikzpicture}}}
+\sum_{{\substack{
n\geqslant 3\\
i=1, \ldots, m\\
k\geqslant m}}}
{\textstyle \frac{1}{(n-2)!(k-m)!}}\
\vcenter{\hbox{\begin{tikzpicture}[scale=1]
	\draw[thick] (0:0) -- (0:2.7);
	\draw[thick] (45:0) -- (45:0.7);
	\draw[thick] (90:0) -- (90:0.7) ;
	\draw[thick] (135:0) -- (135:0.7);
	\draw[thick] (180:0) -- (180:0.7);
	\draw[thick] (225:0) -- (225:0.7);
	\draw[thick] (270:0) -- (270:0.7);
	\draw[thick] (315:0) -- (315:0.7);
	\draw[thick] (2,0) -- (2.5,0.5);
	\draw[thick] (2,0) -- (2,0.7);
	\draw[thick] (2,0) -- (1.5,0.5);
	\draw[thick] (2,0) -- (1.3,0);
	\draw[thick] (2,0) -- (1.5,-0.5);
	\draw[thick] (2,0) -- (2,-0.7);
	\draw[thick] (2,0) -- (2.5,-0.5);
	\draw[fill=white, thick] (0,0) circle [radius=11pt];
	\draw[fill=white, thick] (2,0) circle [radius=11pt];
	\node at (0,0) {\scalebox{1}{$\mathbb{1}_{0,n}$}};
	\node at (2,0) {\scalebox{1}{$\Lambda_k$}};
	\node at (45:0.9) {\scalebox{0.9}{$c_i$}};
	\node at (315:0.9) {\scalebox{0.9}{$d$}};
	\node at (90:0.9) {\scalebox{0.9}{$d$}};
	\node at (270:0.9) {\scalebox{0.9}{$d$}};
	\node at (135:0.9) {\scalebox{0.9}{$d$}};
	\node at (1.33,0.67) {\scalebox{0.9}{$c_{i-1}$}};
	\node at (2,-0.9) {\scalebox{0.9}{$c_m$}};
	\node at (2,0.9) {\scalebox{0.9}{$c_1$}};
	\node at (2.67,0.67) {\scalebox{0.9}{$d$}};
	\node at (2.67,-0.67) {\scalebox{0.9}{$d$}};
	\node at (1.36,-0.64) {\scalebox{0.9}{$c_{i+1}$}};
\end{tikzpicture}}}\ .
\]
\end{proof}

This result means that $\lambda$ is an infinitesimal \emph{and} global deformation of the
CohFT $\alpha$. The isotropic property of $c_1, \ldots, c_m, d$ is responsible for
vanishing of all obstructions
since it implies $\Delta(\lambda)=0$ and $\{\lambda, \lambda\}=0$.

\medskip

The Pandharipande--Zvonkine construction can be generalized as follows.
Let $A=\oplus_{i=0}^D A_i=H^\bullet(X,\K)$ be the $\ZZ$-graded cohomology algebra of a compact
manifold $X$ of even dimension $D>0$. It is a Frobenius algebra with respect to the Poincar\'e
pairing $\langle\ ,\, \rangle$ which has degree $-D$.

Let $r>0$ be the minimal possible positive degree such that $A_r\not=0$. Let
$a\in A_0$ be the unit of $A$, let $\K\{ c_1, \ldots, c_m\}$ be a subspace of $A_r$, isotropic
with respect to the Poincar\'e pairing, and let $d\in A_D,  b_1, \ldots, b_m\in A_{D-r}$
be a  collection of classes dual
to $\{a, c_1,\ldots, c_m \}$, i.e.\
$\langle a,d \rangle=\langle b_i,c_i \rangle=1$ for $i=1,\ldots, m$.
Finally, let  $\K\{ e_1, \ldots, e_s\}$ be the orthogonal complement to the subspace
$\K\{a, b_1, c_1, \ldots, b_m,c_m,d\}$.

\medskip

We consider the even degree element $\alpha\in \widehat{\mathbb{G}}_\od\big(H^\bullet\big(\oM\big)\big)(A)$ defined by
\begin{align*}
\alpha\coloneqq &\sum_{n\geqslant 3}
\frac{1}{(n-1)!}
	\vcenter{\hbox{\begin{tikzpicture}[scale=1]
	\draw[thick] (0:0) -- (0:0.7);
	\draw[thick] (45:0) -- (45:0.7);
	\draw[thick] (90:0) -- (90:0.7) ;
	\draw[thick] (135:0) -- (135:0.7);
	\draw[thick] (180:0) -- (180:0.7);
	\draw[thick] (225:0) -- (225:0.7);
	\draw[thick] (270:0) -- (270:0.7);
	\draw[thick] (315:0) -- (315:0.7);
	\draw[fill=white, thick] (0,0) circle [radius=11pt];
	\node at (0,0) {\scalebox{1}{$\mathbb{1}_{0,n}$}};
	\node at (0:0.9) {\scalebox{0.9}{$d$}};
	\node at (45:0.9) {\scalebox{0.9}{$d$}};
	\node at (90:0.9) {\scalebox{0.9}{$a$}};
	\node at (135:0.9) {\scalebox{0.9}{$d$}};
	\node at (180:0.9) {\scalebox{0.9}{$d$}};
	\end{tikzpicture}}}
+\sum_{\substack{n\geqslant 3 \\ i=1, \ldots, m}}
\frac{1}{(n-2)!}
\vcenter{\hbox{\begin{tikzpicture}[scale=1]
	\draw[thick] (0:0) -- (0:0.7);
	\draw[thick] (45:0) -- (45:0.7);
	\draw[thick] (90:0) -- (90:0.7) ;
	\draw[thick] (135:0) -- (135:0.7);
	\draw[thick] (180:0) -- (180:0.7);
	\draw[thick] (225:0) -- (225:0.7);
	\draw[thick] (270:0) -- (270:0.7);
	\draw[thick] (315:0) -- (315:0.7);
	\draw[fill=white, thick] (0,0) circle [radius=11pt];
	\node at (0,0) {\scalebox{1}{$\mathbb{1}_{0,n}$}};
	\node at (0:0.9) {\scalebox{0.9}{$d$}};
	\node at (45:0.9) {\scalebox{0.9}{$b_i$}};
	\node at (90:0.9) {\scalebox{0.9}{$c_i$}};
	\node at (135:0.9) {\scalebox{0.9}{$d$}};
	\node at (180:0.9) {\scalebox{0.9}{$d$}};
	\end{tikzpicture}}}
+\mathrm{sdim}(A)
                  \sum_{n\geqslant 1} \frac{1}{n!}
\vcenter{\hbox{\begin{tikzpicture}[scale=1]
	\draw[thick] (0:0) -- (0:0.7);
	\draw[thick] (45:0) -- (45:0.7);
	\draw[thick] (90:0) -- (90:0.7) ;
	\draw[thick] (135:0) -- (135:0.7);
	\draw[thick] (180:0) -- (180:0.7);
	\draw[thick] (225:0) -- (225:0.7);
	\draw[thick] (270:0) -- (270:0.7);
	\draw[thick] (315:0) -- (315:0.7);
	\draw[fill=white, thick] (0,0) circle [radius=11pt];
	\node at (0,0) {\scalebox{1}{$\mathbb{1}_{1,n}$}};
	\node at (0:0.9) {\scalebox{0.9}{$d$}};
	\node at (45:0.9) {\scalebox{0.9}{$d$}};
	\node at (90:0.9) {\scalebox{0.9}{$d$}};
	\node at (135:0.9) {\scalebox{0.9}{$d$}};
	\node at (180:0.9) {\scalebox{0.9}{$d$}};
	\end{tikzpicture}}} \\
&+
\sum_{\substack {n\geqslant 3 ,\  t \geqslant 2 \\ n\geqslant t}}
\frac{\langle f_t^*\cdots f_1^*,d\rangle }{(n-t)!|\mathrm{Aut}(f_1,\dots,f_t)|}
\vcenter{\hbox{\begin{tikzpicture}[scale=1]
		\draw[thick] (0:0) -- (0:0.7);
		\draw[thick] (45:0) -- (45:0.7);
		\draw[thick] (90:0) -- (90:0.7) ;
		\draw[thick] (135:0) -- (135:0.7);
		\draw[thick] (180:0) -- (180:0.7);
		\draw[thick] (225:0) -- (225:0.7);
		\draw[thick] (270:0) -- (270:0.7);
		\draw[thick] (315:0) -- (315:0.7);
		\draw[fill=white, thick] (0,0) circle [radius=11pt];
		\node at (0,0) {\scalebox{1}{$\mathbb{1}_{0,n}$}};
		\node at (0:0.9) {\scalebox{0.9}{$d$}};
		\node at (45:0.9) {\scalebox{0.9}{$f_t$}};
		\node at (90:0.9) {\scalebox{0.9}{$\ldots$}};
		\node at (135:0.9) {\scalebox{0.9}{$f_1$}};
		\node at (180:0.9) {\scalebox{0.9}{$d$}};
		\end{tikzpicture}}}\ ,
\end{align*}
where
$\mathrm{sdim}(A)=\dim A_{\mathrm{even}}- \dim A_{\mathrm{odd}}$ is the superdimension of $A$
(i.e.\ the Euler characteristic of $X$)
and where the last sum runs over all possible tuples of basis vectors $f_1,\dots,f_t\in
\{b_1, \ldots, b_m, e_1,\ldots, e_s\}$ such that the sum of the
original cohomological gradings of $f_1,\dots,f_t$ is equal to $(t-1)D$. By
$|\mathrm{Aut}(f_1,\dots,f_t)|$ we denote the order of the automorphism group of the tuple
$f_1,\dots,f_t$.

\begin{lemma}
This element $\alpha$ is a solution to the master equation \eqref{eq:MaEq}. In other words, it is a topological field theory and thus a cohomological field theory.
\end{lemma}

\begin{proof}
Notice that, for  degree reasons coming from
the natural cohomological degree of the algebra $A$,
the left-hand side of the master equation can be non-trivial only in genera $0$ and $1$.
The proof that thus defined $\alpha$ is a solution to the master equation repeats \emph{mutatis mutandis} the proof of Lemma~\ref{lem:PZalpha}.
\end{proof}

\begin{theorem}
Let $A$ and $r$ be as above and
let $\Lambda\in H^\bullet\big(\overline{\mathcal{M}}_{h,m}\big)$ be a minimal class
whose degree has the same parity as $rm$.
Consider
the even degree element $\lambda\in \widehat{\mathbb{G}}_\od\big(H^\bullet\big(\oM\big)\big)(A)$ defined by Equation~\eqref{eq:alphapluslambda}. Then $\alpha+\lambda$ is a solution to the master equation \eqref{eq:MaEq}, which means that $\alpha+\lambda$ defines a CohFT.
\end{theorem}

\begin{proof}
The proof of this statement repeats \emph{mutatis mutandis} the proof of Theorem~\ref{thm:PZalphalambda}.
\end{proof}

\subsection{Quantum homotopy CohFT and the Buryak--Rossi functor}\label{subsec:Quantum}
In order to study formal deformations, one works with $\K[[\hbar]]$-extensions of the shifted Lie algebra whose Maurer--Cartan elements encode morphisms of modular operads. There is however another way to extend the present theory from the ground field $\K$ to the ground ring $\K[[\hbar]]$, which does not amount to extending the structure freely to $\K[[\hbar]]$-linear combinations. This gives rise to a new notion of \emph{quantum homotopy CohFT}.
Remarkably enough, quantum CohFTs
have already appeared in literature, in the work of Buryak and Rossi~\cite{BuryakRossi2016}
on quantization of Dubrovin--Zhang integrable hierarchies associated to CohFTs.
Moreover, the key
step in their quantization procedure after translating it
to our language appears to be a functor from the Deligne groupoid of homotopy CohFTs to
the Deligne groupoid of quantum homotopy CohFT.

\begin{definition}[{$\K[[\hbar]]$-extensions}]
The \emph{$\K[[\hbar]]$-extension} of a modular operad
\[\P=\big(\{\P_g(n)\}_{g,n}, d_\P,\allowbreak \xi_{ij}, \circ^j_i\big)\] defined over $\K$ is the modular operad
\[\P[[\hbar]]\coloneq\big(\{\P_g(n)[[\hbar]]\}_{g,n}, d_\P,
\xi_{ij},
\circ^j_i\big)\]
defined over $\K[[\hbar]]$ by $\K[[\hbar]]$-linear extensions of its structure maps.
In a similar way, we consider the $\K[[\hbar]]$-extensions of shifted modular operads, (shifted) modular cooperads, and complete shifted $\Delta$-Lie algebras. In this latter case, we get
\[(\g[[\hbar]], \F[[\hbar]], \d,\Delta, \{\ , \,\})\ .\]
\end{definition}

One can also consider the following ``quantum versions'' which are not free $\Khbar$-extensions.

\begin{definition}[Quantum versions]
The \emph{quantum (shifted) modular operad} associated to a (shifted) modular operad $\P=\big(
\{\P_g(n)\}_{g,n}, d_\P, \xi_{ij}, \circ^j_i\big)$ defined over $\K$ is the (shifted) modular operad
\[\P^\hbar\coloneq\big(\{\P_g(n)[[\hbar]]\}_{g,n}, d_\P, \hbar\,\xi_{ij}, \circ^j_i \big)\]
defined over $\K[[\hbar]]$ by $\K[[\hbar]]$-linear extensions of its structure maps.
In the same way, the \emph{quantum shifted $\Delta$-Lie algebra} associated to a shifted $\Delta$-Lie algebra
$(\g, \d,\Delta, \{\ , \,\})$
defined over $\K$ is the shifted $\Delta$-Lie algebra
\[\g^\hbar\coloneq(\g[[\hbar]], \d,\hbar\Delta, \{\ , \,\})\]
defined over $\K[[\hbar]]$ by $\K[[\hbar]]$-linear extensions of its structure maps.
\end{definition}

This procedure defines two functors from shifted modular operads (resp.\ complete shifted $\Delta$-Lie algebras) over $\K$ to shifted modular operads
(resp.\ complete shifted $\Delta$-Lie algebras)
over $\K[[\hbar]]$. These functors commute with the totalisation functor of \cref{lem:ModtoDeltaLie}:
\[
\vcenter{\hbox{
\begin{tikzcd}
\mathsf{shifted} \ \mathsf{modular} \ \mathsf{operads} \ \mathsf{over} \ {\K}
\arrow[d, "(-)^\hbar"]
\arrow[r, "\widehat{(-)}"]
 &
 \mathsf{complete} \ \mathsf{shifted} \ \Delta\textsf{-}\mathsf{Lie}\ \mathsf{algebras}  \ \mathsf{over} \ {\K}
\arrow[d, "(-)^\hbar"]
 \\
\mathsf{shifted} \ \mathsf{modular} \ \mathsf{operads}  \ \mathsf{over} \ {\Khbar}
\arrow[r, "\widehat{(-)}"]
&
 \mathsf{complete} \ \mathsf{shifted} \ \Delta\textsf{-}\mathsf{Lie}\ \mathsf{algebras}  \ \mathsf{over} \ {\Khbar}
\end{tikzcd}
}}
\]
since $\left(\prod_{(g,n) \in \NNs} \P_g(n)^{\Sy_n}\right)[[\hbar]]\cong \prod_{(g,n) \in \NNs} \left(\P_g(n)^{\Sy_n}[[\hbar]]\right)$~.
Their composition
produces a complete shifted $\Delta$-Lie algebra over $\Khbar$
which we call
the \emph{quantum totalisation} of a shifted modular operad and denote by~$\widehat{\P}^\hbar$~.

\begin{definition}[Quantum master equation]
The \emph{quantum master equation}
is the equation
\begin{equation}\label{eq:QME}
\d \alpha + \hbar\Delta \alpha + {\textstyle \frac12} \{\alpha, \alpha\}=0
 \end{equation}
in a shifted $\Delta$-Lie algebra over $\Khbar$ .
We only consider its degree $0$ (or even degree) solutions.
\end{definition}

In the case of a quantum shifted $\Delta$-Lie algebra $\gth$, the quantum master equation is nothing but the Maurer--Cartan equation of the associated shifted Lie algebra
$ \big(\gh, \d+\hbar\Delta, \{\ , \,\}\big)$
over $\Khbar$, that we still denote by $\gth$ by a slight abuse of notation. The associated set of solutions  is denoted by $\MC\big(\gth\big)$~.

\medskip

The quantum shifted modular operad associated to a convolution shifted modular operad is given by the following lemma.

\begin{lemma}\label{lem:QuantumConv}
There is an isomorphism of shifted modular operads over $\Khbar$:
\[
\Hom_{\K} \left(\C, \P\right)^\hbar\cong \Hom_{\Khbar} \left(\C[[\hbar]], \P^\hbar\right)\ ,
\]
natural in shifted modular cooperads $\C$ and modular operads $\P$ over $\K$~.
\end{lemma}

\begin{proof}
The proof relies  on the natural isomorphisms
\[
\Hom_{\Khbar} \left(\C_g(n)[[\hbar]], \P_g(n)[[\hbar]]\right)\cong
\Hom_{\K} \left(\C_g(n), \P_g(n)[[\hbar]]\right)
\cong
\Hom_{\K} \left(\C_g(n), \P_g(n)\right)[[\hbar]]\ ,
\]
which  preserves the respective structure maps.
\end{proof}

As in the classical case given in \cref{prop:MorphTw},
the set of Maurer--Cartan elements of the quantum convolution algebra associated to ${\Hom}_{\K}(\C, \P)$
 encodes morphisms from a suitable cobar construction.

\begin{proposition}\label{prop:QuantMorphTw}
There is natural bijection:
\[
\MC\left(\widehat{\Hom}_{\K}(\C, \P)^\hbar
\right)
\cong
\Hom_{\mathsf{mod}\,  \mathsf{op}_{\Khbar}}\left(\Omega \C[[\hbar]], \P^\hbar\right)
\ .
\]
\end{proposition}

\begin{proof}
This is a direct corollary of \cref{prop:MorphTw} and \cref{lem:QuantumConv}.
\end{proof}

\begin{definition}[Quantum homotopy CohFT]\label{def:qhCohFT}
A \emph{quantum homotopy CohFT}, or $\text{\it q-CohFT}_\infty$ for short, is a         ($\ZZ$-graded)
dg symmetric vector space
$(A, d_A, \langle\ ,\, \rangle)$
equipped with an even
morphism of $\Z2$-graded modular operads over $\Khbar$
 \[
 \Omega \B  H_\bullet\big(\oM\big)[[\hbar]]\to \EEnd_A^\hbar\ .
 \]
 \end{definition}

We are going to show that such elaborate structures appear functorially from CohFTs, see below, and that they admit a non trivial action of the prounipotent Grothendieck--Teichm\"uller group, see the next section.
Since $A[[\hbar]]^{\otimes_{\Khbar} n}\cong \big(A^{\otimes_{\K} n}\big)[[\hbar]]$~, a quantum homotopy CohFT amounts to a certain $\Khbar$-linear structure on $A[[\hbar]]$~.

\begin{proposition}\label{prop:QCohFTMC}
The data of a quantum homotopy CohFT is equivalent to the data of a solution to the quantum master equation of the
quantum  shifted $\Delta$-Lie algebra
\[
\g_A^\hbar\cong \left(
\widehat{\mathbb{G}}_\od\big(\Hbul\big(\oM\big)\big)(A)[[\hbar]], \mathrm{d}, \hbar\, \Delta, \{\ ,\,\}
\right)\ .
\]
\end{proposition}

\begin{proof}
It follows from the above definition, \cref{prop:QuantMorphTw}, and \cref{lem:DefComp}.
\end{proof}

\cref{prop:QCohFTMC} allows us to think of the datum of a quantum homotopy CohFT as a series of even degree
\[\alpha^\hbar=\alpha^{(0)}+ \alpha^{(1)}\hbar+\cdots+\alpha^{(k)} \hbar^k+\cdots\ ,\]
with $\alpha^{(k)}\in \widehat{\mathbb{G}}_\od\big(\Hbul\big(\oM\big)\big)(A)$, satisfying
\[ \d_A\alpha^{(k)} - \d_1 \alpha^{(k)} - \d_2 \alpha^{(k)} +\Delta \alpha^{(k-1)} +
{\textstyle \frac12}
\sum_{l+m=k}
\left\{\alpha^{(l)}, \alpha^{(m)}\right\}=0\ ,\]
for any $k\geqslant 0$, under the convention $\alpha^{(-1)}=0$~.
In a similar way, one defines a \emph{quantum CohFT} as an  even graded morphism
 $H_\bullet\big(\oM\big)[[\hbar]]\to \EEnd_A^\hbar$
of modular operads over $\Khbar$~. In terms of the description as a series, a quantum
CohFT is a quantum homotopy CohFT $\alpha^\hbar$ where each $\alpha^{(k)}$ is a
combination of one-vertex graphs.

\medskip

Let $\lambda_0=1,\lambda_1,\dots,\lambda_{g}\in \Hbul\big(\oM_{g,n}\big)$ be the Chern
classes of the Hodge bundle $\mathbb{E}_g \to \oM_{g,n}$~.
We denote by
$\lambda^\hbar$ the element of $H^{\bullet}\big(\oM_{g,n}\big)[\hbar]$ defined by
\[\lambda^\hbar\coloneq \lambda_{g}+ \lambda_{g-1}\hbar+\cdots+ \lambda_{1}\hbar^{g-1}+\lambda_{0}\hbar^{g}\ .\]
(we intentionally omit $g$ and $n$ in the notation for $\lambda^\hbar$, since it will be always clear on which space this class is taken). The main property of the class $\lambda^\hbar$ can be formulated as follows.
\begin{lemma}[\cite{Mumford1983}] \label{lem:mumford}
	In the modular cooperad
	$\Hbul\big(\overline{\mathcal{M}}\big)[\hbar]$,
	the cohomology classes $\lambda^\hbar$ satisfy  the following identities:
\begin{itemize}
	\item[$\diamond$] the image of $\lambda^\hbar$ under
	any expansion map \[\chi_{ij} : \Hbul\big(\overline{\mathcal{M}}_{g+1,n-2}\big)[\hbar] \to \Hbul\big(\overline{\mathcal{M}}_{g,n}\big)[\hbar]\] is equal to~$\hbar\lambda^\hbar$~;
	\item[$\diamond$]  the image of $\lambda^\hbar$ under
	any
	partial decomposition map
	\[\delta_{i}^j : \Hbul\big(\overline{\mathcal{M}}_{g+g',n+n'-2}\big)[\hbar] \to
	\Hbul\big(\overline{\mathcal{M}}_{g,n}\big)[\hbar]
	\otimes \Hbul\big(\overline{\mathcal{M}}_{g',n'}\big)[\hbar]\] is equal to $\lambda^\hbar\otimes \lambda^\hbar$~.
\end{itemize}
\end{lemma}

To any element $\alpha\in \g_A$, we associate  the element $\BR(\alpha)\in \g_A^\hbar$ defined by
the following rule: we multiply the cohomology classes labeling each vertex by the corresponding $\lambda^\hbar$ and the whole graph by $\hbar^{b_1}$, where $b_1$ is the first Betti number of the graph. For instance, we have
  \[\BR\colon
  \vcenter
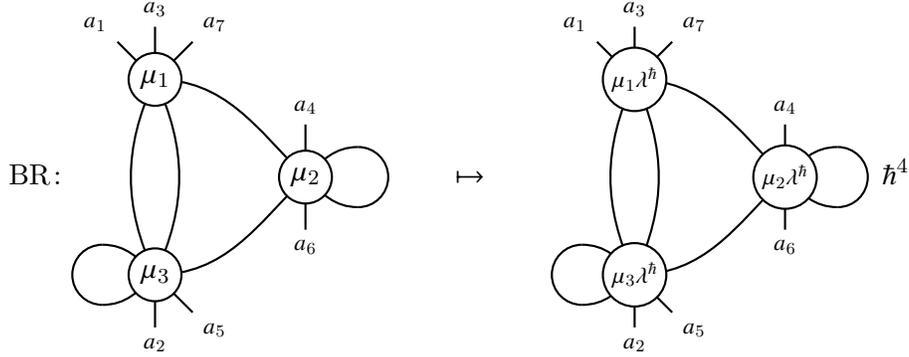
{\hbox{\begin{tikzpicture}[scale=1]

	\draw[thick] (-1, 1.3) to [out=0,in=135] (1,0);
	\draw[thick] (-1, -1.3) to [out=0,in=225] (1,0);
	\draw[thick] (-1, 1.3) to [out=245,in=115] (-1,-1.3);
	\draw[thick] (-1, 1.3) to [out=295,in=65] (-1,-1.3);

	\draw[thick]  (1.7,-0.4) arc [radius=0.4, start angle=270, end angle= 450];
	\draw[thick] (1, 0) to [out=315,in=180] (1.7,-0.4);
	\draw[thick] (1, 0) to [out=45,in=180] (1.7,0.4);

	\draw[thick]  (-1.7,-1.7) arc [radius=0.4, start angle=270, end angle= 90];
	\draw[thick] (-1, -1.3) to [out=225,in=0] (-1.7,-1.7);
	\draw[thick] (-1, -1.3) to [out=135,in=0] (-1.7,-0.9);

	\draw[thick] (-1, 1.3) -- (-1, 2) node[above]  {\scalebox{0.8}{$a_3$}};
	\draw[thick] (-1, 1.3) -- (-1.5, 1.8) node[above left]  {\scalebox{0.8}{$a_1$}};
	\draw[thick] (-1, 1.3) -- (-0.5, 1.8) node[above right]  {\scalebox{0.8}{$a_7$}};
	\draw[thick] (-1, -1.3) -- (-1, -2) node[below]  {\scalebox{0.8}{$a_2$}};
	\draw[thick] (-1, -1.3) -- (-0.5, -1.8) node[below right]  {\scalebox{0.8}{$a_5$}};
	\draw[thick] (1, 0) -- (1, 0.7) node[above]  {\scalebox{0.8}{$a_4$}};
	\draw[thick] (1, 0) -- (1, -0.7) node[below]  {\scalebox{0.8}{$a_6$}};

	\draw[fill=white, thick] (1,0) circle [radius=10pt];
	\draw[fill=white, thick] (-1,1.3) circle [radius=10pt];
	\draw[fill=white, thick] (-1,-1.3) circle [radius=10pt];

	\node at (1,0) {\scalebox{1}{$\mu_2$}};
	\node at (-1,1.3) {\scalebox{1}{$\mu_1$}};
	\node at (-1,-1.3) {\scalebox{1}{$\mu_3$}};
	\end{tikzpicture}}}
\qquad \mapsto \qquad
  \vcenter{\hbox{\begin{tikzpicture}[scale=1]

		\draw[thick] (-1, 1.3) to [out=0,in=135] (1,0);
		\draw[thick] (-1, -1.3) to [out=0,in=225] (1,0);
		\draw[thick] (-1, 1.3) to [out=245,in=115] (-1,-1.3);
		\draw[thick] (-1, 1.3) to [out=295,in=65] (-1,-1.3);

		\draw[thick]  (1.7,-0.4) arc [radius=0.4, start angle=270, end angle= 450];
		\draw[thick] (1, 0) to [out=315,in=180] (1.7,-0.4);
		\draw[thick] (1, 0) to [out=45,in=180] (1.7,0.4);

		\draw[thick]  (-1.7,-1.7) arc [radius=0.4, start angle=270, end angle= 90];
		\draw[thick] (-1, -1.3) to [out=225,in=0] (-1.7,-1.7);
		\draw[thick] (-1, -1.3) to [out=135,in=0] (-1.7,-0.9);

		\draw[thick] (-1, 1.3) -- (-1, 2) node[above]  {\scalebox{0.8}{$a_3$}};
		\draw[thick] (-1, 1.3) -- (-1.5, 1.8) node[above left]  {\scalebox{0.8}{$a_1$}};
		\draw[thick] (-1, 1.3) -- (-0.5, 1.8) node[above right]  {\scalebox{0.8}{$a_7$}};
		\draw[thick] (-1, -1.3) -- (-1, -2) node[below]  {\scalebox{0.8}{$a_2$}};
		\draw[thick] (-1, -1.3) -- (-0.5, -1.8) node[below right]  {\scalebox{0.8}{$a_5$}};
		\draw[thick] (1, 0) -- (1, 0.7) node[above]  {\scalebox{0.8}{$a_4$}};
		\draw[thick] (1, 0) -- (1, -0.7) node[below]  {\scalebox{0.8}{$a_6$}};

		\draw[fill=white, thick] (1,0) circle [radius=12pt];
		\draw[fill=white, thick] (-1,1.3) circle [radius=12pt];
		\draw[fill=white, thick] (-1,-1.3) circle [radius=12pt];

		\node at (1,0) {\scalebox{0.8}{$\mu_2\lambda^\hbar$}};
		\node at (-1,1.3) {\scalebox{0.8}{$\mu_1\lambda^\hbar$}};
		\node at (-1,-1.3) {\scalebox{0.8}{$\mu_3\lambda^\hbar$}};
		\end{tikzpicture}}}\ \, \hbar^4
\]

\begin{lemma}\label{lem:BRmorph}
The above assignment
defines a morphism of complete $\Z2$-graded shifted dg Lie algebras:
\[\BR\ : \ \left(\g_A[[\hbar]], \d+\Delta, \{\ , \,\}\right) \to \g_A^\hbar=
\left(\g_A[[\hbar]], \d+\hbar\Delta, \{\ , \,\}\right)\ .\]
\end{lemma}

\begin{proof}
Since the Chern classes have even degrees,  the map $\BR$ is a linear map of even degree. First, one can
see that the map $\BR$ commutes with the differential $\d_A$ and the shifted Lie bracket $\{\ , \,\}$ since they change neither the labels of the vertices nor  the Betti number of graphs.
Second, we note that $\BR \Delta = \hbar \Delta \BR$~. Indeed, the operator $\Delta$ does not change the labels of the vertices and increases the first Betti number of the graph by $1$, so we need an extra factor of $\hbar$.

\medskip

Since $\Hbul\big(\overline{\mathcal{M}}\big)[\hbar]$ is a modular operad arising from a
topological modular operad, it is a Hopf modular operad, that is, its expansion maps and its
partial decomposition maps commute with the
multiplication
of cohomology classes.
The differential $\d_2$ does not change the Betti number of the graph as it amounts to applying the partial decomposition maps $\delta_{i}^j$~. By the above remark and Lemma~\ref{lem:mumford}, it commutes with the map $\BR$~.
The differential $\d_1$ amounts to applying the expansion maps  $\chi_{ij}$, so it
increases the Betti number of the graph by $1$~.
By Lemma~\ref{lem:mumford}, it generates the factor $\hbar\lambda^\hbar$ on the corresponding vertex, and we gain an extra power of $\hbar$ in front of the graph. Hence, $\mathrm{BR}$ commutes with $\d_2$ as well.
\end{proof}

Let us recall
(see e.g.~\cite[Chapter~2]{DotsenkoShadrinVallette18})
that any complete $\Z2$-graded shifted dg Lie algebra $(\g, \F, \d, \{\ , \,\})$
admits a \emph{gauge group}
$(\F_1\g_{\od}, \BCH, 0)$,
i.e.\ its operation is defined by the Baker--Campbell--Hausdorff formula and $0$ is the identity element.
The gauge group
acts on the set of  Maurer--Cartan elements by the formula
\[\alpha+\sum_{k\geqslant 1} \frac{1}{k!} \left(
\ad_\xi^k(\alpha)+\ad_\xi^{k-1}(\d \xi)\right)\ , \]
for $\alpha \in \MC(\g)$ and $\xi\in \F_1\g_\od$.
The groupoid whose objects are Maurer--Cartan elements and whose morphisms are   the actions of the gauge group is called the \emph{Deligne groupoid}.

\begin{theorem}
The map $\BR$ induces a functor from the Deligne groupoid of homotopy CohFTs over $\K$ to
the Deligne groupoid of quantum homotopy CohFTs over $\Khbar$~, which we call the \emph{Buryak--Rossi functor}.
\end{theorem}

\begin{proof}
This an immediate consequence of the fact that any morphism of complete (shifted) dg Lie algebras induces a functor between the two associated Deligne groupoids.
We consider first the natural inclusion $\left(\g_A, \d+\Delta, \{\ , \,\}\right) \hookrightarrow \left(\g_A[[\hbar]], \d+\Delta, \{\ , \,\}\right)$ of complete shifted dg Lie algebras over~$\K$~, and then the morphism $\BR$ of complete shifted dg Lie algebras over $\K[[\hbar]]$ by \cref{lem:BRmorph}. The composite of the two associated functors between Deligne groupoids produces the functor mentioned in the statement.
\end{proof}

\begin{remark} The Buryak--Rossi functor is used in~\cite{BuryakRossi2016} in the following special
case. Let
$\alpha\coloneq \{\alpha_{g,n}\}_{(g,n)\in \NN^2}$ be a (strict) CohFT, that is a Maurer-Cartan element in $\g_A$ supported on
one-vertex graphs without tadpoles. They consider a system of classes
$\left\{\alpha_{g,n}\lambda^\hbar\right\}_{(g,n)\in \NN^2}$~.
In our language, they get a Maurer-Cartan element of $\g_A^\hbar$, that is a  quantum
CohFT.
As we mentioned above, it is the key step in their approach to quantization of integrable
hierarchies of topological type.
\end{remark}

\section{Universal deformation group}\label{sec:GTgroup}

We begin this section by introducing the operad $\Gra$ of
all natural operations
on the totalisation of a (shifted) modular operad; this operad, spanned by graphs, is the connected
part of the operad $\mathrm{Gra}^\circlearrowleft_0$ of Kontsevich and Willwacher. It naturally
contains the operad
$\S \Delta \Lie$ of shifted $\Delta$-Lie algebras and we define the \emph{universal deformation
 group} of morphisms of modular operads as the even homology groups of the deformation complex of
the
inclusion \[\S \Delta \Lie \hookrightarrow \Gra .\]  This deformation complex is isomorphic to
a quantum extension of the Kontsevich graph complex introduced by Merkulov--Willwacher in
\cite{MW14}.
So we can use their method
to prove that the universal
deformation group contains the prounipotent Grothendieck--Teichm\"uller group and that it acts
functorially on the
moduli spaces
of Maurer--Cartan elements of the quantum
totalisations
of modular operads. We give a bit more details about the material of \emph{op.~cit.\ } using the language and the tools of the pre-Lie deformation theory~\cite{DSV16} this time.
In the end, this allows us to reach
one of the main goals of the present paper: to define an
action of the prounipotent Grothendieck--Teichm\"uller group on the moduli spaces of gauge
equivalence classes of quantum homotopy cohomological field theories.

\subsection{The operad of connected graphs}\label{subsec:CGRa}

\begin{definition}[The connected graphs operad]
The operad $\Gra$ is spanned by connected graphs with
numbered vertices and
edges of homological degree $-1$; loops
(edges connecting a vertex to itself)
are allowed. The partial composition product $\upsilon_1\circ_k\upsilon_2$ amounts to first inserting the graph $\upsilon_2$ at the $k$th vertex of $\upsilon_1$, then relabelling accordingly the vertices, and finally considering the sum of all the possible ways to reconnect the edges in $\upsilon_1$ originally plugged to the vertex $k$, to any possible vertex of $\upsilon_2$.

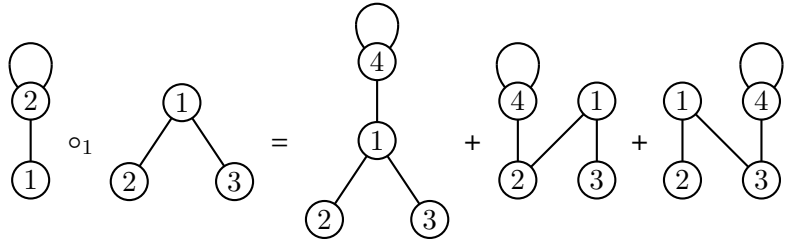
\[
\vcenter{\hbox{\begin{tikzpicture}[scale=0.7]
	 \draw[thick] (0,0.5)--(0,2);

	\draw[thick]  (-0.4,2.7) arc [radius=0.4, start angle=180, end angle= 0];
	\draw[thick] (0, 2) to [out=135,in=270] (-0.4,2.7);
	\draw[thick] (0, 2) to [out=45,in=270] (0.4,2.7);

	 \draw[fill=white, thick] (0,0.5) circle [radius=10pt];
 	 \draw[fill=white, thick] (0,2) circle [radius=10pt];

 	\node at (0,0.5) {\scalebox{1}{$1$}};
 	\node at (0,2) {\scalebox{1}{$2$}};

	\node (n) at (0,-0.5) {};
	\end{tikzpicture}}}
\ \circ_1 \
	\vcenter{\hbox{\begin{tikzpicture}[scale=0.7]
	 \draw[thick] (-1,0.5)--(0,2)--(1, 0.5);

	 \draw[fill=white, thick] (-1,0.5) circle [radius=10pt];
	 \draw[fill=white, thick] (1,0.5) circle [radius=10pt];
 	 \draw[fill=white, thick] (0,2) circle [radius=10pt];

 	\node at (-1,0.5) {\scalebox{1}{$2$}};
 	\node at (1,0.5) {\scalebox{1}{$3$}};
 	\node at (0,2) {\scalebox{1}{$1$}};
	\end{tikzpicture}}}
\ = \
\vcenter{\hbox{\begin{tikzpicture}[scale=0.7]
	 \draw[thick] (-1,-1)--(0,0.5)--(1, -1);
 	 \draw[thick] (0,0.5)--(0,2);

	\draw[thick]  (-0.4,2.7) arc [radius=0.4, start angle=180, end angle= 0];
	\draw[thick] (0, 2) to [out=135,in=270] (-0.4,2.7);
	\draw[thick] (0, 2) to [out=45,in=270] (0.4,2.7);

	 \draw[fill=white, thick] (-1,-1) circle [radius=10pt];
	 \draw[fill=white, thick] (1,-1) circle [radius=10pt];
 	 \draw[fill=white, thick] (0,2) circle [radius=10pt];
  	 \draw[fill=white, thick] (0,0.5) circle [radius=10pt];

 	\node at (-1,-1) {\scalebox{1}{$2$}};
 	\node at (1,-1) {\scalebox{1}{$3$}};
 	\node at (0,2) {\scalebox{1}{$4$}};
 	\node at (0,0.5) {\scalebox{1}{$1$}};

	\node (n) at (0,-2) {};
	\end{tikzpicture}}}
\ + \
\vcenter{\hbox{\begin{tikzpicture}[scale=0.7]
	 \draw[thick] (0,2)--(0,0.5)--(1.5,2)--(1.5,0.5);

	\draw[thick]  (-0.4,2.7) arc [radius=0.4, start angle=180, end angle= 0];
	\draw[thick] (0, 2) to [out=135,in=270] (-0.4,2.7);
	\draw[thick] (0, 2) to [out=45,in=270] (0.4,2.7);

	 \draw[fill=white, thick] (0,0.5) circle [radius=10pt];
 	 \draw[fill=white, thick] (0,2) circle [radius=10pt];
	 \draw[fill=white, thick] (1.5,0.5) circle [radius=10pt];
  	 \draw[fill=white, thick] (1.5,2) circle [radius=10pt];

 	\node at (0,0.5) {\scalebox{1}{$2$}};
 	\node at (0,2) {\scalebox{1}{$4$}};
 	\node at (1.5,0.5) {\scalebox{1}{$3$}};
 	\node at (1.5,2) {\scalebox{1}{$1$}};

	\node (n) at (0,-0.5) {};
	\end{tikzpicture}}}
\ + \
\vcenter{\hbox{\begin{tikzpicture}[scale=0.7]
	 \draw[thick] (0,2)--(0,0.5)--(-1.5,2)--(-1.5,0.5);

	\draw[thick]  (-0.4,2.7) arc [radius=0.4, start angle=180, end angle= 0];
	\draw[thick] (0, 2) to [out=135,in=270] (-0.4,2.7);
	\draw[thick] (0, 2) to [out=45,in=270] (0.4,2.7);

	 \draw[fill=white, thick] (0,0.5) circle [radius=10pt];
 	 \draw[fill=white, thick] (0,2) circle [radius=10pt];
	 \draw[fill=white, thick] (-1.5,0.5) circle [radius=10pt];
  	 \draw[fill=white, thick] (-1.5,2) circle [radius=10pt];

 	\node at (0,0.5) {\scalebox{1}{$3$}};
 	\node at (0,2) {\scalebox{1}{$4$}};
 	\node at (-1.5,0.5) {\scalebox{1}{$2$}};
 	\node at (-1.5,2) {\scalebox{1}{$1$}};

	\node (n) at (0,-0.5) {};
\end{tikzpicture}}} \]
\end{definition}

\begin{remark}
This operad is  a connected  and homological version of an operad introduced by M. Kontsevich
\cite{Kontsevich93, Kontsevich97} in order to encode the natural operations on polyvector fields on
affine manifolds.
Kontsevich's operad
plays a key role in T. Willwacher's theory~\cite{Willwacher15},
where it is denoted by  $\mathrm{Gra}^\circlearrowleft_0$.
More precisely, we have
$\mathrm{Gra}^\circlearrowleft_0\cong \Com\circ \Gra$, where the operad structure on the right-hand
side is given by a suitable distributive law. Notice however the following slight
difference
between the two approaches: in Kontsevich and Willwacher version, the edges receive homological
degree $+1$, that is cohomological degree $-1$ in \emph{op.\ cit.}.
The conceptual reason lies in the origin of these operads: Kontsevich's operad encodes all the
universal operations of a cohomological object (polyvector fields) whereas our operad encodes all
the universal operations of a homological object (homology of topological modular operads).
\end{remark}

\begin{proposition}\label{prop:GraAction} Let $\P$ be
a
shifted modular operad.
\begin{enumerate}
\item The operad of graphs $\Gra$ acts naturally on the totalisation $\whP$; this algebra structure
is
given by a morphism of operads
$\Psi : \Gra \to \End_{\whP}$\ .
\item The assignment
\[
\Delta \mapsto
\rule{0pt}{20pt}
\vcenter{\hbox{\begin{tikzpicture}[scale=0.7]
	\draw[thick]  (-0.4,2.7) arc [radius=0.4, start angle=180, end angle= 0];
	\draw[thick] (0, 2) to [out=135,in=270] (-0.4,2.7);
	\draw[thick] (0, 2) to [out=45,in=270] (0.4,2.7);
 	 \draw[fill=white, thick] (0,2) circle [radius=10pt];
 	\node at (0,2) {\scalebox{1}{$1$}};
	\node (n) at (0,1.5) {};
\end{tikzpicture}}}
\quad \text{and}\quad
\{\, , \} \mapsto
	\vcenter{\hbox{\begin{tikzpicture}[scale=0.7]
	 \draw[thick] (0,0)--(1.5, 0);

	 \draw[fill=white, thick] (0,0) circle [radius=10pt];
	 \draw[fill=white, thick] (1.5,0) circle [radius=10pt];

 	\node at (0,0) {\scalebox{1}{$1$}};
 	\node at (1.5,0) {\scalebox{1}{$2$}};
	\end{tikzpicture}}}
 \]
defines a morphism of operads $\Theta : \S \Delta \Lie \to \Gra$.
\item The shifted $\Delta$-Lie algebra structure on $\whP$ is equal to the
composite of the above two morphisms of operads
\[
\vcenter{\hbox{
\begin{tikzcd}
\S \Delta \Lie
\arrow[rr, "\Phi"] \arrow[dr, "\Theta"']
& & \End_{\whP}\ .
 \\
&\Gra  \arrow[ur, "\Psi"']  &
\end{tikzcd}
}}
\]

\end{enumerate}
\end{proposition}

\begin{proof}\leavevmode
\begin{enumerate}
\item A connected graph $\upsilon\in \Gra(n)$ is
a collection of edges
$\{\mathrm{t}_{ij}\}$ where $i,j$
are labels of
two vertices, which can be the same.
The action of $\upsilon$ on $n$ elements $\gamma_1, \ldots, \gamma_n$ of $\whP$ amounts to applying the bracket $\{\gamma_i, \gamma_j\}$ for any edge ${\mathrm{t}_{ij}}$ of $\upsilon$, with $i\neq j$, and to applying the operator $\Delta(\gamma_i)$ for any loop ${\mathrm{t}_{ii}}$ of $\upsilon$. It is straightforward to see that this defines a $\Gra$-algebra structure.

\item Since the operad $\S \Delta \Lie$ is generated by $\Delta$ and $\{\, , \}$,
we only need to verify that the operations corresponding to
the one-loop graph and the one-edge graph satisfy the square-zero relation, the
derivation relation, and the shifted Jacobi relation.
This is straightforward.

\item This follows from the
definitions of morphisms $\Psi$ and $\Theta$.
\end{enumerate}
\end{proof}

\begin{example}\label{ex:CGRABV}
\cref{prop:GraAction} lifts the shifted $\Delta$-Lie algebra structure on the totalisation of the endomorphism modular operad of \cref{ex:TPoly}. In the case of the odd affine symplectic manifold $V\oplus sV^*$, it endows
$\K[[V\oplus s V^*]][[t]]$ with a canonical $\Gra$-algebra structure. Setting $t=0$ and considering
the cohomological degree convention,
we obtain
the action of  the version of $\Gra$ defined with edges of degree $+1$ on polyvector fields of affine manifolds due to Kontsevich~\cite{Kontsevich97}.
\end{example}

\begin{remark}\label{rem:ConceptualApproach}
There is actually a heuristic reason for the operad $\Gra$ to act on the totalisation of
a shifted modular
operad and for being the operad of all its natural operations. As explained in
 \cref{subsec:ModOp}, the notion of a modular operad is encoded by the monad of graphs whose
 composition product is given by the insertion of graphs. When one forgets which input is which
 under the totalisation functor, it is natural to expect that graphs without leaves act on it. The
 operadic composition product of these graphs is then given by the insertion of such graphs,
 thereby producing the  operad $\Gra$. Such a pattern can already be seen on the level of operads, which are algebras over the monad of rooted trees with leaves equipped with the insertion of trees. The operad encoding all the natural operations acting on the totalisation of an operad is the rooted tree operad, which is spanned by rooted trees without leaves with composition maps given by the insertion of such trees.
In the case of classical operads, the analogue of the operad $\S \Delta \Lie$ which encodes the
minimal structure required to write down the Maurer--Cartan equation is the operad
$\mathrm{pre}\textrm{-}\mathrm{Lie}$ encoding pre-Lie algebras. In this case, the two operads of
all operations (rooted trees) and of the minimal set of operations
($\mathrm{pre}\textrm{-}\mathrm{Lie}$) are isomorphic~\cite{ChapotonLivernet01}. This is not the
case here and it is precisely this discrepancy that will give rise to a universal deformation group, 
see~\cref{sec:UDGGTgroup}.
\end{remark}

\subsection{Cycle action}\label{subsec:CycleAct}
Similarly to
how it was done in \cref{subsec:DefCom}, the
space of equivariant maps
\[\widehat{\Hom}_{\Sy} \left(\CC, \PP\right)\coloneq\prod_{n\geqslant 1} \Hom_{\Sy_n}\left(
\CC(n), \PP(n)\right)\]
from
any cooperad to an operad can be equipped with a pre-Lie algebra operation $\star$ which induces a
Lie algebra bracket denoted by $[\, ,]$
(see~\cite[Section~6.4]{LodayVallette12} and~\cite{DSV16} for more details).

\begin{lemma}\label{lemma:CompletePreLie}
The convolution dg pre-Lie algebra associated to any operad $\PP$ and to the Koszul dual cooperad
$\CC=(\S\Delta\Lie)^{\ac}$  is the dg pre-Lie algebra
\[\aa_\PP\coloneq\left(\widehat{\Hom}_{\Sy} \big((\S\Delta\Lie)^{\ac}, \PP\big), \partial, \star \right)\cong
\Big( \prod_{\substack{n\geqslant 1\\ k\geqslant 0}}\PP(n)^{\Sy_n} \hbar^k, \partial, \star \Big)\ ,
\]
where the
pre-Lie
product is the $\K[[\hbar]]$-extension of the pre-Lie product on the totalisation of the operad
$\PP$, with  $\hbar$ of degree $0$.
\end{lemma}

\begin{proof}
This follows from
the computation of the Koszul dual cooperad:
\[(\S\Delta\Lie)^{\ac} \cong \T^c(s\Delta)\circ \Com^* \cong \T^c(s\Delta)\otimes \Com^*\ ,\]
where $s\Delta$ is an
element of arity $1$ and degree $0$.
The Koszul dual operad is given by the distributive law such that the composite at any place with the element dual to $s\Delta$ produces the same result. This shows that the pre-Lie product on $\aa_\PP$ is the $\K[[\hbar]]$-extension of the pre-Lie product from~$\PP$.
\end{proof}

\begin{remark}\label{rem:CohoConv}
In order to encode all the natural operations acting on an ``unshifted'' modular operad, we should
rather consider the alternative
version of
the operad $\Gra$, with the cohomological degree convention, that is with edges of degree $+1$.
In this case, one should also consider the ``unshifted'' version of the operad $\S \Delta \Lie$.
Then, the homological degree of an element $\gamma \hbar^k$ of $\aa_{\Gra}$ would be equal to
$-2(v-1)+e-2k$, where the number of vertices of the graph $\gamma$ is denoted by $v$ and
the number of edges by $e$. In this way, we recover exactly the cohomological degree
$2(v-1)-e+2k$ of~\cite{Kontsevich97, Willwacher15} and the dg Lie algebra of~\cite{MW14}.
\end{remark}

We denote by
\[\ba_\PP\coloneq\left(\widehat{\Hom}_{\Sy} \left(\overline{(\S\Delta\Lie)}^{\ac}, \PP\right)\right)\cong
\prod_{\substack{n\geqslant 1, \ k\geqslant 0 \\(n,k)\neq (1,0)}}\PP(n)^{\Sy_n} \hbar^k\ ,
\]
the complete dg pre-Lie subalgebra of $\aa_\PP$ consisting of elements of positive weight, where the weight is equal to the power of $\hbar$ plus the arity minus 1.
The set of solutions to the associated Maurer--Cartan equation $\partial \phi + \phi \star \phi=0$
 is denoted by $\MC(\ba_\PP)$; it is in natural bijection with the set of morphisms of operads $\Omega \big((\S\Delta\Lie)^{\ac}\big) \to \PP$. For instance, for any  shifted modular operad  $\P$ , we denote by
$\phi\in \ba_{\End_{\whP}}$ the Maurer--Cartan element corresponding to
\[\Omega \big((\S\Delta\Lie)^{\ac}\big) \xrightarrow{\sim} \S\Delta\Lie \xrightarrow{\Phi} \End_{\whP}  \]
and by
$\theta \in \ba_\Gra$ the Maurer--Cartan element corresponding to
\[\Omega \big((\S\Delta\Lie)^{\ac}\big) \xrightarrow{\sim} \S\Delta\Lie \xrightarrow{\Theta} \Gra  \ .\]
We use the same notations for any $\Gra$-algebra structure $\Psi : \Gra \to \End_\g$, where the induced
shifted $\Delta$-Lie algebra  structure given by $\Phi\coloneq \Psi\circ \Theta$ correspond to a Maurer--Cartan element $\phi$ of $\ba_{\End_\g}$.

\begin{lemma}\label{lem:TW}
Any morphism of operads $\Psi : \Gra \to \End_\g$ defining a $\Gra$-algebra structure on $\g$
 induces a morphism of dg Lie algebras
\[\Psi_* :
\ba_{\Gra}
\to
\ba_{\End_{\g}}\]
and thus a morphism of twisted dg Lie algebras
\[\Psi_* :
\ba_{\Gra}^\theta
\to
\ba_{\End_{\g}}^\phi\ .\]
\end{lemma}

\begin{proof}
The first assertion follows from the general theory of operadic twisting morphisms~\cite[Section~6.4]{LodayVallette12}:
the convolution dg Lie algebra is functorial in the second entry. The second
assertion is a well known fact: any morphism of dg Lie algebras sends
Maurer--Cartan elements to Maurer--Cartan elements and it induces a morphism between the
corresponding twisted dg Lie algebras, see~\cite[Chapter~3,Section~5]{DotsenkoShadrinVallette18} for
instance. We apply this here to $\Psi_*(\theta)=\theta\circ \Psi=\phi$\ .
\end{proof}

The twisted dg Lie algebra
$\ba_{\Gra}^\theta$ controls the deformation theory of the morphism of operads $\Theta$: $\mathrm{Def}\big(\S \Delta \Lie \xrightarrow{\Theta} \Gra\big)$, see~\cite{MerkulovVallette09I}.

\begin{remark}
Notice that the two twisted differentials $\partial^\theta$ and $\partial^\phi$ increase the weight grading by~$1$.
\end{remark}

\begin{convention}
From now on, like in \cref{subsec:DefCohFT},
we drop the $\ZZ$-grading and pass to the $\Z2$-grading,
that is we consider finite sums of even elements or finite sums of odd elements.
For instance, Maurer-Cartan elements
in dg pre-Lie algebras will  be elements of odd degree and Maurer--Cartan elements in shifted $\Delta$-Lie algebras will be elements of even degree.
\end{convention}

\begin{corollary}\label{cor:MorphBCH}
There is a morphism of groups
\[\Psi_* :
\Big(Z_\ev\Big(
\ba_\Gra^\theta
\Big) , \BCH, 0\Big) \to
\Big(Z_\ev
\Big(
\ba_{\End_{\g}}^\phi\Big),
 \BCH, 0\Big)\ ,\]
where $\BCH$ stands for the Baker--Campbell--Hausdorff formula.
\end{corollary}

\begin{proof}
We have to restrict to positive weight elements for the Baker--Campbell--Hausdorff formula to converge. The rest is a direct corollary of \cref{lem:TW}: the morphism of Lie algebras integrates to a morphism of groups.
\end{proof}

Let us recall from~\cite[Chapter~10]{LodayVallette12} that, given a dg module $\g$,
a  Maurer--Cartan elements $\alpha$ of the dg pre-Lie algebra
$\ba_{\End_{\g}}$ corresponds to a $(\SDL)_\infty$-algebra structures on $\g$.
Such a structure amounts to a collection of operations
\[\left\{
\ell^k_n : \g^{\odot n}\to \g
\right\}_{n\geqslant 1, k\geqslant 0}\]
of degree $-1$  (odd degree here)
satisfying the $\K[[\hbar]]$-extension of the relations of a shifted $\Lie_\infty$-algebra, see~\cite[Chapter~3, Section~4]{DotsenkoShadrinVallette18} for instance.
Given now two such
$(\SDL)_\infty$-algebra structures $\alpha$ and $\beta$ on $\g$,
a degree $0$ (even $\Z2$-degree here) element $f\in \aa_{\End_{\g}}$ satisfying
the equation $\partial f= f\star \alpha-\beta \circledcirc f$ corresponds to a $(\SDL)_\infty$-morphism from $\alpha$ to $\beta$.
Such a data amounts to a collection
\[\left\{
f^k_n : \g^{\odot n}\to \g
\right\}_{n\geqslant 1, k\geqslant 0}\] such that
$\left|f^k_n\right|=n+k-1$ and satisfying the $\K[[\hbar]]$-extension of the relations of
a morphism of shifted $\Lie_\infty$-algebras.
We call \emph{$(\SDL)_\infty$-isotopies}, the $(\SDL)_\infty$-morphisms
whose first component $f^0_1=\id$ is the identity map. As usual, the invertible $(\SDL)_\infty$-morphisms are the ones for which the first component $f^0_1$ is invertible.

\begin{lemma}\label{lem:PreLieExp}
Let $\phi\in \MC(\ba_{\End_{\g}})$ be an $(\SDL)_\infty$-algebra structure on $\g$.
The pre-Lie exponential $e^\xi$ of any
$\xi\in Z_\ev
\Big(
\ba_{\End_{\g}}^\phi
\Big)$
is an $(\SDL)_\infty$-isotopy of the $(\SDL)_\infty$-algebra structures $\phi$ on $\g$.
\end{lemma}

\begin{proof}
The pre-Lie exponential $e^\xi$ lives in the pre-Lie algebra $\aa_{\End_{\g}}$, which is complete by \cref{lemma:CompletePreLie}.
Let $\xi$ be an element of $\ba_{\End_{\g}}\cap \ker \partial^\phi$. Using the differential trick from~\cite[Section~5]{DSV16},  the  gauge action of
$e^\xi$ on $\phi$ is given by
\[
e^\xi\cdot \phi = e^{\mathrm{ad}_\xi}(\phi)=\phi
\]
since $\partial^\phi(\xi)=[\phi, \xi]=0$. So $e^\xi$ is an $\infty$-isotopy from the
$(\SDL)_\infty$-algebra structures $\phi$ on $\g$
 to itself by ~\cite[Theorem~3]{DSV16}.
\end{proof}

Let us denote by $\gh\coloneq \g\otimes \K[[\hbar]]$ the $\K[[\hbar]]$-extension of the complete dg module $\g$\ .

\begin{lemma}\label{lem:DeltaLietoLie}
The assignment
\begin{align*}
\begin{array}{crcl}
\Sigma : & \left(\g, \F, \left\{ \ell_n^k \right\}_{n\geqslant 1, k\geqslant 0}\right) &\mapsto&
\left(\gh, \F[[\hbar]], \left\{
\ell_n \coloneqq \sum_{k\geqslant 0} \hbar^k \ell^k_n
\right\}_{n\geqslant 1}
\right)\\
&\rule{0pt}{12pt}\left\{
f^k_n
\right\}_{n\geqslant 1, k\geqslant 0} &\mapsto &
\left\{
f_n \coloneqq \sum_{k\geqslant 0}\hbar^k f^k_n
\right\}_{n\geqslant 1}
\end{array}
\end{align*}
defines a functor from the category of complete shifted $(\SDL)_\infty$-algebras over $\K$ to the category of complete shifted $\Lie_\infty$-algebras over $\K[[\hbar]]$~.
\end{lemma}

\begin{proof}
The $\K[[\hbar]]$-extension of $\g$ ensures that the various sums make sense. To establish all the algebraic relations, it is enough to analyze the two corresponding convolution algebras.
On the one hand, the complete convolution pre-Lie algebra encoding the category of  complete $(\SDL)_\infty$-algebras is $\aa_{\End_\g}$~.
On the other hand,
the complete convolution pre-Lie algebra encoding the category of complete $(\SL)_\infty$-algebras is equal to
\[\b_{\End_\g}\coloneq\left(\widehat{\Hom}_{\Sy} \big((\S\Lie)^{\ac}, \PP\big), \partial, \star \right)\cong
\Big( \prod_{n\geqslant 1}\End_\g(n)^{\Sy_n}, \partial, \star \Big)\ ,
\]
The argument given in the proof of \cref{lemma:CompletePreLie} shows that $\aa_{\End_\g}$ is the $\K[[\hbar]]$-extension of $\b_{\End_\g}$~. More precisely, since $(\S\Lie)^{\ac}\cong \Com^*$, the morphism of operads
$\T(s\Delta)\circ \Com\to \Com$ which sends $(s\Delta)^k\otimes \mu$ to $\mu$ induces a morphism from the former pre-Lie algebra to the latter pre-Lie algebra;  explicitly, this morphism of pre-Lie algebras
amounts to taking the sum over $k$.
This shows that Maurer--Cartan elements are sent to Maurer--Cartan elements and that
 complete shifted $(\SDL)_\infty$-algebras  are sent complete shifted $\Lie_\infty$-algebras.
The argument is similar for the $\infty$-morphisms.
\end{proof}

The restriction of the functor $\Sigma$ to complete shifted $\Delta$-Lie algebras
\[
\big(\g, \F, \d, \Delta, \{\ , \,\}\big) \mapsto \gth=\big(\g[[\hbar]], \F[[\hbar]], \d+\hbar\Delta, \{\ , \,\}\big)
\]
is the quantum functor of \cref{subsec:Quantum}
from the subcategory of complete shifted $\Delta$-Lie algebras over $\K$ to the subcategory of complete shifted $\Lie$-algebras over $\K[[\hbar]]$. This functor preserves
the  respective $\infty$-morphisms.
It also induces a group morphism between the respective notions of $\infty$-isotopies.

\medskip

Recall that there is an action of the group of $\Lie_\infty$-automorphisms
on the set of Maurer--Cartan elements of a complete shifted $\Lie_\infty$-algebra  given by
\[
f\cdot \alpha \coloneqq \sum_{n\geqslant 1} {\textstyle \frac{1}{n!}} f_n(\alpha, \ldots, \alpha)
\ ,\]
see~\cite[Chapter~3, Section~5]{DotsenkoShadrinVallette18} for more details including conceptual explanations.
This is always well defined since the Maurer--Cartan elements are supposed to live in the layer $\F_1$ of the complete filtration. Recall that, in the quantum complete shifted $\Lie$-algebra $\gth$, we consider Maurer--Cartan elements in $(\F_1\, \g)[[\hbar]]$, that is solution to the \emph{quantum} master equation
\[
\d \alpha + \hbar\Delta \alpha + {\textstyle \frac12} \{\alpha, \alpha\}=0\ .
 \]

\begin{proposition}\label{prop:CycleAction}
The composite of the above-mentioned group morphisms defines a group action
\[
Z_\ev\Big(
\aa_\Gra^{\theta}\Big)
\xrightarrow{\Psi_*}
Z_\ev\Big(
\aa_{\End_{\g}}^\phi\Big)
 \xrightarrow{e}
\S\Delta\Lie_\infty\textrm{-}\mathrm{Aut}\big(\g\big)
\xrightarrow{\sum}
\S\Lie_\infty\textrm{-}\mathrm{Aut}\big(\gth\big)\to
\Aut\big(
\MC\big(\gth\big)
\big)
\ ,\]
   functorial in complete $\Gra$-algebras $\Psi : \Gra \to \End_\g$. It is
explicitly given by
\[\lambda \cdot \alpha \coloneqq \sum_{\substack{n\geqslant 1\\ k\geqslant 0}}
{\textstyle \frac{1}{n!}}
\left(e^{\Psi(\lambda)}\right)^k_n(\alpha, \ldots, \alpha)\, \hbar^k\ .\]
\end{proposition}

\begin{proof}
This follows automatically
by applying first \cref{cor:MorphBCH}, then \cref{lem:PreLieExp} and finally the fact that the pre-Lie exponential defines a group morphism from the gauge group with the BCH formula to the group of $\infty$-isotopies
\cite[Theorem~2]{DSV16}.
Notice that any even cycle of $\aa_\Gra^{\theta}$ has trivial arity~$1$ component, so $Z_\ev\Big(
\aa_\Gra^{\theta}\Big)
$ contains no element of weight $0$. It thus forms a well-defined group under the BCH formula.
Finally, one uses \cref{lem:DeltaLietoLie} and the action of shifted $\Lie_\infty$-morphisms on Maurer--Cartan elements.
\end{proof}

\cref{prop:CycleAction} applies to modular operads to provide us with a functorial action of the group $Z_\ev\Big(
\aa_\Gra^{\theta}\Big)$ on Maurer--Cartan elements of the quantum totalisation
$\whP^\hbar$ of  shifted modular operads.

\begin{remark}\label{rem:CommPeS}
Starting from the pre-Lie algebra $\aa_\Gra$, the three maps $\Psi_*$, $e$, and $\Sigma$ commute in all possible ways since $\Psi_*$ and $\Sigma$ respond respectively to pushing along a morphism of operad and pulling back along a morphism of cooperads and since the pre-Lie exponential $e$ is a universal construction for pre-Lie algebras.
Notice also that the image of the set of elements of arity greater or equal to $2$ of $\aa_\Gra$ under the pre-Lie exponential map is equal to the set of group-like elements $\{\1\}\times \prod_{n\geqslant 2, k\geqslant 0}\Gra(n)^{\Sy_n} \hbar^k$, which is endowed with the group structure given by $\circledcirc$, by~\cite[Lemma~1]{DSV16}.
\end{remark}

\begin{remark}
Without passing to the $\Z2$-grading,
the above-mentioned result would be empty since the summand of homological degree 0  of $\aa_\Gra$ is trivial.
\end{remark}

\subsection{Homology action}\label{subsec:HoAction}
We will now look at the action from~\cref{prop:CycleAction} at the level of homology.
\medskip

Any  $\Gra$-algebra structure on $\g$ contains a  shifted $\Delta$-Lie algebra structure
by Point~(2) of \cref{prop:GraAction},
which gives rise to a quantum complete  shifted Lie algebra $\gth$ by \cref{lem:DeltaLietoLie}. Using the Baker--Campbell--Hausdorff formula, the set $\F_1 \g_\od[[\hbar]]$ can be given the structure of a gauge group which acts on the set of Maurer--Cartan elements. The corresponding set of equivalence classes, denoted by $\calMC\big(\gth\big)$, is the \emph{moduli space of Maurer--Cartan elements}.

\begin{theorem}\label{cor:MorphBCHhomo}
The group morphism of \cref{prop:CycleAction} induces a group morphism
\[
\left(H_\ev\Big(
\aa_\Gra^{\theta}\Big), \BCH,0\right)  \to
\Aut\left(
\calMC\big(\gth\big)
\right)\]
   functorial in complete $\Gra$-algebras.
\end{theorem}

This statement follows immediately from the following general result, stated without proof in~\cite[Section~3.1]{MW14}. Recall that any  shifted dg  Lie algebra structure $(\g, \d, \{\, , \})$ is encoded by a Maurer--Cartan element $\phi$ concentrated in arity $2$, in the pre-Lie algebra
\[\bb_{\End_\g}\coloneq
  \left(\prod_{n\geqslant 2}\End_\g(n)^{\Sy_n}, \partial, \star \right)\ .\]
The above arguments show that
even cycles $\lambda \in Z_\ev\Big(\bb_{\End_\g}^{\phi}\Big)$ act on
the set of Maurer--Cartan  elements $\alpha$ of the shifted Lie algebra $\g$
by
the formula
\[\lambda \cdot \alpha \coloneqq \sum_{n\geqslant 1}
{\textstyle \frac{1}{n!}} \left(e^{\lambda}\right)_n(\alpha, \ldots, \alpha)\ .\]

\begin{proposition}\label{lem:HoImpliesGauge}
For any Maurer--Cartan element $\alpha$ of a  complete shifted $\Lie_\infty$-algebra $(\g, \d, \phi)$, the actions of two homologous cycles $\lambda$ and $\mu$ with even degree of\, $\bb_{\End_\g}^{\phi}$ produce two gauge equivalent elements:
$\lambda\cdot \alpha \sim \mu \cdot \alpha\ .$
\end{proposition}

\begin{proof}
We consider the Sullivan algebra $\K[t, dt]$ of the $1$-simplex. Notice that under the present  homological degree convention  $|t|$ is even and $|dt|$ is odd, with differential $d(t)=dt$.
Its complete tensor with $\g$, defines a complete shifted $\Lie_\infty$-algebra denoted by $\g[t, dt]\coloneq \g \widehat{\otimes} \K[t, dt]$~.
Let $\lambda,\mu\in Z_\ev\Big(\bb_{\End_\g}^{\phi}\Big)$
be two homologous cycles, that is there exists $f \in \bb_{\End_\g}^{\phi}$ of odd degree satisfying $\partial^{\phi}(f)=\lambda-\mu$.
The unit and the augmentation of $\K[t, dt]$ show that $\g$ is a retract of $\g[t, dt]$~.
By a slight abuse of notations, we use the same letter to depict the extensions of elements from
$\bb_{\End_{\g}}$ to $\bb_{\End_{\g[t,dt]}}$\ .
The even degree element $\xi\coloneq t\lambda+(1-t)\mu-f dt$ is a cycle in $\bb_{\End_{\g[t, dt]}}^\phi$. So its pre-Lie exponential $e^\xi$ is an $\Lie_\infty$-automorphism of $(\g[t, dt], \phi)$. Its action $e^\xi\cdot \alpha$ on $\alpha$ provides us with a  Maurer--Cartan element in $\g[t, dt]$ which is equal to
$e^\mu\cdot \alpha$ at $t=0$ and to $e^\lambda\cdot \alpha$ at $t=1$.
This shows that these two Maurer--Cartan elements are homotopically  equivalent. Since homotopy equivalence is equivalent to  gauge equivalence, see~\cite[Corollary~5.3]{RN19} for instance, this concludes the proof.
\end{proof}

\begin{remark}
This proof shows that  two homologous cycles in
$Z_\ev\Big(\bb_{\End_\g}^{\phi}
\Big)$
 induce two homotopy equivalent $\Lie_\infty$-automorphisms of the shifted $\Lie_\infty$-algebra $\g$ under
the pre-Lie exponential. It was originally proved in~\cite[Lemma~3]{Willwacher14}.  We refer the reader to~\cite{DotsenkoPoncin16, Vallette14} for the various equivalent forms of this latter equivalence relation.
\end{remark}

\begin{remark}
In fact, using computations in pre-Lie and dendriform algebras, it is possible to come up with explicit formulas for the gauge equivalence, and not just prove its existence. We shall address the arising combinatorial phenomena elsewhere.
\end{remark}

\begin{proof}[Proof of \cref{cor:MorphBCHhomo}]
We notice first that the BCH formula passes to homology.
Then since the maps $\Psi_*$, $e$, and $\Sigma$ commute, see \cref{rem:CommPeS}, and since the maps $\Psi_*$ and $\Sigma$ are morphisms of dg pre-Lie algebras, it is enough to prove that the action of
the group $Z_\ev\Big(\bb_{\End_{\gh}}^{\Sigma^\phi}\Big)$ on $\MC\big(\gth\big)$  passes to homology and to gauge equivalence classes respectively. To do so, one needs to prove that, for any Maurer--Cartan element of the quantum shifted Lie algebra $\gth$, the actions of two homologous even cycles of $\bb_{\End_{\gh}}^{\Sigma^\phi}$ produce two gauge equivalent elements.
Denoting by $\varphi$ the arity $2$ part of $\phi$ corresponding only to the $\hbar$-extension of the bracket $\{\, , \}$ of $\gh$,
one can see that the two twisted dg Lie algebras
\[
\bb_{\End_{(\gh, \d)}}^{\Sigma\phi}
=
\bb_{\End_{(\gh, \d+\hbar\Delta)}}^{\varphi}
\]
are equal. \cref{lem:HoImpliesGauge} applied to the complete quantum   shifted Lie algebra  $\gth$ concludes the proof.
\end{proof}

\subsection{Universal  deformation group and the Grothendieck--Teichm\"uller group}\label{sec:UDGGTgroup}

\begin{definition}[Universal  deformation group]
The \emph{universal  deformation group} is the prounipotent group
\[
\mathrm{G}\coloneq \left( H_\ev\Big(
\aa_\Gra^{\theta}\Big),   \mathrm{BCH}, 0
\right)
\]
which integrates
the complete Lie algebra consisting of homology classes represented by graphs with an even number of edges
in the convolution Lie algebra supported by
$\prod_{n\geqslant 1, k\geqslant 0}\Gra(n)^{\Sy_n} \hbar^k$
and twisted by the Maurer--Cartan element
\[
\theta=
	\vcenter{\hbox{\begin{tikzpicture}[scale=0.5]
	 \draw[thick] (0,0)--(1.5, 0);

	 \draw[fill=white, thick] (0,0) circle [radius=10pt];
	 \draw[fill=white, thick] (1.5,0) circle [radius=10pt];
	\end{tikzpicture}}}
	+
	\vcenter{\hbox{\begin{tikzpicture}[scale=0.5]
	\draw[thick]  (-0.4,2.7) arc [radius=0.4, start angle=180, end angle= 0];
	\draw[thick] (0, 2) to [out=135,in=270] (-0.4,2.7);
	\draw[thick] (0, 2) to [out=45,in=270] (0.4,2.7);
 	 \draw[fill=white, thick] (0,2) circle [radius=10pt];
	\node (n) at (0,1.5) {};
\end{tikzpicture}}}\ \hbar\ .
 \]
\end{definition}

The elements of even degree in $\aa_\Gra$ are products of graphs with unlabeled vertices and an even number of edges, indexed by some $\hbar^k$. Since the vertices have odd degree, in order to perform computations, we always first choose a representative of a graph with a total order on its set of vertices.
Under the invariants-coinvariants identification, one also considers the sum of all possible labeling of vertices divided by the order of the symmetry group of the graph, see~\cite[Remark~2.3]{Willwacher15}.
The Lie bracket is equal to the skew-symmetrization of the insertion of a graph at each vertex of another graph and the sum of all the possible ways to attach the edges of this vertex to the first graph.

\begin{example}
The following
is
a good exercise. The tetrahedron
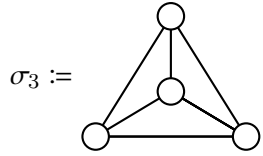
\[\sigma_3\coloneq
	\vcenter{\hbox{\begin{tikzpicture}[scale=0.5]
	 \draw[thick] (0,0)--(2, -1.2)--(0,2)--(-2,-1.2)--(2,-1.2)--(0,0);
	 \draw[thick] (0,0)--(2,-1.2)--(-2,-1.2)--(0,0)--(0,2);
	 \draw[fill=white, thick] (0,0) circle [radius=10pt];
	 \draw[fill=white, thick] (2,-1.2) circle [radius=10pt];
	 \draw[fill=white, thick] (-2,-1.2) circle [radius=10pt];
	 \draw[fill=white, thick] (0,2) circle [radius=10pt];
	\end{tikzpicture}}}
\]
is a cycle, i.e.\
$\partial^\theta(\sigma_3)=0$,
since
\[
\left[	\vcenter{\hbox{\begin{tikzpicture}[scale=0.5]
	 \draw[thick] (0,0)--(1.5, 0);

	 \draw[fill=white, thick] (0,0) circle [radius=10pt];
	 \draw[fill=white, thick] (1.5,0) circle [radius=10pt];
	\end{tikzpicture}}}\
, \sigma_3\right]=0
\qquad \text{and} \qquad
\left[\vcenter{\hbox{\begin{tikzpicture}[scale=0.5]
	\draw[thick]  (-0.4,2.7) arc [radius=0.4, start angle=180, end angle= 0];
	\draw[thick] (0, 2) to [out=135,in=270] (-0.4,2.7);
	\draw[thick] (0, 2) to [out=45,in=270] (0.4,2.7);
 	 \draw[fill=white, thick] (0,2) circle [radius=10pt];
	\node (n) at (0,1.5) {};
\end{tikzpicture}}}
\ , \sigma_3
\right]=0
\ .\]
\end{example}

To understand the universal deformation group in some detail, we need to know the homology of graph complexes which is very difficult to compute, see Section~2 of the ICM article~\cite{Willwacher18} for a survey.
Partial information can be extracted from Willwacher's results, like the main theorem of~\cite{Willwacher15} 
which states that the degree $0$ cohomology group of Kontsevich's graph complex is
the Lie algebra of
the prounipotent Grothendieck--Teichm\"uller group $\GRT_1$.
This
group was introduced by V. Drinfeld in~\cite{Drinfeld90} after
Grothendieck's \emph{Esquisse d'un programme}.
We refer the reader for instance to the survey~\cite{Merkulov19} for full definitions.

\medskip

Following~\cite{Kontsevich97, Willwacher15}, the dg Lie algebra $\aa_\Gra^{\theta}$ can be endowed with the following cohomological degree:
\[\left|\gamma\, \hbar^k\right|\coloneq 2(v-1)-e+2k\ ,\]
for any  graph $\gamma$ with $v$ vertices and $e$ edges, see \cref{rem:CohoConv}. Therefore the universal   deformation group splits with respect to this
degree:
\[
\mathrm{G}=\prod_{l\in \ZZ} \mathrm{G}^l\ , \qquad \text{where} \quad
\mathrm{G}^l\coloneq H^{2l}\Big(\aa_\Gra^{\theta}\Big)\ .\]

\begin{theorem}[{\cite[Proposition~2.1]{MW14}}]\label{thm:GRT1}\leavevmode
\begin{enumerate}
\item
The Grothendieck--Teichm\"uller group $\GRT_1$ sits inside the universal   deformation group as its cohomological degree $0$ part:
\[\mathrm{G}^0 \cong  \GRT_1\ .
\]
\item
The negative cohomological summands of the universal   deformation group are trivial:
\[\mathrm{G}^l=0\ , \quad \text{for}\ \  l<0\ .\]
\end{enumerate}
\end{theorem}

\begin{proof}
We recall the arguments of the proof of~\cite[Proposition~2.1]{MW14} under the present notations in order to be self-contained since we will use them later on. Notice that the chain complex $\aa_\Gra^{\theta}$ is denoted by
$(\mathsf{fGC_2}[[u]], d+u\Delta)$
in \emph{op.\ cit.}. We consider the $\hbar$-adic  filtration given by
\[\mathrm{F}_p\, \aa_\Gra^{\theta} \coloneq
\prod_{\substack{n\geqslant 1\\ l\geqslant -p}}\Gra(n)^{\Sy_n} \hbar^l
\]
 and  the convolution algebra
\[\b_\Gra^{\vartheta}\coloneq
\left(\prod_{n\geqslant 1}\Gra(n)^{\Sy_n}, \partial^{\vartheta}\right)\ ,\]
where $\vartheta$ stands for the Maurer--Cartan elements $
	\vcenter{\hbox{\begin{tikzpicture}[scale=0.5]
	 \draw[thick] (0,0)--(1.5, 0);

	 \draw[fill=white, thick] (0,0) circle [radius=10pt];
	 \draw[fill=white, thick] (1.5,0) circle [radius=10pt];
	\end{tikzpicture}}}$\ .
One can see that the first page of the associated spectral sequence is given by
\[E^0_{p,q}\cong \left(\prod_{n\geqslant 1}\Gra(n)^{\Sy_n}\right)^{p-q}\hbar^{-p}\ ,\]
for $p\leqslant 0$ and by $E^0_{p,q}=0$ otherwise, with $d^0=\partial^\vartheta$.
Since the chain complex $\b_\Gra^{\vartheta}$ is equal to the chain complex denoted by
$\mathsf{fGC}_{2, conn}^\circlearrowleft$ in~\cite{Willwacher15}, we know by Proposition~3.4 and Theorem~1.1
of \emph{op.\ cit.} that its negatively graded cohomology groups vanish and that its degree $0$
cohomology group is isomorphic, as complete Lie algebra, to the Grothendieck--Teichm\"uller Lie
algebra $\grt_1$\ .
So the second page of the spectral sequence has the following form
\[\begin{tikzpicture}[scale=1]

	\draw[->] (-4,0)--(1,0);
	\draw[->] (0,-3.65)--(0,1);

	\draw[fill=white, white] (0,0) circle [radius=10pt];
	\draw[fill=white, white] (0,-1) circle [radius=10pt];
	\draw[fill=white, white] (0,-2) circle [radius=10pt];
	\draw[fill=white, white] (0,-3) circle [radius=10pt];

 	\node at (0,0) {\scalebox{1}{$\grt_1$}};
 	\node at (-1,-1) {\scalebox{1}{$\grt_1 \hbar$}};
 	\node at (-2,-2) {\scalebox{1}{$\grt_1 \hbar^2$}};
 	\node at (-3,-3) {\scalebox{1}{$\grt_1\hbar^3$}};
 	\node at (0,-1) {\scalebox{1}{$*$}};
 	\node at (0,-2) {\scalebox{1}{$*$}};
 	\node at (-1,-2) {\scalebox{1}{$*$}};
 	\node at (0,-3) {\scalebox{1}{$*$}};
 	\node at (-1,-3) {\scalebox{1}{$*$}};
 	\node at (-2,-3) {\scalebox{1}{$*$}};

 	\node at (0.95,-0.25) {\scalebox{0.8}{$p$}};
 	\node at (0.25,0.95) {\scalebox{0.8}{$q$}};
\end{tikzpicture}.
\]
This shows that the spectral sequence is regular; it is clearly  exhaustive, complete, and bounded above, following definitions from~\cite[Chapter~5]{WeibelBook}.
 We can thus apply the Complete Convergence Theorem~\cite[Theorem~5.5.10]{WeibelBook}, which ensures that the spectral sequence converges to the homology groups of $\aa_\Gra^{\theta}$. This shows the second claim.
In order to fully prove the first claim, one needs to further apply the above arguments to the evaluation map $\aa_\Gra^\theta \to \b_\Gra^\vartheta$ at $\hbar=0$, which is a morphism of dg Lie algebras. This induces a group isomorphism $\mathrm{G}^0\cong \GRT_1$.
\end{proof}

\begin{remark}
So far, very little is known about
$\mathrm{G}^l\coloneq H^{2l}\Big(\aa_\Gra^{\theta}\Big)$, for $l>0$\ .
\end{remark}

When the work on the present project started, its primary goal was to obtain the following result.

\begin{theorem}\label{thm:GRT1Action}
The Grothendieck--Teichm\"uller group $\GRT_1$ acts functorially on the moduli spaces
of gauge equivalence classes of solutions to the quantum master equation of the quantum totalisation of
$\Z2$-graded shifted modular operads.
As a particular case, this includes an action of the group $\GRT_1$ on the moduli spaces
of morphisms of modular operads of the form $\Omega \C[[\hbar]] \to\P^\hbar$, which contains the case of
 the moduli space
of gauge equivalence classes of quantum homotopy cohomological field theories.
\end{theorem}

\begin{proof}
The first point is a direct corollary of \cref{thm:GRT1}, \cref{cor:MorphBCHhomo}, and \cref{prop:GraAction}.
For the second and third points, one has to further use respectively \cref{prop:QuantMorphTw} and \cref{prop:QCohFTMC}.
\end{proof}

\begin{example}
In the case of the endomorphism modular operad associated to the
odd affine symplectic manifold $A=V\oplus s^{-1} V^*$, see \cref{ex:TPoly} and \cref{ex:CGRABV},  Maurer--Cartan elements of the quantum totalisation algebra are 
\emph{actions} in
the Batalin--Vilkovisky formalism and each of them  endows $A$ with a \emph{quantum BV manifold structure}. So \cref{thm:GRT1Action} recovers the main result of~\cite{MW14} and establishes an action of the Grothendieck--Teichm\"uller group $\GRT_1$ on the moduli spaces of
formal quantum BV manifolds.
\end{example}

\begin{example}
We would like to emphasize the highly interesting case of the deformation theory of the identity map of the homology modular operad $H_\bullet\big(\oM\big)$ of the Deligne--Mumford--Knudsen modular operad.
This latter one is controlled by the convolution shifted $\Delta$-Lie algebra
$\widehat{\Hom}\left(\B H_\bullet\big(\oM\big), H_\bullet\big(\oM\big)\right)$ ,
which is also a $\Gra$-algebra; its underlying space is
\[\prod_{(g,n)\in \NNs} \left(\mathbb{G}_\od\big(\Hbul\big(\oM\big)\big)_g(n)\otimes H_\bullet\big(\oM_{g,n}\big)\right)^{\Sy_n}\ .\]
A shifted version of this deformation complex should encode the deformation theory of the modular operad structure on the homology $H\big(\oM\big)$ of Deligne--Mumford--Knudsen modular operad. Since this latter of is a formal modular operad~\cite{GNPR05}, such a chain complex will control the deformation theory of the Deligne--Mumford--Knudsen modular operad $\oM$ itself.
This would bring as close to Grothendieck's original approach, see the introduction, and it should  lead to a better understanding of the  relationship between Drinfeld's approach of the Grothendieck--Teichm\"uller group $\GRT_1$ in terms of braided monoidal categories and Grothendieck's  approach. This will be studied in future work.
\end{example}

The purpose of the introduction of $\hbar$ and quantum versions  was \emph{a priori} to control possible divergences due to the actions of some infinite series of graphs with the same number of vertices. The following concrete implementation of the proof of \cref{thm:GRT1} shows  that this does not seem to be always mandatory. (One other purpose of this lemma is to allow us to recall from~\cite{MW14}  the construction of representatives of cohomology classes in the complex $\aa_\Gra^{\theta}$ from the ones in the complex $\b_\Gra^{\vartheta}$).

\medskip

\begin{lemma}\label{lem:finite}\leavevmode
\begin{enumerate}
\item Any representative of a degree $0$ cohomological class of $\b_\Gra^{\vartheta}$ which a linear combination of graphs induces a representative of the associated degree $0$ cohomological class of $\aa_\Gra^{\theta}$ which is a polynomial in $\hbar$ with linear combinations of graphs as coefficients.

\item For any $n\geqslant 1$, the component of arity $n$ of the pre-Lie exponential of any such representatives is a polynomial in $\hbar$.
\end{enumerate}
\end{lemma}

\begin{proof}\leavevmode
\begin{enumerate}
\item Let us first repeat the general arguments of~\cite[Remark~2.2]{MW14} which construct representatives of  the homology of
$\aa_\Gra^{\theta}$ from representatives of the homology of $\b_\Gra^{\vartheta}$.
We denote by
\[\omega\coloneq \vcenter{\hbox{\begin{tikzpicture}[scale=0.5]
	\draw[thick]  (-0.4,2.7) arc [radius=0.4, start angle=180, end angle= 0];
	\draw[thick] (0, 2) to [out=135,in=270] (-0.4,2.7);
	\draw[thick] (0, 2) to [out=45,in=270] (0.4,2.7);
 	 \draw[fill=white, thick] (0,2) circle [radius=10pt];
	\node (n) at (0,1.5) {};
\end{tikzpicture}}}\ \]
 the Maurer--Cartan element of $\aa_\Gra$ such that
$\theta=\vartheta+\omega\hbar$\ .
Let $\sigma^{(0)}$ be a representative of cohomological class of degree $0$ in $\b_\Gra^{\vartheta}$. In $\aa_\Gra^{\theta}$, since
$\partial^\vartheta$ and $\partial^\omega$ commute up to sign,
$\partial^\omega\left(
\sigma^{(0)}
\right)$ is a $-1$-cocycle in $\b_\Gra^{\vartheta}$\ .
Since $H^{-1}\left(\b_\Gra^{\vartheta} \right)=0$  by~\cite[Proposition~3.4 and Theorem~1.1]{Willwacher15},
there exists $\sigma^{(1)}$ of cohomological degree $-2$ in $\b_\Gra^{\vartheta}$ such that
$\partial^\vartheta\left(\sigma^{(1)}\right)=-\partial^\omega\left(\sigma^{(0)}\right)$\ .
This gives
\[\partial^\theta\left(\sigma^{(0)}+\sigma^{(1)}\hbar\right) = \partial^\omega\left(\sigma^{(1)}\right)\hbar^2\ .
\]
By induction, let us now suppose that there exist $\sigma^{(2)},\ldots, \sigma^{(n)}$ in $\b_\Gra^{\vartheta}$ of respective cohomological degree $-4, \ldots, -2n$ and such that
\[\partial^\theta\left(
\sigma^{(0)}+\sigma^{(1)}\hbar+\cdots + \sigma^{(n)}\hbar^n
\right)=\partial^\omega\left(\sigma^{(n)}\right)\hbar^{n+1}\ .
\]
This induces that $\partial^\omega\left(
\sigma^{(n)}
\right)$ is a $-2n-1$-cocycle in $\b_\Gra^{\vartheta}$. By the same argument,
there exists $\sigma^{(n+1)}$ of cohomological degree $-2n-2$ in $\b_\Gra^{\vartheta}$ such that $\partial^\vartheta\left(\sigma^{(n+1)}\right)=-\partial^\omega\left(\sigma^{(n)}\right)$\ .
In the end, we get a representative
\[\sigma^\hbar\coloneq \sigma^{(0)}+\sigma^{(1)}\hbar+\cdots + \sigma^{(n)}\hbar^n+\cdots \]
in $\aa_\Gra^{\theta}$ corresponding to the cohomology class represented by
$\sigma^{(0)}$ in $\b_\Gra^{\vartheta}$~.

Since this procedure decreases the number of vertices of the graphs, when $\sigma^{(0)}$ is a linear combination of graphs, the representative $\sigma^\hbar$ is a polynomial in $\hbar$ whose coefficients are finite linear combinations of graphs.

\item Notice that $\Gra(1)$ is two-dimensional and spanned by the two one-vertex graphs without any edge and with a tadpole respectively. The tadpole graph $\omega$ corresponding to this latter graph in $\aa_\Gra^{\theta}$ cannot appear in any representative $\sigma^\hbar$ since it has an odd number of vertices. The one-vertex graph without any edge in $\aa_\Gra^{\theta}$ cannot appear too since its image under the differential $\partial^\vartheta$ is equal to $\vcenter{\hbox{\begin{tikzpicture}[scale=0.5]
	 \draw[thick] (0,0)--(1.5, 0);

	 \draw[fill=white, thick] (0,0) circle [radius=10pt];
	 \draw[fill=white, thick] (1.5,0) circle [radius=10pt];
	\end{tikzpicture}}}$~, which does not live in the image of the differential $\partial^\omega$~.
In the end, any representative $\sigma^\hbar$ produced by the procedure described above involves a finite number of graphs, each with at least 2 vertices.
Therefore, for any $n\geqslant 1$, the component of arity $n$ of its pre-Lie exponential is a polynomial in $\hbar$.
\end{enumerate}
\end{proof}

Let us recall the Deligne--Drinfeld--Ihara conjecture which states that the Grothendieck--Teichm\"uller Lie algebra $\grt_1$ is isomorphic to the following free complete Lie algebra
\[\widehat{\mathrm{Lie}}(\sigma_3, \sigma_5, \sigma_7, \ldots)\cong \grt_1\ .\]
The most important result in this direction (and for us) is  that the former
embeds
into the latter, see~\cite{Brown12}.
Representatives, still denoted  $\sigma_{2m+1}$, for the homology classes corresponding to these generators
in $\b_\Gra^{\vartheta}$
were given in~\cite{RW14}: they are linear combinations of graphs with $2m+2$ vertices and $4m+2$ edges.
The first one is actually equal to the tetrahedron:
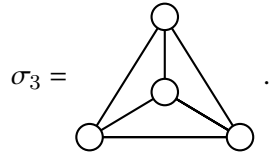
\[\sigma_3=	\vcenter{\hbox{\begin{tikzpicture}[scale=0.5]
	 \draw[thick] (0,0)--(2, -1.2)--(0,2)--(-2,-1.2)--(2,-1.2)--(0,0);
	 \draw[thick] (0,0)--(2,-1.2)--(-2,-1.2)--(0,0)--(0,2);
	 \draw[fill=white, thick] (0,0) circle [radius=10pt];
	 \draw[fill=white, thick] (2,-1.2) circle [radius=10pt];
	 \draw[fill=white, thick] (-2,-1.2) circle [radius=10pt];
	 \draw[fill=white, thick] (0,2) circle [radius=10pt];
	\end{tikzpicture}}}\ .\]
More important here, these representatives are closed with respect to the differential $\partial^\omega$, see
\cite[Theorem~1.7, iv)]{RW14}.
So these linear combinations $\sigma_{2m+1}$ of graphs are also representatives of the same generators  in
$H^0\Big(
\aa_\Gra^{\theta}\Big)$~. Their action on any solution to the quantum master equation  $\d\alpha+\hbar\Delta\alpha+\frac12 \{\alpha, \alpha\}=0$ of any quantum shifted Lie algebra $\gth$ is given by
\[
\sigma_{2m+1}\cdot \alpha =
\sum_{n\geqslant 1}
{\textstyle \frac{1}{n!}}
\left(e^{\sigma_{2m+1}}\right)_n(\alpha, \ldots, \alpha)
\ .\]

The two formulas for the universal deformation group action
\[\sigma\cdot \alpha=\sum_{\substack{m\geqslant 0 \\ n\geqslant 1}} {\textstyle \frac{1}{m!n!}} \big(\underbrace{
(\cdots((\sigma \star \sigma) \star \sigma) \cdots )\star \sigma
}_{m \ \text{times}\ \sigma}\big)(n) (\underbrace{\alpha, \cdots, \alpha}_{n \ \text{times}\ \alpha})\ , \]
for $\sigma \in \prod_{n\geqslant 1}\Gra(n)^{\Sy_n}$
closed with respect to the two differentials $\partial^\vartheta$ and $\partial^\omega$,
and for the gauge group action
\[\alpha+\sum_{k\geqslant 1} \frac{1}{k!} \left(
\ad_\xi^k(\alpha)+\ad_\xi^{k-1}(\d \xi+\hbar\Delta \xi)\right)\ , \]
for $\xi\in \F_1\g_\od[[\hbar]]$,
are very different in nature. So one expects that the action of the universal deformation group on moduli spaces of gauge equivalence classes is not trivial.
 In order to prove it formally, it is enough to consider
the simplest possible example of strict quantum CohFTs defined in~\cref{subsec:Quantum}.

\begin{proposition}
The action of the tetrahedron on  quantum homotopy CohFTs is in general not gauge trivial.
\end{proposition}

\begin{proof}
Let us consider a one-dimension vector space $A$ with a basis element $a\in A$ and with a scalar product $\langle a, a \rangle = 1$. Consider a quantum CohFT $\alpha$ with the target space $A$ supported only on one-vertex graphs without edges, decorated by $\lambda^\hbar$ (with all leaves decorated by the basis element $a\in A$). The constant term in $\hbar$ of the highest Euler characteristic is
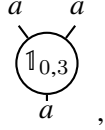
\[
\begin{tikzpicture}[optree, scale=1.1]
    \node{}
      child { node[circ]{$\mathbb{1}_{0,3}$}
        child { node[label=above:$a$]{} edge from parent  }
        child { node[label=above:$a$]{} edge from parent  }
        edge from parent node[label=below:$a$]{}} ;
   \end{tikzpicture}\ ,
   \]
where $\mathbb{1}_{0,3}$ stands for the unit of $H^\bullet\big(\overline{\mathcal{M}}_{0,3}\big)$~.
The lowest genus term constant in $\hbar$ of $\big(\sigma_3\cdot \alpha\big)-\alpha$ is equal to the tetrahedron
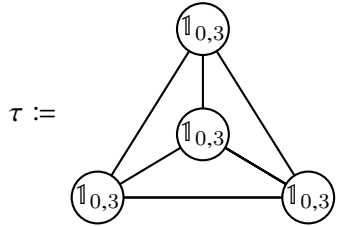
\[\tau\coloneq\vcenter{\hbox{\begin{tikzpicture}[scale=0.7]
	 \draw[thick] (0,0)--(2, -1.2)--(0,2)--(-2,-1.2)--(2,-1.2)--(0,0);
	 \draw[thick] (0,0)--(2,-1.2)--(-2,-1.2)--(0,0)--(0,2);
	 \draw[fill=white, thick] (0,0) circle [radius=14pt];
	 \draw[fill=white, thick] (2,-1.2) circle [radius=14pt];
	 \draw[fill=white, thick] (-2,-1.2) circle [radius=14pt];
	 \draw[fill=white, thick] (0,2) circle [radius=14pt];
 	\node at (0,0) {\scalebox{1}{$\mathbb{1}_{0,3}$}};
 	\node at (2,-1.2) {\scalebox{1}{$\mathbb{1}_{0,3}$}};
 	\node at (-2,-1.2) {\scalebox{1}{$\mathbb{1}_{0,3}$}};
 	\node at (0,2) {\scalebox{1}{$\mathbb{1}_{0,3}$}};
	\end{tikzpicture}}}\]
times a combinatorial coefficient not equal to $0$.
Suppose that the element $\sigma_3\cdot \alpha$ is gauge equivalent to $\alpha$, that is  there exists
$\xi\in \F_1\g_\od[[\hbar]]$ such that
\[\sigma_3\cdot \alpha= \alpha+\sum_{k\geqslant 1} \frac{1}{k!} \left(
\ad_\xi^k(\alpha)+\ad_\xi^{k-1}(\d \xi+\hbar\Delta \xi)\right)\ .\]
Notice that, for any $k\geqslant 1$, all the terms $\ad_\xi^k(x)$ are combinations of graphs with
at least one
      separating edge, that is one edge that would produce two disjoint graphs if removed.
Since the tetrahedron contains no separating edge, the term $\tau$ can only appear in
$\d \xi+\hbar\Delta \xi$, and actually in $\d \xi$, since it carries no $\hbar$ term. The terms of
$\xi$ producing $\tau$ under $\d$ are graphs without any leaves. Since $\d_1$ produces a tadpole
and since $\tau$ does not contain any tadpole, $\tau$ should be produced by $\d_2$. Since the
differential $\d_2$ creates one vertex and one
edge,
the only way to have $\tau$ as a term of $\d_2(\xi)$ is for $\xi$ to contain the graph
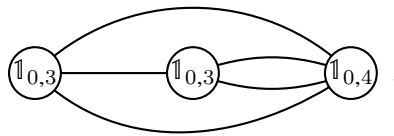
\[\vcenter{\hbox{\begin{tikzpicture}[scale=0.7]
	 \draw[thick] (0,0)--(-3,0);
	\draw[thick] (-3, 0) to [out=45,in=135] (3,0);
	\draw[thick] (-3, 0) to [out=315,in=215] (3,0);
	\draw[thick] (0, 0) to [out=20,in=160] (3,0);
	\draw[thick] (0, 0) to [out=340,in=200] (3,0);
	 \draw[fill=white, thick] (0,0) circle [radius=14pt];
	 \draw[fill=white, thick] (-3,0) circle [radius=14pt];
	 \draw[fill=white, thick] (3,0) circle [radius=14pt];
 	\node at (0,0) {\scalebox{1}{$\mathbb{1}_{0,3}$}};
 	\node at (-3,0) {\scalebox{1}{$\mathbb{1}_{0,3}$}};
 	\node at (3,0) {\scalebox{1}{$\mathbb{1}_{0,4}$}};
	\end{tikzpicture}}}\ , \]
where $\mathbb{1}_{0,4}$ stands for the unit of $H^\bullet\big(\overline{\mathcal{M}}_{0,4}\big)$. Since the respective actions of $\Sy_3$ and $\Sy_4$ on $H^0\big(\overline{\mathcal{M}}_{0,3}\big)$ and $H^0\big(\overline{\mathcal{M}}_{0,4}\big)$  are trivial and since the degree of edges is odd, the two parallel edges in this graph can be switched generating a sign. This shows that this graph is actually equal to $0$ in $\g^\hbar$ and this concludes the proof.
\end{proof}

\subsection{The genus preserving case}\label{subsec:hbar1}
\cref{lem:finite} and the form of the representatives $\sigma_{2m+1}$ of the
conjectural
generators of the Grothendieck--Teichm\"uller Lie algebra $\grt_1$
seem
to indicate that one can consider the case $\hbar=1$.
From the point of view of the convolution algebra, evaluating at $\hbar=1$ means that we are considering the deformation complex of the morphism of operads
$\S \Lie \to \Gra$ given  by
\[\{\, , \} \mapsto
	\vcenter{\hbox{\begin{tikzpicture}[scale=0.7]
	 \draw[thick] (0,0)--(1.5, 0);

	 \draw[fill=white, thick] (0,0) circle [radius=10pt];
	 \draw[fill=white, thick] (1.5,0) circle [radius=10pt];

 	\node at (0,0) {\scalebox{1}{$1$}};
 	\node at (1.5,0) {\scalebox{1}{$2$}};
	\end{tikzpicture}}}\ .\]
One way to
extend this and
take into account the tadpole is to ``internalize'' it: the arity $1$ element
\[\omega\coloneq \vcenter{\hbox{\begin{tikzpicture}[scale=0.7]
	\draw[thick]  (-0.4,2.7) arc [radius=0.4, start angle=180, end angle= 0];
	\draw[thick] (0, 2) to [out=135,in=270] (-0.4,2.7);
	\draw[thick] (0, 2) to [out=45,in=270] (0.4,2.7);
 	 \draw[fill=white, thick] (0,2) circle [radius=10pt];
 	\node at (0,2) {\scalebox{1}{$1$}};
\end{tikzpicture}}}
\]
of the operad $\Gra$ is an operadic Maurer--Cartan element. So one can twist the operad structure with it to produce  the dg operad
\[\Gra^\omega\coloneq \left(\Gra, d^\omega, \{\circ_i\}\right)\ ,\]
see~\cite[Chapter~4, Section~1]{DotsenkoShadrinVallette18}.
Since the element
$\vcenter{\hbox{\begin{tikzpicture}[scale=0.5]
	 \draw[thick] (0,0)--(1.5, 0);

	 \draw[fill=white, thick] (0,0) circle [radius=10pt];
	 \draw[fill=white, thick] (1.5,0) circle [radius=10pt];

 	\node at (0,0) {\scalebox{0.8}{$1$}};
 	\node at (1.5,0) {\scalebox{0.8}{$2$}};
	\end{tikzpicture}}}$ is closed with respect to the twisted differential $d^\omega$, the morphism of operads referred to above induces another morphism $\varTheta\ : \ \S \Lie \to \Gra^\omega$~.

\begin{lemma}
Let $(\g, \d)$ be a dg $\Gra$-algebra and let $\Delta : \g \to \g$ denote the action of $\omega$ . The same  operations define a dg $\Gra^\omega$-algebra structure on $(\g, \d+\Delta)$ .
\end{lemma}

\begin{proof}
This is a direct application of Example 1.2 and Proposition~1.3 of~\cite[Chapter~4]{DotsenkoShadrinVallette18}. Let $\Psi : \Gra \to \End_{(\g, \d)}$ denote the morphism of operads which corresponds to the dg $\Gra$-algebra structure on $\g$ . The image $\Delta=\Psi(\omega)$ is a Maurer--Cartan element of the endomorphism operad of $\g$ and $\Psi$ induces a morphism of dg operads
\[
\varPsi\ : \ \Gra^\omega \to \End^\Delta_{(\g, \d)}=\End_{(\g, \d+\Delta)}
\ .\]
\end{proof}
In the end, the shifted dg  Lie algebra structure $\varPhi$ on $(\g, \d+\Delta)$ is equal to the composite of the two above morphisms of operads:
\[
\vcenter{\hbox{
\begin{tikzcd}
\S  \Lie
\arrow[rr, "\varPhi"] \arrow[dr, "\varTheta"']
& & \End_{(\g, \d+\Delta)}\ .
 \\
&\Gra^\omega  \arrow[ur, "\varPsi"']  &
\end{tikzcd}
}}
\]
At that point, all the arguments of \cref{subsec:CycleAct} apply \emph{mutatis mutandis}; so we get the following analogous versions of \cref{prop:CycleAction} and \cref{cor:MorphBCHhomo} respectively. Recall from the proof of \cref{thm:GRT1} that we denote by $\vartheta$ the Maurer--Cartan element
$
	\vcenter{\hbox{\begin{tikzpicture}[scale=0.5]
	 \draw[thick] (0,0)--(1.5, 0);

	 \draw[fill=white, thick] (0,0) circle [radius=10pt];
	 \draw[fill=white, thick] (1.5,0) circle [radius=10pt];
	\end{tikzpicture}}}$ in $\b_{\Gra^\omega}$; we denote by $\varphi$ the induced Maurer--Cartan in
	$\b_{\End_{(\g, \d+\Delta)}}$, that is the one corresponding to the shifted Lie bracket.

\begin{proposition}\label{prop:ActionBIS}
There are two group actions
\begin{eqnarray*}
&Z_\ev\Big(
\b_{\Gra^\omega}^{\vartheta}\Big)
\xrightarrow{\varPsi_*}
Z_\ev\Big(
\b_{\End_{(\g, \d+\Delta)}}^\varphi\Big)
 \xrightarrow{e}
\S\Lie_\infty\textrm{-}\mathrm{Aut}\big((\g, \d+\Delta)\big)
\to
\Aut\big(
\MC(\g)
\big)
\quad \text{and} &\\
&\left(H_\ev\Big(
\b_{\Gra^\omega}^{\vartheta}\Big), \BCH,0\right)  \to
\Aut\left(
\calMC(\g)
\right) \qquad &
\end{eqnarray*}
both  functorial in complete $\Gra$-algebras $\Psi : \Gra \to \End_\g$ and
explicitly given by
\[\lambda \cdot \alpha \coloneqq \sum_{n\geqslant 1}
{\textstyle \frac{1}{n!}}
\left(e^{\varPsi(\lambda)}\right)_n(\alpha, \ldots, \alpha)\ .\]
\end{proposition}

\begin{remark}
One obtains the exact same results by considering, from the beginning, the dg Lie subalgebra of $\aa_\Gra$ consisting of ``genus preserving maps'', that series of elements of the form $\gamma \hbar^k$ where $k$ is equal to the genus of $\gamma$.
\end{remark}

This case actually produces a trivial theory, as the following statement shows.

\begin{proposition}\label{prop:h=1Triv}
The homology of the chain complex $\b_{\Gra^\omega}^{\vartheta}$ is one-dimensional and concentrated in odd degree  with representative given by the tadpole
\[H\Big(\b_{\Gra^\omega}^{\vartheta}\Big)\cong
\left[\vcenter{\hbox{\begin{tikzpicture}[scale=0.5]
	\draw[thick]  (-0.4,2.7) arc [radius=0.4, start angle=180, end angle= 0];
	\draw[thick] (0, 2) to [out=135,in=270] (-0.4,2.7);
	\draw[thick] (0, 2) to [out=45,in=270] (0.4,2.7);
 	 \draw[fill=white, thick] (0,2) circle [radius=10pt];
	\node (n) at (0,1.5) {};
\end{tikzpicture}}} \right]\ .
\]
This implies that the group
\[H_\ev\Big(
\b_{\Gra^\omega}^{\vartheta}\Big)\]
is trivial.
\end{proposition}

\begin{proof}
In this proof, we use the homological $\ZZ$-grading.
By a slight abuse of notation, we still denote by $\omega$ the tadpole element in $\b_\Gra$. The twisted dg Lie algebra that we consider is isomorphic to
\[\b_{\Gra^\omega}^{\vartheta}\cong \b_{\Gra}^{\vartheta+\omega}\cong
\left(\prod_{n\geqslant 1}\Gra(n)^{\Sy_n}, \partial^{\vartheta}+\partial^{\omega}\right)\ .\]
Since the two commuting differentials $\partial^{\vartheta}$ and $\partial^{\omega}$ preserve the number of tadpoles, this chain complex is the direct sum $\g\cong N\oplus T$ of its two chain sub-complexes spanned by graphs without tadpoles ($N$) and graphs with at least one tadpole ($T$) respectively. The former is acyclic by~\cite[Corollary~4]{KhorWillZiv2017}. We consider the filtration of $T$ defined by the genus of the graphs: let  $\mathrm{F}_p\, T$ consist of series of graphs of genus greater or equal to $-p$.
The first page $E^0_{p,q}$ of the associated spectral sequence consists of series of  graphs
of genus $-p$ with $-p-q$ edges. Its  differential is equal to $d^0=\partial^\vartheta$. Proposition~3.4 of~\cite{Willwacher15} shows that the homology of $T$ with respect to $\partial^\vartheta$ consists of just one class represented by $\omega$~. In the present terms, this means that $E^1_{p,q}=0$ for all
$p$ and $q$, except for $E^1_{-1,0}=[\omega]$~.
Since $\mathrm{F}_0=T$, this filtration is exhaustive.
Since parallel edges cannot appear due to sign issues, given any number $n$ of vertices, there is an integer $p$ such that all the graphs with $n$ vertices have genus less or equal than $p$.
This shows that the filtration is complete. The vanishing of
the spectral sequence at the second page shows that it is regular. So it converges to its homology by
the Complete Convergence Theorem~\cite[Theorem~5.5.10]{WeibelBook}
and this concludes the proof.
\end{proof}

\section{Givental--Grothendieck--Teichm\"uller group}\label{sec:GGT}

In the previous section, we applied the methods of the operadic deformation theory to the
Kontsevich--Willwacher operad $\Gra$ made up of 
natural operations on the totalisations
$\whP$ of modular operads
to introduce
a universal symmetry group  acting functorially on the moduli space of
Maurer--Cartan elements of $\whP$.
The most interesting example for us here is the convolution algebra $\g_A$
encoding homotopy CohFTs (\cref{subsec:DefCohFT}). In this section, we introduce a bigger operad
$\GGra_A$
of natural operations acting on this more restricted class of algebras.
The elements of the convolution algebra $\g_A$ contain three different types of ingredients:
an underlying graph, elements of $A$, and
cohomology classes of the Deligne--Mumford--Knudsen moduli spaces
$\oM$
 of stable curves.
 As a consequence, we construct the  operad $\GGra_A$ from the connected graphs operad $\Gra$ by adding extra decorations which involve elements of $A$ and extra labels defined using the intersection theory of $\oM$.
 More precisely, recall that there is an extensively studied collection of subalgebras of the cohomology algebras of the moduli spaces of curves closed under the push-forwards with respect to all natural maps: the subalgebras of \emph{tautological classes}, see e.g.~\cite{Faber,ZvonkineSurvey,Schmitt} for a survey. The additive generators of these subalgebras are fully understood: they are represented by the natural strata, which can be considered as the results of the iterative application of the operations $\xi_{ij}$ and $\circ_i^j$, decorated by $\psi$- and $\kappa$-classes in all possible ways. We use these classes
and their
push-forwards with respect to the natural maps to construct the operad $\GGra_A$
which we call
the \emph{tautological graph operad}.
Finally, using the same methods as in \cref{sec:GTgroup} \emph{mutatis mutandis}, we obtain  an even bigger universal
symmetry group
acting on (quantum) homotopy CohFTs.

\medskip

In the case of CohFTs,
our
approach recovers the formulas of the classical Givental group action.
This is quite remarkable since our construction assumes
no \emph{a priori} knowledge of the Givental--Teleman theory.
Moreover, the extension of the Givental group action to \emph{homotopy} CohFTs is new.
The full power of the approach followed here shines through on the level of quantum homotopy CohFTs:
the universal symmetry group arising there is a huge group
which we call
the \emph{Givental--Grothendieck--Teichm\"uller group}
because it contains
both the Givental group \emph{and} the  prounipotent Grothendieck--Teichm\"uller group.
Its complete computation currently seems out of reach and is thus a promising subject for further studies.
As a first step in exploring applications of this group, we were able to express the Buryak--Rossi functor (\cref{subsec:Quantum}) in term of the action of a specific Manin--Zograf element living in the Givental--Grothendieck--Teichm\"uller group; this conceptual interpretation implies nice properties of the Buryak--Rossi functor with respect to Givental type actions.

\begin{remark}
In order to set up the theory developed in this section, we will need to work
under the stronger assumption that the various dg symmetric vector spaces $A$ are finite dimensional and that their pairings are non-degenerate.
The arguments proceed without much change if $A$ is a Hilbert space, and this ensures that interesting infinite-dimensional cases, such as differential forms on a K\"ahler manifold, can be included in this formalism as well.
However, in such cases, there is a nontrivial choice to make in the definition of the unstable components $\EEnd_A(0,2)$ and $\EEnd_A(0,1)$. The same applies if we wish to generalize this construction to the case of a more general target modular operad $\P$: one must find meaningful candidates for the unstable components $\P_0(2)$ and $\P_0(1)$. A more intricate question is how one can generalize the source, that is the cohomology modular cooperad of moduli spaces. We hope to address this question elsewhere.
\end{remark}

\subsection{The tautological graph operad}

Before presenting the new definition, let us start with two general remarks that might clarify our constructions. First of all, the $\Gra$-algebra structure on $\g_A$ uses just the modular operad structure maps. However, a special feature of the convolution algebra $\g_A$, controlling homotopy CohFTs, is the geometric intuition arising from thinking of the arity $n$ in terms of marked points of stable curves. In particular, one has the $\psi$-classes at all individual marked points, and the operators of multiplication by those classes behave well with respect to the natural maps between moduli spaces.
Also, if one allows the operations acting on $\g_A$ to depend on~$A$, some precautions are needed when $A$ has a nontrivial differential: the issue is to ensure that the action of the extended operations are compatible with the differential.
The strategy that we choose is to only consider cycles with respect to the differential of $A$ but one could alternatively have extra differentials arising in some of the formulas below.

\begin{definition}[The extended connected graph operad]
Let $(A, d_A, \langle\ ,\, \rangle)$ be a dg symmetric vector space.
The \emph{extended connected graph operad} $\EGra_A$ is spanned by connected graphs with \emph{usual} edges of homological degree $-1$ and \emph{dashed} edges of homological degree $0$, both including loops, labeled by symmetric elements of $Z_\bullet\left(A^{\otimes 2}\right)\allowbreak[[\psi_1,\psi_2]]$, that is elements invariant under the simultaneous action of the transposition $(1 2)$ on the tensor factors and on the elements $\psi_1$ and $\psi_2$.
Given an edge $e$ between two vertices $v$ and $w$, we denote these labeling elements by $M_{vw}^{(e)}\in Z_\bullet\left(A^{\otimes 2}\right)[[\psi_1,\psi_2]]^{\Sy_2}$, or simply by $M_{vw}$ when there is only one edge between two vertices
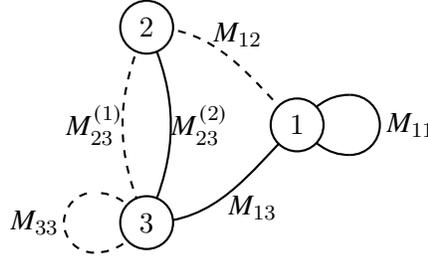
\[ \vcenter{\hbox{\begin{tikzpicture}[scale=1]
  \draw[dashed, thick] (-1, 1.3) to [out=0,in=135] (1,0) ;
  \draw[thick] (-1, -1.3) to [out=0,in=225] (1,0);
  \draw[dashed, thick] (-1, 1.3) to [out=245,in=115] (-1,-1.3);
  \draw[thick] (-1, 1.3) to [out=295,in=65] (-1,-1.3);

  \draw[thick]  (1.7,-0.4) arc [radius=0.4, start angle=270, end angle= 450];
  \draw[thick] (1, 0) to [out=315,in=180] (1.7,-0.4);
  \draw[thick] (1, 0) to [out=45,in=180] (1.7,0.4);

  \draw[dashed, thick]  (-1.7,-1.7) arc [radius=0.4, start angle=270, end angle= 90];
  \draw[dashed, thick] (-1, -1.3) to [out=225,in=0] (-1.7,-1.7);
  \draw[dashed, thick] (-1, -1.3) to [out=135,in=0] (-1.7,-0.9);

  \draw[fill=white, thick] (1,0) circle [radius=10pt];
  \draw[fill=white, thick] (-1,1.3) circle [radius=10pt];
  \draw[fill=white, thick] (-1,-1.3) circle [radius=10pt];

  \node at (0.2,1.2) {\scalebox{1}{$M_{12}$}};
  \node at (-1.7,0) {\scalebox{1}{$M_{23}^{(1)}$}};
  \node at (-0.3,0) {\scalebox{1}{$M_{23}^{(2)}$}};
  \node at (-2.5,-1.3) {\scalebox{1}{$M_{33}$}};
  \node at (0.4,-1.1) {\scalebox{1}{$M_{13}$}};
  \node at (2.5,0) {\scalebox{1}{$M_{11}$}};

  \node at (1,0) {\scalebox{1}{$1$}};
  \node at (-1,1.3) {\scalebox{1}{$2$}};
  \node at (-1,-1.3) {\scalebox{1}{$3$}};
    \end{tikzpicture}}}
\]
Like in the case of the connected graph operad $\Gra$, the partial composition products are given by
inserting a graph at a vertex of another one, relabelling the vertices, and considering the sum of all the possible ways to connect the edges, previously attached to the deleted vertex, to the inserted graph.
The homological degree of a labeled graph is equal to the sum of the degrees of the edges plus the degrees of all the labeling elements.
This graded operad has a differential $\dE$ equal to the sum of all the ways to replace a dashed edge labeled by a symmetric element $M_{vw}\in Z_\bullet\left(A^{\otimes 2}\right)[[\psi_1,\psi_2]]^{\Sy_2}$ with a usual edge between the same vertices  labeled by $(\psi_1+\psi_2)M_{vw}$~.
\end{definition}

It is straightforward to see that this defines a dg operad.

\begin{remark}
Notice that parallel edges are a priori allowed here due to their labelings.
However, when $A$ is finite dimensional with a non-degenerate bilinear form $\langle\ ,\, \rangle$,
if we consider a graph made up of usual edges only and decorated by $\mathbb{I}\,\psi_1^0\psi_2^0$, where $\mathbb{I}\in A^{\otimes 2}$ is
the image of the identity map of $A$ under
the pairing $\langle\ ,\, \rangle$, then we cannot have parallel edges for  degree reasons. In this case, we get a graph in $\Gra$; this shows that the connected graph operad embeds naturally into the extended connected graph operad: $\Gra \hookrightarrow \EGra_A$~.
\end{remark}

The next ingredient of our construction amounts to incorporate some new operations of arity one.
The first kind of new arity one operations comes from a particular aspect of the geometry of the moduli spaces of curves. Recall that Manin and Zograf~\cite{ManinZograf2000} considered the vector space $L$ of systems  of classes $\big\{\ell_{g,n}\in \Hbul\big(\oM_{g,n},\mathbb{Q}\big)\big\}_{(g,n)\in \NNs}$ satisfying the properties
\[
\sigma^*\ell_{g,n} = \ell_{g-1,n+2} \quad  \text{and} \quad
\rho^*\ell_{g,n} = \ell_{g',n'+1}\otimes 1 + 1 \otimes \ell_{g-g',n-n'+1}~, \]
where the maps
\[\sigma  =\xi_{n\, n+1}  \colon \oM_{g-1,n+2}\to \oM_{g,n} \quad \text{and} \quad
\rho=\circ_{n'+1}^{n-n'+1} \colon \oM_{g',n'+1}\times \oM_{g-g',n-n'+1}\to\oM_{g,n}
\]
come respectively from the
 the identification of the last two marked points
 and the gluing of
two curves along their last marked points.
 According to \cite[Prop.~8.4]{Teleman2012}, the space $L$ is spanned by the $\kappa$-classes and by the Chern characters of the Hodge bundle $\mathrm{ch}_i$, that is,
 \[
\ell_{g,n} = \left(\sum_{i\geqslant 1} k_i\kappa_i + \sum_{j \geqslant 0} s_{2j+1} \mathrm{ch}_{2j+1}\right)|_{\oM_{g,n}}~,
 \]
for arbitrary choices of  constants $k_i$ and $s_{2j+1}$ in $\mathbb{Q}$. We consider the abelian Lie algebra $L_{MZ}$
made up of \emph{Manin--Zograf elements} that are sums $\ell=E+F$ of series in $e_i$, for $i\geqslant 1$, and in $f_{2j+1}$, for $j\geqslant 0$, respectively:
\[L_{MZ} \coloneq \left\{\ell=E+F=
\sum_{i \geqslant  1} k_ie_i + \sum_{j \geqslant 0} s_{2j+1}f_{2j+1}\ , \ k_i, s_{2j+1} \in \K
\right\}~.\]
\\

The second kind of unary operations correspond to the unstable components of the modular operad $\EEnd_A$, though still mixed with extra datum coming from the geometry of moduli spaces of curves.

\begin{definition}[The Lie algebra of unstable unary operations]
The \emph{Lie algebra of unstable unary operations} $L_{\mathrm{unst}}$ is the graded vector space $\psi Z_\bullet\left(A^{\otimes 2}\right)[[\psi]]\oplus \psi^2Z_\bullet(A)[[\psi]]$ equipped with the standard Lie bracket
 \[
[(R_1,T_1),(R_2,T_2)]\coloneq
\left(R_2\circ_2^1 R_1-(-1)^{|R_1||R_2|}R_1\circ_2^1 R_2,
-R_1\circ_2^1 T_2+(-1)^{|R_2||T_1|}R_2\circ_2^1 T_1\right)~,
 \]
induced by the modular operad structure on $\EEnd_A[[\psi]]$~.
\end{definition}

\begin{remark} \label{rem:relaxRT}
The divisibility by $\psi$ in the first summand and by $\psi^2$ in the second summand do not appear relevant for the definition of the Lie algebra, but becomes crucial for convergence of some of the formulas below, most notably that of Proposition \ref{prop:translation}.

In principle, one could still include constant terms with respect to $\psi$ in the first summand and degree one terms in $\psi$ in the second summand without any problem, but it would change some of the homology computations below, see~\cref{rem:ConstUnstable}.
\end{remark}

Finally, we note that the formula
\begin{equation}\label{eq:ManinZograf}
\ell\cdot(R,T)\coloneq
\big(0,-E(\psi)T \big)~, \text{where}\  \ell=E+F\ \text{and}\ E(\psi)\coloneq \sum_{i \geqslant  1} k_i\psi^i~,
\end{equation}
defines an action of $L_{MZ}$ on the Lie algebra $L_{\mathrm{unst}}$ by pairwise commuting derivations.
This action defines a Lie algebra structure on
 \[
\mathfrak{L}\coloneq L_{MZ}\oplus L_{\mathrm{unst}}~.
 \]
Viewing its universal enveloping algebra $\mathrm{U}(\mathfrak{L})$ as an operad concentrated in arity one, we are now ready to define the protagonist of this section.

\begin{definition}[Tautological graph operad]
The \emph{tautological graph operad} $\GGra_A$ is the quotient of the coproduct of operads $\EGra_A\vee \mathrm{U}(\mathfrak{L})$ by the operadic commutator relations
\begin{gather}\label{eq:DistLawTautGra}
\big[(\ell, R,T),\Gamma\big]=
(\ell, R,T)\circ_1 \Gamma - (-1)^{|(\ell, R,T)||\Gamma|}
\sum_{i=1}^n
\Gamma\circ_i (\ell, R,T)=
R\cdot \Gamma~,
\end{gather}
for $\Gamma\in \EGra_A(n)$,
where $R\cdot\Gamma$ is equal to
the sum over all edges (usual and dashed) connecting two vertices $v$ and $w$ (possibly equal to each other)
of decorated graphs obtained by replacing the decoration $M_{vw}$ by
$-R(\psi_1)\circ_2^1 M_{vw} - (-1)^{|M_{vw}||R|} M_{vw}\circ_2^1 R^{(12)}(\psi_2)$~. Since $R\cdot
(-)$ is an operadic derivation of $\EGra_A$ that commutes with $\dE$, for any $(\ell, R,T)\in
\mathfrak{L}$, this latter differential induces a well-defined differential on $\GGra_A$ that we
denote by $\DE$.

\end{definition}

\begin{notation}
We will use the following suggestive graphical convention to depict the unary elements of the operad $\mathrm{U}(\mathfrak{L})\subset\GGra$ corresponding to the unstable and the Manin--Zograf operations respectively:
\[\ell \longleftrightarrow
\vcenter{\hbox{\begin{tikzpicture}[scale=0.6]
        \draw[fill=white, thick] (0,0) circle [radius=10pt];
        \node at (0,0) {\scalebox{1}{$\ell$}};
        \end{tikzpicture}}}
\ , \quad
R\longleftrightarrow
\sum_{k\geqslant 1}\,
\vcenter{\hbox{\begin{tikzpicture}[scale=0.6]
        \draw[fill=white, thick] (0,0.3) circle [radius=22pt];
        \draw[thick] (0,30pt)--(0,20pt);
        \draw[thick] (0,30pt)--(0,40pt);
        \node at (0,10pt) {\scalebox{1}{$a_1^k\psi^k$}};
        \node at (0,53pt) {\scalebox{1}{$a_2^k$}};
        \end{tikzpicture}}}
\ , \quad
T\longleftrightarrow
\sum_{k\geqslant 2}\,
\vcenter{\hbox{\begin{tikzpicture}[scale=0.6]
        \draw[fill=white, thick] (0,0) circle [radius=22pt];
        \draw[thick] (0,22pt)--(0,12pt);
        \node at (0,2pt) {\scalebox{1}{$a^k\psi^k$}};
        \end{tikzpicture}}}
        \ \  .
\]
\end{notation}

In the finite dimensional and non-degenerate case, the tautological graph operad contains the connected graph operad
\[\Gra \hookrightarrow\EGra_A \hookrightarrow \GGra_A~.\]
The following theorem is a central result which justifies the definition of the tautological operad  $\GGra_A$:
 it gives a precise description of its underlying dg $\mathbb{S}$-module and it shows that it is an operad made up of natural operations, extending the ones of $\Gra$ of \cref{prop:GraAction}, acting on the convolution algebra $\g_A$ controlling the deformation theory of homotopy cohomological field theories.

\begin{theorem}\label{th:GGra0}\leavevmode
\begin{enumerate}
\item The relation~\eqref{eq:DistLawTautGra} of $\GGra_A$ defines a distributive law between $\EGra_A$ and $\mathrm{U}(\mathfrak{L})$; in particular, the underlying dg $\mathbb{S}$-module of $\GGra_A$ is isomorphic to $\EGra_A\circ \mathrm{U}(\mathfrak{L})$.

\item  Convolution algebras $\g_A$ admit a functorial complete $\GGra_A$-algebra structure. It extends the  complete $\Gra$-algebra structure, when $A$ is finite dimensional with a non-degenerate pairing.
\[
\vcenter{\hbox{
\begin{tikzcd}
\Gra
\arrow[rr, "\Psi"] \arrow[dr, hook]
& & \End_{\g_A}
 \\
&\GGra  \arrow[ur]  &
\end{tikzcd}
}}
\]
\end{enumerate}
\end{theorem}

\begin{proof}
The first assertion is a particular case of the following general property. Let $\mathfrak{L}$ be a Lie algebra
acting on an operad $\PP$ by operadic derivations, i.e. there is a morphism of Lie algebras
$\rho \colon \mathfrak{L} \to \mathrm{Der}(\PP)$.  We consider the operad $\mathrm{O}$ defined as the quotient of the coproduct $\PP\vee \mathrm{U}(\mathfrak{L})$ by the operadic ideal generated by the relations
\begin{gather}\label{eq:LieRelationDer}\tag{$\ast$}
\big[X, \mu\big]=
X\circ_1 \mu - (-1)^{|X||\mu|}
\sum_{i=1}^n
\mu\circ_i X=
\rho(X)(\mu)~,
\end{gather}
for $X\in \mathfrak{L}$ and $\mu\in \PP(n)$~.
We claim that this relation~\eqref{eq:LieRelationDer} induces a distributive law between $\PP$ and $\mathrm{U}(\mathfrak{L})$.
To this end, we use the straightforward extension of \cite[Proposition~8.6.2]{LodayVallette12} to quadratic-linear presentations of operads.
So we just need to prove that the underlying $\Sy$-module of the operad $\mathrm{O}$ is isomorphic to the composite product $\PP\circ \mathrm{U}(\mathfrak{L})$.
For this, we use the methods of rewriting systems and operadic Gr\"obner bases \cite{BremnerDotsenko16}.
Both $\PP$ and $\mathrm{U}(\mathfrak{L})$ admit canonical quadratic-linear presentations.
For the operad $\PP$, one may choose any linear basis $\PP$ as a set of generators and the multiplication table for relations.
Such a presentation leads to a Gr\"obner basis for any ordering extending the weight grading.
For the enveloping algebra $\mathrm{U}(\mathfrak{L})$, one takes a basis of the vector space $\mathfrak{L}$ as a set of generators, in which case the Gr\"obner basis property for the degree-lexicographic ordering follows from the Jacobi identity.
This implies that the associated presentation of the coproduct $\PP\vee \mathrm{U}(\mathfrak{L})$ leads to a terminating and confluent rewriting system. We choose to interpret relations \eqref{eq:LieRelationDer} as a rewriting rule once read from left to right:
\[X\circ_1 \mu  \mapsto (-1)^{|X||\mu|}\sum_{i=1}^n\mu\circ_i X+\rho(X)(\mu)~. \]
The application of any rewriting rule of the operad $\mathrm{O}$ decreases strictly the lexicographic order on the tuples given by the weight grading, the degree-lexicographic orderings of words of elements from $\mathfrak{L}$ read from the leaves to the root, and minus the sum over the elements of $\mathfrak{L}$ of the distance from that element  to the root.
Since this order is bounded below, the overall rewriting system is terminating, and,
since the Lie algebra $\mathfrak{L}$ is acting on the operad $\PP$ by operadic derivations, it is confluent.  This concludes the fact that we do have a distributive law in this general context.
Finally, it is enough to check that the relations~\eqref{eq:DistLawTautGra} defining the operad $\mathrm{O}=\GGra_A$ is given by an action of the Lie algebra $\mathfrak{L}=L_{MZ}\oplus L_{\mathrm{unst}}$ by operadic derivations of $\PP=\EGra_A$. \\

Regarding the second assertion, we need to define universal actions of decorated graphs and the new unary operations on $\mathfrak{g}_A$, that is two morphisms of dg operads $\EGra_A \to \End_{\g_A}$ and $\mathrm{U}(\mathfrak{L}) \to \End_{\g_A}$~, such that the images of relations~\eqref{eq:DistLawTautGra} vanish under the induced morphism $\EGra_A\vee \mathrm{U}(\mathfrak{L})  \to \End_{\g_A}$~.

\medskip

The first morphism $\EGra_A \to \End_{\g_A}$ is defined as follows. Given $\Gamma\in \EGra_A(n)$ and $\gamma_1, \ldots, \gamma_n \in \g_A$, we first insert every graph $\gamma_i$ at the vertex $i$ of $\Gamma$, for every $i=1, \ldots, n$ and we consider the sum of all the ways to connect the edges (usual or dashed) attached to each vertex $i$ to the the leaves of $\gamma_i$. (When the number of adjacent edges is strictly greater than the number of leaves of the graph from $\g_A$, the result is equal to $0$.) There are now two possibilities.

\medskip

\textsc{Usual edges.} When two vertices are are linked by a usual edge, we keep it but consider the following coefficients and labelings:
\[
    \vcenter{\hbox{\begin{tikzpicture}[scale=1]
        \draw[thick] (0,0)--(0.65, 0);
        \draw[thick] (1,0)--(3, 0);
        \draw[thick] (3.35,0)--(4, 0);

        \draw[thick] (-0.65, 0)--(0,0);
        \draw[thick] (-0.47, -0.47)--(0,0);
        \draw[thick] (-0.47, 0.47)--(0,0);
        \draw[thick] (4.47, -0.47)--(4,0);
        \draw[thick] (4.47, 0.47)--(4,0);

        \draw[fill=white, thick] (0,0) circle [radius=10pt];
        \draw[fill=white, thick] (4,0) circle [radius=10pt];

        \node at (0,0) {\scalebox{1}{$\mu$}};
        \node at (0.49,0.15) {\scalebox{0.8}{$i$}};
        \node at (0.83,0) {\scalebox{1}{$a$}};
        \node at (3.22,0.05) {\scalebox{1}{$a'$}};
        \node at (3.55,0.17) {\scalebox{0.8}{$j$}};
        \node at (4.05,0.05) {\scalebox{1}{$\mu'$}};
        \node at (2,0.25) {\scalebox{1}{$M$}};
        \end{tikzpicture}}}
        \quad
\longmapsto
\quad
\sum_{k,l\geqslant 0} \left\langle a, a_1^{kl}\right\rangle \left\langle a_2^{kl}, a'\right\rangle \ \,
    \vcenter{\hbox{\begin{tikzpicture}[scale=1]
        \draw[thick] (0,0)--(2, 0);

        \draw[thick] (-0.8, 0)--(0,0);
        \draw[thick] (-0.58, -0.58)--(0,0);
        \draw[thick] (-0.58, 0.58)--(0,0);
        \draw[thick] (2.58, -0.58)--(2,0);
        \draw[thick] (2.58, 0.58)--(2,0);
        \draw[fill=white, thick] (0,0) circle [radius=15pt];
        \draw[fill=white, thick] (2,0) circle [radius=15pt];

        \node at (0,0.02) {\scalebox{1}{$\mu\,\psi^k_i$}};
        \node at (1.98,0.05) {\scalebox{1}{$\mu'\psi^l_j$}};
\end{tikzpicture}}}~,
\]
where $i$ (respectively $j$) is the label of both the external edge and the corresponding point of the moduli space of curves and where
\[M=\sum_{k,l\geqslant 0} \left(a_1^{kl}\otimes a_2^{k}\right) \psi_1^k\psi_2^l~.\]
When the usual edge is a loop, the definition is the same except that we now label the vertex by the cohomology class $\mu\,\psi^k_i\,\psi^l_j$~.

\medskip

\textsc{Dashed edges.} When two vertices are are linked by a dashed edge, we contract it and consider the following coefficients and labels:
\[
    \vcenter{\hbox{\begin{tikzpicture}[scale=1]
        \draw[thick] (0,0)--(0.65, 0);
        \draw[thick, dashed] (1,0)--(3, 0);
        \draw[thick] (3.35,0)--(4, 0);

        \draw[thick] (-0.65, 0)--(0,0);
        \draw[thick] (-0.47, -0.47)--(0,0);
        \draw[thick] (-0.47, 0.47)--(0,0);
        \draw[thick] (4.47, -0.47)--(4,0);
        \draw[thick] (4.47, 0.47)--(4,0);

        \draw[fill=white, thick] (0,0) circle [radius=10pt];
        \draw[fill=white, thick] (4,0) circle [radius=10pt];

        \node at (0,0) {\scalebox{1}{$\mu$}};
        \node at (0.49,0.15) {\scalebox{0.8}{$i$}};
        \node at (0.83,0) {\scalebox{1}{$a$}};
        \node at (3.22,0.05) {\scalebox{1}{$a'$}};
        \node at (3.55,0.17) {\scalebox{0.8}{$j$}};
        \node at (4.05,0.05) {\scalebox{1}{$\mu'$}};
        \node at (2,0.25) {\scalebox{1}{$M$}};
        \end{tikzpicture}}}
        \quad
\longmapsto
\quad
\sum_{k,l\geqslant 0} \left\langle a, a_1^{kl}\right\rangle \left\langle a_2^{kl}, a'\right\rangle \ \,
    \vcenter{\hbox{\begin{tikzpicture}[scale=1]

        \draw[thick] (-0.65, 0)--(0,0);
        \draw[thick] (-0.47, -0.47)--(0,0);
        \draw[thick] (-0.47, 0.47)--(0,0);
        \draw[thick] (0.47, -0.47)--(0,0);
        \draw[thick] (0.47, 0.47)--(0,0);

        \draw[fill=white, thick] (0,0) circle [radius=10pt];

        \node at (0,0) {\scalebox{1}{$\nu$}};

\end{tikzpicture}}}\ \ ,
\]
where
\[M=\sum_{k,l\geqslant 0} \left(a_1^{kl}\otimes a_2^{kl}\right) \psi_1^k\psi_2^l
\quad
\text{and} \quad
\nu=\left(\circ_i^j\right)_*\left(\mu\,\psi^k_i \otimes \mu'\psi^l_j\right)~, \]
with $\left(\circ_i^j\right)_*$ the push-forward map on cohomology, see \cite[Section~1.1]{DotsenkoShadrinVallette11}.
When the dashed edge is a loop, the definition is the same except that we now label the vertex by the cohomology class  $\left(\xi_{ij}\right)_*\left(\mu\,\psi^k_i \psi^l_j\right)$~.

\medskip

Since the push-forwards $\left(\circ_i^j\right)_*$ and $\left(\xi_{ij}\right)_*$ of sewing maps commute with each other and commute with the multiplication by $\psi_k$, with $k$ different from $i,j$ (the latter is obvious from the projection formula and the fact $\sigma^*\psi_k = \psi_k$ and $\rho^*\psi_k = \psi_k$, $k=1,\dots,n$), the actions of usual and dashed edges (including loops) commute and the map $\EGra_A  \to \End_{\g_A}$ is well-defined and is a morphism of operads.

\medskip

Since the actions of edges commute, we just need to check the compatibility with the respective differentials for the action of one edge solely, for every type of edges.
Let us use the following notations for the operations on $\g_A$ obtained as images under the morphism $\EGra_A  \to \End_{\g_A}$ of the following simple graphs:
\begin{align*}
\vcenter{\hbox{\begin{tikzpicture}[scale=0.7]
	\draw[thick]  (-0.4,2.7) arc [radius=0.4, start angle=180, end angle= 0];
	\draw[thick] (0, 2) to [out=135,in=270] (-0.4,2.7);
	\draw[thick] (0, 2) to [out=45,in=270] (0.4,2.7);
 	 \draw[fill=white, thick] (0,2) circle [radius=10pt];
 	\node at (0,2) {\scalebox{1}{$1$}};
	\node (n) at (0,1.5) {};
  \node at (0,3.4) {\scalebox{1}{$M$}};
\end{tikzpicture}}}
 \ \,  \longmapsto \ \Delta_M
 \ \ , \quad
 	\vcenter{\hbox{\begin{tikzpicture}[scale=0.7]
	 \draw[thick] (0,0)--(2, 0);
	 \draw[fill=white, thick] (0,0) circle [radius=10pt];
	 \draw[fill=white, thick] (2,0) circle [radius=10pt];
 	\node at (0,0) {\scalebox{1}{$1$}};
 	\node at (2,0) {\scalebox{1}{$2$}};
  \node at (1,0.3) {\scalebox{1}{$M$}};
  	\node (n) at (1,-0.5) {};
	\end{tikzpicture}}}
\ \,  \longmapsto \ \{\, , \}_M
\ \  , \quad
\vcenter{\hbox{\begin{tikzpicture}[scale=0.7]
	\draw[thick, dashed]  (-0.4,2.7) arc [radius=0.4, start angle=180, end angle= 0];
	\draw[thick, dashed] (0, 2) to [out=135,in=270] (-0.4,2.7);
	\draw[thick, dashed] (0, 2) to [out=45,in=270] (0.4,2.7);
 	 \draw[fill=white, thick] (0,2) circle [radius=10pt];
 	\node at (0,2) {\scalebox{1}{$1$}};
	\node (n) at (0,1.5) {};
  \node at (0,3.4) {\scalebox{1}{$M$}};
\end{tikzpicture}}}
 \ \,  \longmapsto \ ^h\!\Delta_{M}
 \ \ , \quad
 	\vcenter{\hbox{\begin{tikzpicture}[scale=0.7]
	 \draw[thick, dashed] (0,0)--(2, 0);
	 \draw[fill=white, thick] (0,0) circle [radius=10pt];
	 \draw[fill=white, thick] (2,0) circle [radius=10pt];
 	\node at (0,0) {\scalebox{1}{$1$}};
 	\node at (2,0) {\scalebox{1}{$2$}};
  \node at (1,0.3) {\scalebox{1}{$M$}};
  	\node (n) at (1,-0.5) {};
	\end{tikzpicture}}}
\ \,  \longmapsto \ ^h\!\{ \, , \}_M
\ \  .
\end{align*}
Since the partial decomposition maps $\delta_i^j$ and $\chi_{ij}$ of the cohomology modular cooperad $\Hbul\big(\oM\big)$ commute with the multiplication by
$\psi_k$, with $k$ different from $i,j$ (which, again, follows from the fact that $\sigma^*\psi_k = \psi_k$ and $\rho^*\psi_k = \psi_k$, $k=1,\dots,n$),
we have
\[\left[\d\,, \Delta_{M}\right] =0
\quad\text{and} \quad
\left[\d\,,\{ \, , \}_M\right] =0\ ,
\]
so  the actions of usual edges commute with the differential $\d=\d_A-\d_1-\d_2$ of $\g$.
The excess intersection formula \cite[Section~A.4, Equation~(11)]{GP03}
 implies the following chain homotopy relations:
\[
\left[\d\, ,\,^h\!\Delta_{M}\right] = \Delta_{(\psi_1+\psi_2)M}
\quad\text{and} \quad
\left[\d\, ,\,^h\!\{ \, , \}_M\right] = \{ \, , \}_{(\psi_1+\psi_2)M}~.
\]
These relations imply precisely that the map $\EGra_A  \to \End_{\g_A}$ is a morphism of dg operads.

\medskip

Let us now define the actions of the unstable and the Manin--Zograf unary operations: $\mathrm{U}(\mathfrak{L}) \to \End_\g$~.

\medskip

\textsc{Unstable unary operations.} The action of an unstable operation
\[\sum_{k\geqslant 1}\,
\vcenter{\hbox{\begin{tikzpicture}[scale=0.6]
        \draw[fill=white, thick] (0,0.3) circle [radius=22pt];
        \draw[thick] (0,30pt)--(0,20pt);
        \draw[thick] (0,30pt)--(0,40pt);
        \node at (0,10pt) {\scalebox{1}{$a_1^k\psi^k$}};
        \node at (0,53pt) {\scalebox{1}{$a_2^k$}};
        \end{tikzpicture}}} \]
of $\psi Z_\bullet(A^{\otimes 2})[[\psi]]$ on a graph $\gamma\in \g$ amounts to summing, over all the leaves of $\gamma$, the graphs obtained by modifying the vertex attached to it as follows
\[
    \vcenter{\hbox{\begin{tikzpicture}[scale=1]
        \draw[thick] (0,0)--(0.65, 0);
        \draw[thick] (-0.65, 0)--(0,0);
        \draw[thick] (-0.47, -0.47)--(0,0);
        \draw[thick] (-0.47, 0.47)--(0,0);

        \draw[fill=white, thick] (0,0) circle [radius=10pt];

        \node at (0,0) {\scalebox{1}{$\mu$}};
        \node at (0.49,0.15) {\scalebox{0.8}{$i$}};
        \node at (0.83,0) {\scalebox{1}{$a$}};
        \end{tikzpicture}}}
        \quad
\longmapsto
\quad
\sum_{k\geqslant 1} \left\langle a, a_1^{k}\right\rangle \ \,
    \vcenter{\hbox{\begin{tikzpicture}[scale=1]

        \draw[thick] (0,0)--(0.8, 0);
        \draw[thick] (-0.8, 0)--(0,0);
        \draw[thick] (-0.58, -0.58)--(0,0);
        \draw[thick] (-0.58, 0.58)--(0,0);

        \draw[fill=white, thick] (0,0) circle [radius=15pt];

        \node at (0,0.02) {\scalebox{1}{$\mu\,\psi^k_i$}};
        \node at (0.67,0.15) {\scalebox{0.8}{$i$}};
        \node at (1.05,0) {\scalebox{1}{$a_2^k$}};
\end{tikzpicture}}}\ \ .
\]
The action of an unstable operation
\[
\sum_{k\geqslant 2}\,
\vcenter{\hbox{\begin{tikzpicture}[scale=0.6]
        \draw[fill=white, thick] (0,0) circle [radius=22pt];
        \draw[thick] (0,22pt)--(0,12pt);
        \node at (0,2pt) {\scalebox{1}{$a^k\psi^k$}};
        \end{tikzpicture}}}
\]
of $\psi^2 Z_\bullet(A)[[\psi]]$ on a graph $\gamma\in \g$ amounts to summing, over all the leaves of $\gamma$, the graphs obtained by modifying the vertex attached to it as follows
\[
    \vcenter{\hbox{\begin{tikzpicture}[scale=1]
        \draw[thick] (0,0)--(0.65, 0);
        \draw[thick] (-0.65, 0)--(0,0);
        \draw[thick] (-0.47, -0.47)--(0,0);
        \draw[thick] (-0.47, 0.47)--(0,0);

        \draw[fill=white, thick] (0,0) circle [radius=10pt];

        \node at (0,0) {\scalebox{1}{$\mu$}};
        \node at (0.49,0.15) {\scalebox{0.8}{$i$}};
        \node at (0.83,0) {\scalebox{1}{$a$}};
        \end{tikzpicture}}}
        \quad
\longmapsto
\quad
\sum_{k\geqslant 2} \left\langle a, a^{k}\right\rangle \ \,
    \vcenter{\hbox{\begin{tikzpicture}[scale=1]
        \draw[thick] (-0.65, 0)--(0,0);
        \draw[thick] (-0.47, -0.47)--(0,0);
        \draw[thick] (-0.47, 0.47)--(0,0);

        \draw[fill=white, thick] (0,0) circle [radius=10pt];

        \node at (0,0) {\scalebox{1}{$\nu$}};
\end{tikzpicture}}}\ \ ,
\]
 with $\nu=\left(\pi_i\right)_*\left(\mu\,\psi^k_i \right)$, where $\pi_i$ is the map that forgets the $i$th point.

 \medskip

 \textsc{Manin--Zograf operations.}
 The action of a Manin--Zograf element
 $\vcenter{\hbox{\begin{tikzpicture}[scale=0.6]
        \draw[fill=white, thick] (0,0) circle [radius=10pt];
        \node at (0,0) {\scalebox{1}{$\ell$}};
        \end{tikzpicture}}}$
of $L_{MZ}$
 on a graph $\gamma\in \g$ amounts to summing, over all the vertices of $\gamma$, the graphs obtained by  multiplying the cohomology class labeling the vertex with the corresponding Manin--Zograf class:
\[
    \vcenter{\hbox{\begin{tikzpicture}[scale=1]
        \draw[thick] (0,0)--(0.65, 0);
        \draw[thick] (-0.65, 0)--(0,0);
        \draw[thick] (-0.47, -0.47)--(0,0);
        \draw[thick] (-0.47, 0.47)--(0,0);

        \draw[fill=white, thick] (0,0) circle [radius=10pt];

        \node at (0,0) {\scalebox{1}{$\mu$}};
        \end{tikzpicture}}}
        \quad
\longmapsto
\quad
    \vcenter{\hbox{\begin{tikzpicture}[scale=1]
        \draw[thick] (0,0)--(0.8, 0);
        \draw[thick] (-0.8, 0)--(0,0);
        \draw[thick] (-0.58, -0.58)--(0,0);
        \draw[thick] (-0.58, 0.58)--(0,0);

        \draw[fill=white, thick] (0,0) circle [radius=15pt];
        \node at (0,0) {\scalebox{1}{$\mu\, \ell_{g,n}$}};
        \end{tikzpicture}}}
\ \ .
\]

\medskip

In order to show that these assignments define a morphism $\mathrm{U}(\mathfrak{L}) \to \End_\g$ of operads, we need to check that the commutator of the actions of any pair of elements of $\mathfrak{L}$ is equal to the action of their Lie bracket. There are 6 cases depending on which type of elements we consider:
first type of unstable operation $R$, second type of unstable operation $T$, or Manin--Zograf element $\ell$.
Let us denote by $\Psi_R$, $\Psi_T$, and $\Psi_\ell$ the respective images in $\End_\g$.
\begin{description}
\item[$\lbrack\Psi_{R_1}, \Psi_{R_2}\rbrack=\Psi_{\lbrack R_1, R_2\rbrack}$] A straightforward computation shows that the commutator of the actions of $R_1$ and $R_2$ is equal to the action of $R_1\circ_1^2 R_2-(-1)^{|R_1||R_2|}R_2\circ_1^2 R_1=[R_1, R_2]$~.

\item[$\lbrack\Psi_{T_1}, \Psi_{T_2}\rbrack=0$] The actions of $T_1$ and $T_2$ commute by
\[
\left(\pi_i\right)_*\left(
\left(\pi_j\right)_* \left(\mu\, \psi_j^k\right)\psi_i^l\right)=
\left(\pi_j\right)_*\left(
\left(\pi_i\right)_* \left(\mu\, \psi_i^l\right)\psi_j^k\right)\]
for $k,l\geqslant 2$ (in fact, even $k,l\geqslant 1$ would be sufficient), which follows from the pull-back formula for the $\psi$-classes $\pi_j^*\psi_i^l = \psi_l - D\cdot \pi^*_j \psi_i^{l-1}$, where $D$ is the divisor represented by two-component curves, with one component of genus $0$ with the points marked by $i$ and $j$ on it, and the other component is of complementary genus with all other marked points on it. We have $\psi_j D = 0$ and thus we obtain the following expression
\[
\left(\pi_i\right)_*\left(
\left(\pi_j\right)_* \left(\mu\, \psi_j^k\right)\psi_i^l\right)=
\left(\pi_i\right)_*\left(
\left(\pi_j\right)_* \left(\mu\, \psi_j^k\pi_j^* \psi_i^l\right)\right) = \left(\pi_{ij}\right)_*\left(
\mu\, \psi_j^k \psi_i^l\right)~,
\]
where $\pi_{ij}$ forgets the $i$th and $j$th marked points, which
is symmetric under the interchange $(i,k)\leftrightarrow (j,l)$~.

\item[$\lbrack\Psi_{\ell_1}, \Psi_{\ell_2}\rbrack=0$] The actions of the Manin--Zograf elements commute by the commutativity of the cohomology algebra.

\item[$\lbrack\Psi_{R}, \Psi_{T}\rbrack=\Psi_{\lbrack R, T\rbrack}$] The commutator of the action of an operation $R$ with the action of an operation $T$ produces vertices labeled by cohomology classes of the form
$-\left(\pi_i\right)_*\left(\mu\, \psi_i^{k+l}\right)$~, which amounts precisely to the action of the bracket $[R, T]=-R\circ_2 ^1 T$~.

\item[$\lbrack\Psi_{R}, \Psi_{\ell}\rbrack=0$] The action of an unstable operation of first type commutes with the action of a Manin--Zograf element by the commutativity of the cohomology algebra.

\item[$\lbrack\Psi_{T}, \Psi_{\ell}\rbrack=\Psi_{\lbrack T, \ell\rbrack}$] The commutator of the actions of $T$ and $\ell=E+F$ is equal to the action of $E(\psi)T$. To see this, recall first that
$\kappa_a-\pi_i^*\kappa_a =  \psi_{i}^a$ and $\mathrm{ch}_{2a-1}-\pi_i^*\mathrm{ch}_{2a-11} = 0$, for $a\geqslant 1$. Then use the projection formula to compute
\[
\left(\pi_i\right)_*\left(\mu\, \ell_{g,n}\, \psi_i^k\right)-
\left(\pi_i\right)_*\left(\mu\, \psi_i^k\right)\ell_{g,n}=
\left(\pi_i\right)_*\left(\mu\, \left(\ell_{g,n}- \pi_i^*\ell_{g,n}\right)\, \psi_i^k\right)=
\left(\pi_i\right)_*\left(\mu\, E(\psi_i)\, \psi_i^k\right)~.
\]
\end{description}

\medskip

We claim that the action of any unary operation from $\mathfrak{L}$ commutes with the differential $\d=\d_A-\d_1-\d_2$ of $\g$.
There are 3 cases depending on which type of unary operation we consider.
\begin{description}
\item[$\lbrack\Psi_{R}, \d\rbrack=0$] It comes from the commutativity of the partial decomposition maps $\delta_i^j$ and $\chi_{ij}$ of the modular cooperad $\Hbul\big(\oM\big)$ with  the multiplication by $\psi_k$ with $k$ different from $i,j$ (recall that $\sigma^*\psi_k = \psi_k$ and $\rho^*\psi_k = \psi_k$, $k=1,\dots,n$).

\item[$\lbrack\Psi_{T}, \d\rbrack=0$]
It comes from the commutativity of the partial decomposition maps $\delta_i^j$ and $\chi_{ij}$ of the modular cooperad $\Hbul\big(\oM\big)$  with the push-forward $\left(\pi_k\right)_*$ of the forgetful map of a point $k$ different from $i,j$, see e.g. the diagram in~\cite[Proof of Equation~(1.8)]{ArbarelloCornalba}.

\item[$\lbrack\Psi_{\ell}, \d\rbrack=0$]
It comes from the commutativity of the partial decomposition maps $\delta_i^j$ and $\chi_{ij}$ of the modular cooperad $\Hbul\big(\oM\big)$  with the multiplication by Manin--Zograf classes $\ell_{g,n}$
(recall that $\sigma^*\ell_{g,n} = \ell_{g-1,n+2}$ and $
\rho^*\ell_{g,n} = \ell_{g',n'+1}\otimes 1 + 1 \otimes \ell_{g-g',n-n'+1}$).
\end{description}
This shows that  $\mathrm{U}(\mathfrak{L}) \to \End_{\g_A}$ is a morphism of dg operads.

\medskip

At this point, we have established a morphism of dg operads $\EGra_A\vee \mathrm{U}(\mathfrak{L}) \to \End_{\g_A}$~. It remains to show that the images of relations~\eqref{eq:DistLawTautGra} vanish. Since the actions of usual and dashed edges commute, it is enough to compute the commutators of the actions of
the three types of new unary operations
with the four operations induces by usual and dashes edges and loops. We use the notation
$R\cdot M \coloneq -R(\psi_1)\circ_2^1 M - (-1)^{|M||R|} M\circ_2^1 R^{(12)}(\psi_2)$~.

\begin{description}
\item[\rm $\lbrack\Psi_{R},  \{ \, , \}_M\rbrack= \{ \, , \}_{R\cdot M}$ \textrm{and} $\lbrack\Psi_{R},  \Delta_M\rbrack= \Delta_{R\cdot M}$] It is a straightforward computation.

\item[\rm $\lbrack\Psi_{R},  ^h\!\{ \, , \}_M\rbrack= ^h\!\{ \, , \}_{R\cdot M}$ \textrm{and} $\lbrack\Psi_{R},  ^h\!\Delta_M\rbrack= ^h\!\Delta_{R\cdot M}$]
It comes from the commutativity between the push-forward of the sewing maps $\circ_i^j$ and $\xi_{ij}$ and the multiplication by $\psi_k$ for $k$ different from $i,j$ (as we explained above, the latter follows from the projection formula and the fact $\sigma^*\psi_k = \psi_k$ and $\rho^*\psi_k = \psi_k$, $k=1,\dots,n$).

\item[\rm $\lbrack\Psi_{T},  \{ \, , \}_M\rbrack=0$ \textrm{and} $\lbrack\Psi_{T},  \Delta_M\rbrack=0$]
It is a straightforward computation.

\item[\rm $\lbrack\Psi_{T},  ^h\!\{ \, , \}_M\rbrack=0$ \textrm{and} $\lbrack\Psi_{T},  ^h\!\Delta_M\rbrack=0$]
It comes from the following commutativity between the push-forward of the sewing maps $\circ_i^j$ and $\xi_{ij}$ and the push-forward of the forgetful map $\pi_p$ at a different point:
\begin{align*}
\left(\pi_p\right)_* \left(\left(\circ_i^j\right)_*\left( \mu\, \psi_i^k \otimes \mu'\, \psi_j^l\right) \psi_p^q\right)
& =
 \left(\circ_i^j\right)_*\left( \left(\pi_p\right)_*\left( \mu\,\psi_p^q\right) \psi_i^k \otimes \mu'\, \psi_j^l\right)~,
\\
\left(\pi_p\right)_* \left(\left(\xi_{ij}\right)_*\left( \mu\, \psi_i^k \, \psi_j^l\right) \psi_p^q\right)
& =
 \left(\xi_{ij}\right)_*\left( \left(\pi_p\right)_*\left( \mu\,\psi_p^q\right) \psi_i^k \, \psi_j^l\right)~
\end{align*}
(here in the first equation the point $p$ is assumed to be on the vertex labeled by $\mu$).

\item[\rm $\lbrack\Psi_{\ell},  \{ \, , \}_M\rbrack=0$ \textrm{and} $\lbrack\Psi_{\ell},  \Delta_M\rbrack=0$]
It is basically the definition of $\Psi_{\ell}$ combined with the following observation; multiplication of a vertex class by $\ell$ commutes with multiplication by $\psi$-classes on the edges by commutativity of the cohomology algebra.

\item[\rm $\lbrack\Psi_{\ell},  ^h\!\{ \, , \}_M\rbrack=0$ \textrm{and} $\lbrack\Psi_{\ell},  ^h\!\Delta_M\rbrack=0$]
It comes from the the projection formula and the basic property of the Manin--Zograf classes with respect to the sewing maps $\circ_i^j$ and $\xi_{ij}$. Namely, with $
(\circ_i^j)^*\ell_{g,n} = \ell_{g',n'+1}\otimes 1 + 1 \otimes \ell_{g-g',n-n'+1}$ and $\xi_{ij}^*\ell_{g,n} = \ell_{g-1,n+2}$ we have:
\begin{align*}
& \left(\circ_i^j\right)_*\left(\ell_{g',n'+1}\mu\otimes \mu' + \mu \otimes \mu'\ell_{g-g',n-n'+1}\right) = \left(\circ_i^j\right)_*\left( (\mu\otimes \mu') \left(\circ_i^j\right)^*\ell_{g,n} \right) = \left(\circ_i^j\right)_*\left( (\mu\otimes \mu') \right)\ell_{g,n}\,, \\
& \left(\xi_{ij}\right)_* \left(\mu\,\ell_{g,n} \right) = \left(\xi_{ij}\right)_* \left(\mu\,\left(\xi_{ij}\right)^*(\ell_{g,n}) \right) = \left(\xi_{ij}\right)_* \left(\mu \right)
\ell_{g,n}\,.
\end{align*}
\end{description}

This concludes the proof that the assignment $\GGra_A\to \End_{\g_A}$ defines a morphism of dg operads.

\medskip

To discuss the functoriality of the present construction, we denote by $\GGra_A$ the tautological graph operad associated to $A$. The assignment $A \mapsto \GGra_A$ defines a functor from the category of
dg symmetric vector spaces
 to the category of dg operads. Any chain map $f \colon A \to B$ which preserves the respective bilinear forms induces a morphism
$\GGra_f \colon \GGra_A \to \GGra_B$
of dg operads
and a chain map
$\g_f \colon \g_A \to \g_B$. The functoriality of the present $\GGra_A$-algebra structure is expressed by the following commutative diagram:
\[
\vcenter{\hbox{
\begin{tikzcd}
\GGra_A \arrow[rr] \arrow[d, "\GGra_f"']
& & \End_{\g_A} \arrow[d, "(\g_f)_*"']
 \\
\GGra_B \arrow[r] & \End_{\g_B} \arrow[r, "(\g_f)^*"'] &  \End^{\g_A}_{\g_B}~.
\end{tikzcd}
}}
\]

When $A$ is finite dimensional with a non-degenerate symmetric bilinear form, the graphs made up of usual edges only decorated by $\mathbb{I}\,\psi_1^0\psi_2^0$, where $\mathbb{I}\in A^{\otimes 2}$ is
the image of the identity map of $A$ under
the pairing $\langle\ ,\, \rangle$, act as in \cref{prop:GraAction} Point~(1). So the present action extends that of the operad $\Gra$:
\[
\vcenter{\hbox{
\begin{tikzcd}
\Gra
\arrow[rr, "\Psi"] \arrow[dr, hook]
& & \End_{\g_A}
 \\
&\GGra_A\ .  \arrow[ur]  &
\end{tikzcd}
}}
\]
\end{proof}

\subsection{The tautological graph complex and the Givental--Teleman theory} \label{sec:OldGivental}
In this section, we apply the methods introduced in \cref{sec:GTgroup}, and more precisely in \cref{subsec:hbar1}, to the operad $\GGra$ instead of $\Gra$ and to the $\GGra_A$-algebra $\g_A$ instead of any $\Gra$-algebra $\g$.

\begin{assumption}\label{ass:fdnd}
From now on and until the end of \cref{sec:GGT}, we will work with the $\mathbb{Z}/2\mathbb{Z}$ grading convention
 and we assume that the dg symmetric vector space $A$ is finite-dimensional and equipped with a non-degenerate pairing.
In this case, we denote by $\mathbb{I}\in A^{\otimes 2}$  the image of the identity map of $A$ under
the pairing. We note that the vector space $Z_\bullet\left(A^{\otimes 2}\right)$ splits as a direct sum of $\K \mathbb{I}$ and its orthogonal complement.
As a consequence, we obtain a splitting
\[\Mrk\coloneq Z_\bullet\left(A^{\otimes 2}\right)[[\psi_1,\psi_2]]^{\Sy_2}
\cong \K \mathbb{I} \oplus \overline{\Mrk}\]
of the space of symmetric elements which we use in edge labels. We shall refer to the usual edges labeled by $\mathbb{I}$ as \emph{standard edges}.
\end{assumption}

We still denote the unary tadpole of  $\Gra\subset\GGra_A$ by
\[\omega\coloneq \vcenter{\hbox{\begin{tikzpicture}[scale=0.7]
	\draw[thick]  (-0.4,2.7) arc [radius=0.4, start angle=180, end angle= 0];
	\draw[thick] (0, 2) to [out=135,in=270] (-0.4,2.7);
	\draw[thick] (0, 2) to [out=45,in=270] (0.4,2.7);
 	 \draw[fill=white, thick] (0,2) circle [radius=10pt];
 	\node at (0,2) {\scalebox{1}{$1$}};
\end{tikzpicture}}}~.
\]
By a slight abuse, we use the same notations
\[
\vcenter{\hbox{
\begin{tikzcd}
\S  \Lie
\arrow[rr, "\varPhi"] \arrow[dr, "\varTheta"']
& & \End_{\left(\g_A, \d+\Delta\right)}\ ,
 \\
&\GGra_A^\omega  \arrow[ur, "\varPsi"']  &
\end{tikzcd}
}}
\]
where the bottom  operad is twisted by the Maurer--Cartan element $\omega$~.
By construction, the  deformation complex $\b_{\GGra_A^\omega}^{\vartheta}$ of the morphism $\varTheta$ carries a complete dg Lie algebra structure.

\begin{definition}[Givental group/Lie algebra]
The induced (complete) Lie algebra
\[\giv\coloneq \left(H_\ev\Big(
\b_{\GGra_A^\omega}^{\vartheta}\Big), [\,,] \right)\]
is called the \emph{Givental Lie algebra}. The group
\[\GIV\coloneq \left( H_\ev\Big(
\b_{\GGra_A^\omega}^{\vartheta}\Big), \BCH,0\right)\]
integrating this Lie algebra is called the \emph{Givental group}.
\end{definition}

\begin{theorem}\label{th:ActOnCohFT}
The Givental group acts functorially
\[\GIV  \to
\Aut\left(
\calMC(\g_A)
\right)\]
on the moduli spaces of  homotopy CohFTs under the pre-Lie exponential formula
\[\lambda \cdot \alpha \coloneqq \sum_{n\geqslant 1}
{\textstyle \frac{1}{n!}}
\left(e^{\varPsi(\lambda)}\right)_n(\alpha, \ldots, \alpha)\ .\]
\end{theorem}

\begin{proof}
The same methods as in \cref{sec:GTgroup}, especially \cref{prop:ActionBIS}, apply \emph{mutatis mutandis}.
Both assignments $A \mapsto \GIV$ and $A\mapsto \Aut\left(
\calMC(\g_A)\right)$ define functors from the category of
dg symmetric vector spaces.
 to the category of groups and the assignment $\GIV  \to
\Aut\left(
\calMC(\g_A)
\right)$ is a natural transformation between them.
\end{proof}

The underlying graded vector space of $\b^\vartheta_{\GGra_A^\omega}$ is equal to
$\prod_{n\geqslant 1} \GGra_A(n)^{\mathbb{S}_n}$ with the differential given by
\begin{equation}\label{eq:diffGraphCx}
\Gamma \mapsto
\DE(\Gamma) + \left[
\vcenter{\hbox{\begin{tikzpicture}[scale=0.5]
    \draw[thick] (0,0)--(1.5, 0);
    \draw[fill=white, thick] (0,0) circle [radius=10pt];
    \draw[fill=white, thick] (1.5,0) circle [radius=10pt];
    \end{tikzpicture}}}
+
\vcenter{\hbox{\begin{tikzpicture}[scale=0.5]
    \draw[thick]  (-0.4,2.7) arc [radius=0.4, start angle=180, end angle= 0];
    \draw[thick] (0, 2) to [out=135,in=270] (-0.4,2.7);
    \draw[thick] (0, 2) to [out=45,in=270] (0.4,2.7);
    \draw[fill=white, thick] (0,2) circle [radius=10pt];
    \node (n) at (0,1.5) {};
    \end{tikzpicture}}}\,,\, \Gamma\right]\ .
\end{equation}
Like in \cref{prop:h=1Triv}, we are able to fully compute the associated homology groups but they turn out to be non-trivial in this case.

\begin{theorem}\label{cor:GiventalGroup}
The homology groups of $\b^\vartheta_{\GGra_A^\omega}$
are spanned by symmetric tensors of the following graphs
  \begin{align} \label{eq:cohom-GiventFull}
&  \sum_{k\geqslant 0}\,\left(
\vcenter{\hbox{\begin{tikzpicture}[scale=0.6]
        \draw[fill=white, thick] (0,0.3) circle [radius=30pt];
        \draw[thick] (0,28pt)--(0,48pt);
        \node at (0,10pt) {\scalebox{1}{$a_1^k\psi^{k+1}$}};
        \node at (0,62pt) {\scalebox{1}{$a_2^k$}};
        \end{tikzpicture}}}
        - \sum_{i+j = k} (-1)^j \,\left(
\vcenter{\hbox{\begin{tikzpicture}[scale=0.9]
  \draw[dashed,thick]  (-0.4,2.7) arc [radius=0.4, start angle=180, end angle= 0];
  \draw[dashed,thick] (0, 2) to [out=135,in=270] (-0.4,2.7);
  \draw[dashed,thick] (0, 2) to [out=45,in=270] (0.4,2.7);
   \draw[fill=white, thick] (0,2) circle [radius=8pt];
  \node at (-0.7,2.4) {\scalebox{1}{$\psi_1^i$}};
  \node at (0,3.4) {\scalebox{1}{$a_1^{k}\otimes a_2^{k}$}};
  \node at (0.7,2.4) {\scalebox{1}{$\psi_2^j$}};
\end{tikzpicture}}}
 +
    \vcenter{\hbox{\begin{tikzpicture}[scale=0.75]
        \draw[dashed,thick] (0,0)--(4, 0);
        \draw[fill=white, thick] (0,0) circle [radius=10pt];
        \draw[fill=white, thick] (4,0) circle [radius=10pt];
        \node at (0.7,0.4) {\scalebox{1}{$\psi_1^i$}};
        \node at (2,0.4) {\scalebox{1}{$a_1^{k}\otimes a_2^{k}$}};
        \node at (3.3,0.4) {\scalebox{1}{$\psi_2^j$}};
        \node at (1, -0.8) { };
        \end{tikzpicture}}}
\right)\right)
\ ,
\\ & \notag
\vcenter{\hbox{\begin{tikzpicture}[scale=0.9]
  \draw[thick]  (-0.4,2.7) arc [radius=0.4, start angle=180, end angle= 0];
  \draw[thick] (0, 2) to [out=135,in=270] (-0.4,2.7);
  \draw[thick] (0, 2) to [out=45,in=270] (0.4,2.7);
   \draw[fill=white, thick] (0,2) circle [radius=8pt];
  \node at (0,3.4) {\scalebox{1}{$a_1\otimes a_2$}};
\end{tikzpicture}}}\ ,
\ \
\sum_{l\geqslant 2}\,
\vcenter{\hbox{\begin{tikzpicture}[scale=0.6]
        \draw[fill=white, thick] (0,0) circle [radius=22pt];
        \draw[thick] (0,22pt)--(0,12pt);
        \node at (0,2pt) {\scalebox{1}{$a^l\psi^l$}};
        \end{tikzpicture}}} \ ,
\ \
\text{and}
\ \
\vcenter{\hbox{\begin{tikzpicture}[scale=0.75]
        \draw[fill=white, thick] (0,0) circle [radius=10pt];
        \node at (0,0) {\scalebox{0.8}{$\ell$}};
        \end{tikzpicture}}}\ ,
\end{align}
for all elements $a_1^{k}\otimes a_2^{k} \in Z_\bullet\left(A^{\otimes 2}\right)$ such that
$\left(a_1^{k}\otimes a_2^{k}\right)^{(12)}=(-1)^{k} a_1^{k}\otimes a_2^{k}$, for $k\geqslant 0$,
all elements $a_1\otimes a_2 \in Z_\bullet\left(A^{\otimes 2}\right)^{\Sy_2}$,
all elements $a^l\in Z_\bullet(A)$, for $l\geqslant 0$, and all element $\ell\in L_{MZ}$~.
The underlying vector space of the Givental group/Lie algebra is spanned by symmetric tensors of these graphs of even degree.
\end{theorem}

\begin{proof}
Notice that
the differential \eqref{eq:diffGraphCx} is equal to the sum of the following four terms:
\begin{description}
\item[$\d_1$] coming for $\DE$, i.e. changing dashed edges into usual edges and multiplying the label by $\psi_1+\psi_2$,

\item[$\d_2$] coming from the commutator with $
\vcenter{\hbox{\begin{tikzpicture}[scale=0.5]
    \draw[thick] (0,0)--(1.5, 0);
    \draw[fill=white, thick] (0,0) circle [radius=10pt];
    \draw[fill=white, thick] (1.5,0) circle [radius=10pt];
    \end{tikzpicture}}}$, i.e. splitting vertices into two, linking them by a standard edge, and distributing the attached edges and labelings over them; it might also make some elements $R$ labeling the vertex act on the standard edge and thus making it into a usual one,

\item[$\d_3$] coming from the first term of the commutator with $\vcenter{\hbox{\begin{tikzpicture}[scale=0.5]
    \draw[thick]  (-0.4,2.7) arc [radius=0.4, start angle=180, end angle= 0];
    \draw[thick] (0, 2) to [out=135,in=270] (-0.4,2.7);
    \draw[thick] (0, 2) to [out=45,in=270] (0.4,2.7);
    \draw[fill=white, thick] (0,2) circle [radius=10pt];
    \node (n) at (0,1.5) {};
    \end{tikzpicture}}}$~, i.e. creating a new standard edge between two different vertices,

\item[$\d_4$] coming from the other terms of the commutator with $\vcenter{\hbox{\begin{tikzpicture}[scale=0.5]
    \draw[thick]  (-0.4,2.7) arc [radius=0.4, start angle=180, end angle= 0];
    \draw[thick] (0, 2) to [out=135,in=270] (-0.4,2.7);
    \draw[thick] (0, 2) to [out=45,in=270] (0.4,2.7);
    \draw[fill=white, thick] (0,2) circle [radius=10pt];
    \node (n) at (0,1.5) {};
    \end{tikzpicture}}}$~, i.e. creating a new usual tadpole with label coming the iterated actions of some elements $R_1,\ldots, R_k$ labeling the vertex.
\end{description}

For the sake of the forthcoming arguments, we will use the homological $\ZZ$-grading provided by the opposite of the number of usual edges. Let us consider the filtration
\[\mathrm{F}_p\coloneq \prod_{n\geqslant -p} \GGra_A(n)^{\mathbb{S}_n}\]
whose $p$th term consists of series of graphs with at least $-p$ vertices.
The first page $E^0_{-p,q}$ of the associated spectral sequence is isomorphic to
$\GGra_A(p)^{\mathbb{S}_{p}}_{-p+q}$, for $p\geqslant 1$, and $0$ otherwise. Its differential is equal to
$d^0=\d_1+\d_3+\d_4$.

\medskip

We shall first demonstrate that $E^1_{p, q}=0$, for $p\neq-1$ and any $q\in \NN$, or, in other words, that the chain complexes $\GGra_A(n)^{\mathbb{S}_n}$ equipped with the differential $\d_1+\d_3+\d_4$ are acyclic for $n\geqslant 2$~.
For that, we remark that considering the homology commutes with considering the invariants with respect to the action of a finite group. Thus, it is sufficient to that the chain complex $\left(\GGra_A(n), \d_1+\d_3+\d_4\right)$, where we slightly abuse the notation for the differential, is acyclic for $n\geqslant 2$. To this end, we consider another filtration whose $p$-th term consists of series of graphs with the set of edges and tadpoles decorated by at least $-p$ elements from $\overline{\Mrk}$~. The first page of the associated spectral sequence is a chain complex made up of linear combinations of decorated graphs with $n$ vertices with only $\d_3$ for differential. It is straightforward to see that the following degree $1$ map $h$ is a contracting homotopy for that complex: the image of a decorated graph $\Gamma$ under $h$ is equal to $0$ if there is no standard edge between the first and the second vertices, otherwise the map $h$ removes such an edge, which is anyway unique. This second filtration is obviously exhaustive and bounded below, so its spectral sequence converges by the Classical Convergence Theorem~\cite[Theorem~5.5.1]{WeibelBook}.

\medskip

We have established that the spectral sequence associated to the first filtration vanishes at page two, where only $E^1_{-1,q}$~, for $q\leqslant 1$, is non-trivial. So the first filtration is regular and it is also obviously complete, exhaustive, and bounded above:  its spectral sequence converges by the Complete Convergence Theorem~\cite[Theorem~5.5.10]{WeibelBook}. It remains to determine the spaces $E^1_{-1,q}$~, that is compute the homology of the complex of one-vertex graphs. Such graphs may have several attached tadpoles, usual or dashed, labeled by elements of $\Mrk$, and the vertex carries a label from $\mathrm{U}(\mathfrak{L})$~. The dg associative algebra whose underlying space is $\GGra_A(1)$ and whose differential is
\begin{equation}\label{eq:OneVertexDiff}
 \DE + \left[
\vcenter{\hbox{\begin{tikzpicture}[scale=0.5]
    \draw[thick]  (-0.4,2.7) arc [radius=0.4, start angle=180, end angle= 0];
    \draw[thick] (0, 2) to [out=135,in=270] (-0.4,2.7);
    \draw[thick] (0, 2) to [out=45,in=270] (0.4,2.7);
    \draw[fill=white, thick] (0,2) circle [radius=10pt];
    \node (n) at (0,1.5) {};
    \end{tikzpicture}}}\,,\, -\right]=\d_1+\d_3+\d_4\
\end{equation}
is isomorphic to the enveloping algebra of the following  dg Lie algebra  $\mathfrak{L}^{\mathrm{ext}}$.
The underlying vector space of $\mathfrak{L}^{\mathrm{ext}}$ is equal to $\mathfrak{L}\oplus \Mrk\oplus s^{-1}\Mrk$, where  $s^{-1}\Mrk$ corresponds to the usual tadpoles and $\Mrk$ corresponds to the dashed ones. The nontrivial brackets are the Lie brackets of $\mathfrak{L}$ and the ones coming from the following actions
\begin{align*}
&R\cdot M \coloneq -R(\psi_1)\circ_2^1 M - (-1)^{|M||R|} M\circ_2^1 R^{(12)}(\psi_2)~, \\
&R\cdot s^{-1}M \coloneq -R(\psi_1)\circ_2^1 s^{-1}M - (-1)^{(|M|-1)|R|} s^{-1}M\circ_2^1 R^{(12)}(\psi_2)~,
\end{align*}
for $M\in \Mrk$. Its differential
 $\d_{\mathfrak{L}^{\mathrm{ext}}}$ is defined by $\mathfrak{L} \oplus \Mrk \oplus s^{-1}\Mrk\to s^{-1}\Mrk$~, for which
\begin{align*}
&\d_{\mathfrak{L}^{\mathrm{ext}}} (M) \coloneq s^{-1}(\psi_1+\psi_2)  M~, \\
&\d_{\mathfrak{L}^{\mathrm{ext}}}  (R) \coloneq s^{-1}\big(R(\psi_1) + R^{(12)}(\psi_2)\big)~, \\
& \d_{\mathfrak{L}^{\mathrm{ext}}} (T) = \d_{\mathfrak{L}^\mathrm{ext}} (\ell_{g,n}) = \d_{\mathfrak{L}^{\mathrm{ext}}} \big(s^{-1}M\big)\coloneq  0~.
\end{align*}
The classical theorem \cite[Proposition~B.2.1]{QuillenRHT} of Quillen asserts that the homology of the universal enveloping algebra of $\mathfrak{L}^{\mathrm{ext}}$ is isomorphic to the universal enveloping algebra of the
homology of $(\mathfrak{L}^{\mathrm{ext}}, \d_{\mathfrak{L}^{\mathrm{ext}}})$. Therefore, the homology  of $\GGra_A(1)$ with respect to the differential $\d_1+\d_3+\d_4$ is isomorphic to the universal enveloping algebra
of the Lie subalgebra of~$\mathfrak{L}^{\mathrm{ext}}$ spanned by
  \begin{equation}\label{eq:cohom-GiventPartial}
  \sum_{k\geqslant 0}\,\left(
\vcenter{\hbox{\begin{tikzpicture}[scale=0.6]
        \draw[fill=white, thick] (0,0.3) circle [radius=30pt];
        \draw[thick] (0,28pt)--(0,48pt);
        \node at (0,10pt) {\scalebox{1}{$a_1^k\psi^{k+1}$}};
        \node at (0,62pt) {\scalebox{1}{$a_2^k$}};
        \end{tikzpicture}}}
        - \sum_{i+j = k} (-1)^j \,
    \vcenter{\hbox{\begin{tikzpicture}[scale=0.9]
  \draw[dashed,thick]  (-0.4,2.7) arc [radius=0.4, start angle=180, end angle= 0];
  \draw[dashed,thick] (0, 2) to [out=135,in=270] (-0.4,2.7);
  \draw[dashed,thick] (0, 2) to [out=45,in=270] (0.4,2.7);
   \draw[fill=white, thick] (0,2) circle [radius=8pt];
  \node at (-0.7,2.4) {\scalebox{1}{$\psi_1^i$}};
  \node at (0,3.4) {\scalebox{1}{$a_1^{k}\otimes a_2^{k}$}};
  \node at (0.7,2.4) {\scalebox{1}{$\psi_2^j$}};
\end{tikzpicture}}}
\right)\ ,
\ \
            \vcenter{\hbox{\begin{tikzpicture}[scale=0.9]
  \draw[thick]  (-0.4,2.7) arc [radius=0.4, start angle=180, end angle= 0];
  \draw[thick] (0, 2) to [out=135,in=270] (-0.4,2.7);
  \draw[thick] (0, 2) to [out=45,in=270] (0.4,2.7);
   \draw[fill=white, thick] (0,2) circle [radius=8pt];
  \node at (0,3.4) {\scalebox{1}{$a_1\otimes a_2$}};
\end{tikzpicture}}}
\ , \ \
\sum_{l\geqslant 2}\,
\vcenter{\hbox{\begin{tikzpicture}[scale=0.6]
        \draw[fill=white, thick] (0,0) circle [radius=22pt];
        \draw[thick] (0,22pt)--(0,12pt);
        \node at (0,2pt) {\scalebox{1}{$a^l\psi^l$}};
        \end{tikzpicture}}} \ ,
\ \
\text{and}
\ \
\vcenter{\hbox{\begin{tikzpicture}[scale=0.75]
        \draw[fill=white, thick] (0,0) circle [radius=10pt];
        \node at (0,0) {\scalebox{0.8}{$\ell$}};

        \end{tikzpicture}}}\ ,
        \end{equation}
for all elements $a_1^{k}\otimes a_2^{k} \in Z_\bullet\left(A^{\otimes 2}\right)$ such that
$\left(a_1^{k}\otimes a_2^{k}\right)^{(12)}=(-1)^{k} a_1^{k}\otimes a_2^{k}$, for $k\geqslant 0$,
all elements $a_1\otimes a_2 \in Z_\bullet\left(A^{\otimes 2}\right)^{\Sy_2}$,
all elements $a^l\in Z_\bullet(A)$, for $l\geqslant 0$, and all elements $\ell\in L_{MZ}$~.

\medskip

This proves that the homology of the chain complex $\b^\vartheta_{\GGra_A^\omega}$ is isomorphic to the underlying space of the universal enveloping algebra
of the  Lie subalgebra of~$\mathfrak{L}^{\mathrm{ext}}$ spanned by the elements depicted in \eqref{eq:cohom-GiventPartial}. As the final step, we recall that we used a spectral sequence that allowed us to temporarily ignore the part $\d_2$ of the differential, and the first set of elements in~\eqref{eq:cohom-GiventPartial} does not commute with $\vcenter{\hbox{\begin{tikzpicture}[scale=0.5]
    \draw[thick] (0,0)--(1.5, 0);
    \draw[fill=white, thick] (0,0) circle [radius=10pt];
    \draw[fill=white, thick] (1.5,0) circle [radius=10pt];
    \end{tikzpicture}}}$~.
 By a direct computation, one obtains homology representatives in the original complex given by~\eqref{eq:cohom-GiventFull}. It remains to conclude with the Poincar\'e--Birkhoff--Witt theorem.
\end{proof}

\begin{remark}\label{rem:ConstUnstable}
In the light of~\cref{rem:relaxRT}, we note that both the statement and the proof of~\cref{cor:GiventalGroup} change once we generalize the definitions of $R$ and $T$ elements suggested there. For instance, the homology of the differential~\eqref{eq:OneVertexDiff} acting on one-vertex graphs contains no tadpole term with a usual edge.
\end{remark}

Let us notice the following special feature of the Givental group action on homotopy CohFTs.

\begin{lemma}
The Givental group sends CohFTs without unit to CohFTs without unit.
\end{lemma}

\begin{proof}
Recall from \cref{cor:CohFTTFTtree} that CohFTs without unit coincide with
 Maurer--Cartan elements $\alpha\in\g_A$ supported on the one-vertex graphs without tadpoles.
The pre-Lie exponential formula given in \ref{th:ActOnCohFT} shows that the action of
any element of \eqref{eq:cohom-GiventFull} of
 the Givental group preserves the space of $\g_A$ of  one-vertex graphs without tadpoles, that is sends
 CohFTs without unit to CohFTs without unit.
\end{proof}

This property allows us to compare the present fundamental theory to the already existing group actions on CohFTs:  in the present case, we actually recover the Givental--Teleman theory. Our main reference for the classical Givental group action is~\cite{Teleman2012}, further surveys are available in~\cite{ShadrinBCOV,PandPixZvo2015}.
Since in most cases the Givental--Teleman theory is applied to CohFTs with a unit, we repeatedly stress that a (possibly existing) unit is not a part of the structure.

\begin{proposition}\label{prop:RecoverGivTel}  \leavevmode
Let $(A, \langle\ ,\, \rangle)$  be a graded symmetric vector space, i.e. with
 $\d_A=0$, satisfying \cref{ass:fdnd}, and let $\alpha \in \MC(\g_A)$ be a CohFT without unit.
The action
\begin{equation}r \cdot \alpha = \sum_{n\geqslant 1}
{\textstyle \frac{1}{n!}}
\left(e^{\varPsi(r)}\right)_n(\alpha, \ldots, \alpha)
\end{equation}
of the Givental group element
\[
r\coloneq\sum_{k\geqslant 0}\,\left(
\vcenter{\hbox{\begin{tikzpicture}[scale=0.6]
        \draw[fill=white, thick] (0,0.3) circle [radius=30pt];
        \draw[thick] (0,28pt)--(0,48pt);
        \node at (0,10pt) {\scalebox{1}{$a_1^k\psi^{k+1}$}};
        \node at (0,62pt) {\scalebox{1}{$a_2^k$}};
        \end{tikzpicture}}}
        - \sum_{i+j = k} (-1)^j \,\left(
\vcenter{\hbox{\begin{tikzpicture}[scale=0.9]
  \draw[dashed,thick]  (-0.4,2.7) arc [radius=0.4, start angle=180, end angle= 0];
  \draw[dashed,thick] (0, 2) to [out=135,in=270] (-0.4,2.7);
  \draw[dashed,thick] (0, 2) to [out=45,in=270] (0.4,2.7);
   \draw[fill=white, thick] (0,2) circle [radius=8pt];
  \node at (-0.7,2.4) {\scalebox{1}{$\psi_1^i$}};
  \node at (0,3.4) {\scalebox{1}{$a_1^{k}\otimes a_2^{k}$}};
  \node at (0.7,2.4) {\scalebox{1}{$\psi_2^j$}};
\end{tikzpicture}}}
 +
    \vcenter{\hbox{\begin{tikzpicture}[scale=0.75]
        \draw[dashed,thick] (0,0)--(4, 0);
        \draw[fill=white, thick] (0,0) circle [radius=10pt];
        \draw[fill=white, thick] (4,0) circle [radius=10pt];
        \node at (0.7,0.4) {\scalebox{1}{$\psi_1^i$}};
        \node at (2,0.4) {\scalebox{1}{$a_1^{k}\otimes a_2^{k}$}};
        \node at (3.3,0.4) {\scalebox{1}{$\psi_2^j$}};
        \node at (1, -0.8) { };
        \end{tikzpicture}}}
\right)\right)~,
\]
with
$a_1^{k}\otimes a_2^{k} \in A^{\otimes 2}$ such that
$\left(a_1^{k}\otimes a_2^{k}\right)^{(12)}=(-1)^k a_1^{k}\otimes a_2^{k}$, for $k\geqslant 0$,
coincides with the Givental--Teleman formula
\[\exp\left({\widehat{r}(z)}\right).\{\alpha\}\ , \quad \text{with} \quad \widehat{r}(z)=\sum_{k\geqslant 1}^\infty r_k z^k
\ , \quad \text{where} \quad r_k \colon a \mapsto \left\langle a, a_1^{k-1}\right\rangle a_2^{k-1}~, \]
for the $R$-group action on CohFTs \emph{without unit}.
\end{proposition}

\begin{proof}
We will prove that both infinitesimal actions coincide. We start by noting that, under the identification of $A^{\otimes 2}$ with $\Hom(A,\allowbreak A)\cong A^*\otimes A$~, the condition
$\left(a_1^{k}\otimes a_2^{k}\right)^{(12)}=(-1)^k a_1^{k}\otimes a_2^{k}$ naturally appearing from studying the symmetric elements in our formalism immediately becomes the defining relation $\widehat{r}(z)+{\widehat{r}}^*(-z)=0$ for the Lie algebra of the Givental--Teleman $R$-group.

\medskip

In the case of CohFTs, we have, in the formalism of \cref{subsec:DefCohFT}, just a system of labels
$\alpha=\{\alpha_{g,n}\}_{(g,n)\in \NN^2}$, for the one-vertex graphs without tadpoles
$\alpha_{g,n}\in \big(\Hbul\big(\oM_{g,n}\big)\big)(A)$ that satisfy
the factorization properties of CohFTs. More precisely, the vertex of a one vertex graph in genus
$g$ with the leaves labeled by  $a_1,\dots,a_n$ is labeled by $\frac 1{n!}\sum_{\sigma\in S_n} \alpha_{g,n}(a_{\sigma(1)},\dots,a_{\sigma(n)})$, cf. Examples in~\cref{subsec:PZ}).
The pre-Lie exponential action of \cref{th:ActOnCohFT} is the exponent of the Givental Lie algebra
action on $\alpha$ given by $\sum_{n\geqslant 1}
{\textstyle \frac{1}{n!}}
\left(\varPsi(r)\right)_n(\alpha, \ldots, \alpha)$, which, in arity $n$ and genus $g$, is equal to:
\begin{multline*}
\sum_{k\geqslant 0}\sum_{m=1}^n r_{k+1}^{(m)} \alpha_{g,n}\psi_m^{k+1}
-
\frac 12 \sum_{k\geqslant 0}\sum_{i+j=k} (-1)^j
  \left(\xi_{n+1,n+2}\right)_* \langle -, r_{k+1}- \rangle^{(n+1,n+2)} \alpha_{g-1,n+2}\psi_{n+1}^i \psi_{n+2}^j
\\
-
\frac 12 \sum_{k\geqslant 0}\sum_{i+j=k} (-1)^j \sum \left(\circ_{n'+1}^{n-n'+1}\right)_* \langle -, r_{k+1}- \rangle^{[n'+1,n-n'+1]} \alpha_{g',n'+1}\psi_{n'+1}^i \otimes\alpha_{g-g',n-n'+1} \psi_{n-n'+1}^j~,
\end{multline*}
where $r_{k+1}^{(m)}$ means that $r_{k+1}$ acts on the $m$th input of $\alpha_{g,n}$,
the expression $\langle -, r_{k+1}- \rangle^{(n+1,n+2)}$ indicates that we substitute the bivector dual to $\langle -, r_{k+1}- \rangle$ to the $(n+1)$th and $(n+2)$th inputs respectively of $\alpha_{g-1,n+2}$,
and the expression $\langle -, r_{k+1}- \rangle^{[n'+1,n-n'+1]}$ indicates that we substitute the bivector dual to  $\langle -, r_{k+1}- \rangle$ to the $(n'+1)$th input of $\alpha_{g',n'+1}$ and $(n-n'+1)$th input of $\alpha_{g-g',n-n'+1}$ respectively. The last sum in the third term runs over all stable ways to simultaneously split $g=g'+(g-g')$ and the set of inputs labeled by $\{1,\dots,n\}$ into two complementary subsets of cardinality $n'$ and $n-n'$. The coefficient $1/2$ in the second summand comes from the order of the automorphism group of the second summand in the formula for $r$,
\[
-\sum_{i+j = k} (-1)^j \
\vcenter{\hbox{\begin{tikzpicture}[scale=0.9]
			\draw[dashed,thick]  (-0.4,2.7) arc [radius=0.4, start angle=180, end angle= 0];
			\draw[dashed,thick] (0, 2) to [out=135,in=270] (-0.4,2.7);
			\draw[dashed,thick] (0, 2) to [out=45,in=270] (0.4,2.7);
			\draw[fill=white, thick] (0,2) circle [radius=8pt];
			\node at (-0.7,2.4) {\scalebox{1}{$\psi_1^i$}};
			\node at (0,3.4) {\scalebox{1}{$a_1^{k}\otimes a_2^{k}$}};
			\node at (0.7,2.4) {\scalebox{1}{$\psi_2^j$}};
\end{tikzpicture}}}\,,
\]
and the coefficient $1/2$ in the third summand comes from the pre-Lie exponent.

\medskip

This formula is standard in Givental--Teleman theory, see e.g.~\cite{ShadrinBCOV,DotsenkoShadrinVallette11,PandPixZvo2015};  modulo the appropriate adjustment of the notation it coincides with~\cite[Definition 6.1]{Teleman2012}.

\medskip

An alternative way to prove this proposition would be to match the terms of the full group action $r\cdot \alpha$  with~\cite[Definition 2.2]{PandPixZvo2015}: the pre-Lie exponential $e^r$ produces all the stable graphs $\Gamma$ (with dashed edges) with coefficient $\frac{1}{|\mathrm{Aut}(\Gamma)|}$ and their action under $\varPsi$  on $\alpha$ coincides with the one given in~\emph{op.~cit.}.
\end{proof}

\begin{remark}
An interesting open question is to compare the present construction with the previously known construction of the Givental group action on homotopy hypercommutative algebras (also known as homotopy flat F-manifolds without unit) coming from change of trivialization of the trivial circle action, see \cite{DSV-Trivialisation}. That latter construction is \emph{a priori} not compatible with the cyclic group action. However, one can expect that it should have a cyclic formulation that might be shown to be equivalent to the genus $0$ part of the abovementioned action of the Givental group element $r$.
\end{remark}

One might now wonder how to interpret the actions of the other generators of the Givental group on CohFTs.
We claim that the action of the second type of generators of \eqref{eq:cohom-GiventFull}, that is the ones corresponding to the second type of unary unstable operations, recovers the translation action given in~\cite[Section 6.3]{Teleman2012}, see also~\cite[Section 2.2]{PandPixZvo2015}.

\begin{proposition}\label{prop:translation}
Let $(A, \langle\ ,\, \rangle)$  be a graded symmetric vector space, i.e. with
 $\d_A=0$, satisfying \cref{ass:fdnd}, and let $\alpha \in \MC(\g_A)$ be a CohFT without unit.
The action
\begin{equation}\label{eq:TranslationGivental}
T \cdot \alpha = \sum_{n\geqslant 1}
{\textstyle \frac{1}{n!}}
\left(e^{\varPsi(T)}\right)_n(\alpha, \ldots, \alpha)
\quad \text{of}\quad
T\coloneq \sum_{l\geqslant 2}\,
\vcenter{\hbox{\begin{tikzpicture}[scale=0.6]
        \draw[fill=white, thick] (0,0) circle [radius=22pt];
        \draw[thick] (0,22pt)--(0,12pt);
        \node at (0,2pt) {\scalebox{1}{$a^l\psi^l$}};
        \end{tikzpicture}}}
\in  \psi^2 A[[\psi]]
\end{equation}
coincides with the translation action on CohFTs \emph{without unit}.
\end{proposition}

\begin{proof}
Since $T$ is an arity $1$ operation, only the term $n=1$ in Formula~\eqref{eq:TranslationGivental} can contribute non-trivially. So the pre-Lie exponential resumes to a classical exponential and the action on $\alpha$ boils down to
\[\big(\exp (\varPsi(T))(\alpha)\big)_{g,n}=
\sum_{m\geqslant 0}
\sum_{l_1, \ldots, l_m\geqslant 2}
\frac 1{m!} (\pi_*)^m
\left\langle a^{l_m},\dots,a^{l_1}; \alpha_{g,n+m}\right\rangle
\psi_{n+m}^{l_1}\ldots \psi_{n+1}^{l_m}~,\]
where $\left\langle a^{l_m},\dots,a^{l_1}; \alpha_{g,n+m}\right\rangle$ amounts to applying the pairing to all the pairs made up of $a^{l_i}$ and the element of $A$ labeling the $(n-i+1)$th leaf, for $1\leqslant i\leqslant m$.
This formula coincides with the class given in~\cite[Section 6.3]{Teleman2012} and~\cite[Definition 2.5]{PandPixZvo2015}.
\end{proof}

The last type of generators of the Givental group are the Manin--Zograf elements $\ell$. Comparing their definition
with the discussion of a twist with Hodge characters and kappa-classes in~\cite[Sections 1.6 and 8.2-8.3]{Teleman2012}, we immediately get  the following statement.

\begin{proposition}\label{prop:ManinZografAction} The action $\ell \cdot \alpha$ by a Manin--Zograf element $\ell$ on a CohFT $\alpha$ \emph{without unit}
amounts to multiplying each component $\alpha_{g,n}$ by $\ell_{g,n}$~.
\end{proposition}

\begin{remark}
In fact, this part of the action is not really studied separately in the literature, since for the CohFTs with a unit (and even for CohFTs with a possibly non-flat unit) this action is equal, on the Lie algebra level, to a linear combination of the other elements of the Givental Lie algebra, see~\cite[Proposition 8.3]{Teleman2012}.
However, in the case where we have no unit at all (and no semi-simplicity to use pieces of Teleman's classification), the action of Manin--Zograf elements cannot be reduced to the action of the other generators of the Givental Lie algebra.
\end{remark}

\subsection{The quantum version of the tautological graph complex}
In this section, we apply the methods introduced in \cref{sec:GTgroup}, and more precisely in \cref{subsec:HoAction} and in \cref{sec:UDGGTgroup}, to the operad $\GGra_A$ instead of $\Gra$ and to the $\GGra_A$-algebra $\g_A$ instead of any $\Gra$-algebra~$\g$. Recall that we work under \cref{ass:fdnd}.
By a slight abuse, we use the same notations
\[
\vcenter{\hbox{
\begin{tikzcd}
\S \Delta \Lie
\arrow[rr, "\Phi"] \arrow[dr, "\Theta"']
& & \End_{\g_A}\ .
 \\
&\GGra_A  \arrow[ur, "\Psi"']  &
\end{tikzcd}
}}
\]
By construction, the  deformation complex $\aa_{\GGra_A}^{\theta}$ of the morphism $\Theta$ carries a complete dg Lie algebra structure.

\begin{definition}[Givental--Grothendieck--Teichm\"uller group/Lie algebra]
The induced (complete) Lie algebra
\[\givgrt\coloneq \left(H_\ev\Big(
\aa_{\GGra_A}^{\theta}\Big), [\,,] \right)\]
is called the \emph{Givental--Grothendieck--Teichm\"uller Lie algebra}. Its integration group
\[\GIVGRT\coloneq \left( H_\ev\Big(
\aa_{\GGra_A}^{\theta}\Big), \BCH,0\right)\]
is called the \emph{Givental--Grothendieck--Teichm\"uller group}.
\end{definition}

\begin{theorem}\label{th:ActOnQCohFT}
The Givental--Grothendieck--Teichm\"uller group acts functorially
\[\GIVGRT  \to
\Aut\left(
\calMC\left(\g^\hbar_A\right)
\right)\]
on the moduli spaces of quantum homotopy CohFTs under the pre-Lie exponential formula
\[\lambda \cdot \alpha \coloneqq \sum_{\substack{n\geqslant 1\\ k\geqslant 0}}
{\textstyle \frac{1}{n!}}
\left(e^{\Psi(\lambda)}\right)^k_n(\alpha, \ldots, \alpha)\, \hbar^k\ .\]
\end{theorem}

\begin{proof}
The same methods as in \cref{sec:GTgroup}, especially \cref{cor:MorphBCHhomo}, apply \emph{mutatis mutandis}.
Like in the proof of \cref{th:ActOnCohFT}, the assignment $\GIVGRT  \to
\Aut\left(
\calMC\left(\g^\hbar_A\right)
\right)$ defines a natural transformation between two functors from the category of dg symmetric vector spaces  to the category of groups.
\end{proof}

The underlying graded vector space of $\aa^\theta_{\GGra_A}$
 is equal to $\prod_{n\geqslant 1} \GGra_A(n)^{\mathbb{S}_n}[[\hbar]]$ with the differential given by
\[
\Gamma \mapsto
\DE (\Gamma) + \left[
\vcenter{\hbox{\begin{tikzpicture}[scale=0.5]
		\draw[thick] (0,0)--(1.5, 0);
		\draw[fill=white, thick] (0,0) circle [radius=10pt];
		\draw[fill=white, thick] (1.5,0) circle [radius=10pt];
		\end{tikzpicture}}}
+ \hbar\,
\vcenter{\hbox{\begin{tikzpicture}[scale=0.5]
		\draw[thick]  (-0.4,2.7) arc [radius=0.4, start angle=180, end angle= 0];
		\draw[thick] (0, 2) to [out=135,in=270] (-0.4,2.7);
		\draw[thick] (0, 2) to [out=45,in=270] (0.4,2.7);
		\draw[fill=white, thick] (0,2) circle [radius=10pt];
		\node (n) at (0,1.5) {};
		\end{tikzpicture}}}\,,\, \Gamma\right]\ .
\]
Like the universal deformation group $\mathrm{G}$ of \cref{sec:UDGGTgroup} but unlike the Givental group/Lie algebra, that we were able to fully compute in \cref{cor:GiventalGroup}, we are only able to compute some non-trivial parts of the homology groups. The following theorem justifies the naming conventions.

\begin{theorem}\label{thm:MainGGT}\leavevmode
\begin{enumerate}
\item
The canonical inclusion of operads $\Gra \hookrightarrow \GGra_A$ induces an embedding at the homology level
\[H_\bullet\left(\aa_{\Gra}^\theta\right)\hookrightarrow H_\bullet\left(\aa_{\GGra_A}^\theta\right)~.\]
The chain map
\[
\b^\vartheta_{\GGra_A^\omega}\hookrightarrow \aa_{\GGra_A}^\theta\ \ \text{given by} \ \
\Gamma\mapsto \hbar^{b_1(\Gamma)}\Gamma~,\]
where $b_1(\Gamma)$ is the genus of the underlying graph of $\Gamma$, i.e. its first  Betti number, induces a
embedding at the homology level
\[H_\bullet\left(\b^\vartheta_{\GGra_A^\omega}\right)\hookrightarrow H_\bullet\left(\aa_{\GGra_A}^\theta\right)~.\]

\item The various deformation groups are related as  follows
\[
\vcenter{\hbox{
\begin{tikzcd}
\GRT_1
\arrow[r, hook]
& \mathrm{G} \arrow[r, hook]  &
 \GIVGRT
&\arrow[l, hook']  \GIV
\end{tikzcd}~.
}}
\]
\end{enumerate}
\end{theorem}

\begin{proof}
Regarding the first point, it is straightforward to check that the two maps
\[\aa^\theta_{\Gra}\hookrightarrow \aa_{\GGra_A}^\theta \quad
\text{and} \quad
\b^\vartheta_{\GGra_A^\omega}\hookrightarrow \aa_{\GGra_A}^\theta\]
preserve the respective differentials. In the first case, the latter chain complex splits with respect to the former one:
\[\aa_{\GGra_A}^\theta \cong \aa^\theta_{\Gra} \oplus K~,\]
where $K$ is made up of series of graphs where at least one contains either a non-trivial label at a vertex, a dashed edge, or a usual edge labeled by an element of $\overline{\Mrk}$ (under the notation of \cref{ass:fdnd}).
In the second case, the homology groups of the source chain complex have been fully computed in \cref{cor:GiventalGroup}; their representatives \eqref{eq:cohom-GiventFull} provide us with the cycles of
$\aa_{\GGra_A}^\theta$ whose leading terms are symmetric tensors of
\begin{align} \label{eq:GiventalElementshbar}
&  \sum_{k\geqslant 0}\,
\vcenter{\hbox{\begin{tikzpicture}[scale=0.6]
        \draw[fill=white, thick] (0,0.3) circle [radius=30pt];
        \draw[thick] (0,28pt)--(0,48pt);
        \node at (0,10pt) {\scalebox{1}{$a_1^k\psi^{k+1}$}};
        \node at (0,62pt) {\scalebox{1}{$a_2^k$}};
        \end{tikzpicture}}}
        - \sum_{i,j \geqslant 0} (-1)^j \,
\hbar\vcenter{\hbox{\begin{tikzpicture}[scale=0.9]
  \draw[dashed,thick]  (-0.4,2.7) arc [radius=0.4, start angle=180, end angle= 0];
  \draw[dashed,thick] (0, 2) to [out=135,in=270] (-0.4,2.7);
  \draw[dashed,thick] (0, 2) to [out=45,in=270] (0.4,2.7);
   \draw[fill=white, thick] (0,2) circle [radius=8pt];
  \node at (-0.7,2.4) {\scalebox{1}{$\psi_1^i$}};
  \node at (0,3.4) {\scalebox{1}{$a_1^{i+j}\otimes a_2^{i+j}$}};
  \node at (0.7,2.4) {\scalebox{1}{$\psi_2^j$}};
\end{tikzpicture}}}
         - \sum_{i,j \geqslant 0} (-1)^j \,
    \vcenter{\hbox{\begin{tikzpicture}[scale=0.75]
        \draw[dashed,thick] (0,0)--(5, 0);
        \draw[fill=white, thick] (0,0) circle [radius=10pt];
        \draw[fill=white, thick] (5,0) circle [radius=10pt];
        \node at (0.7,0.4) {\scalebox{1}{$\psi_1^i$}};
        \node at (2.5,0.4) {\scalebox{1}{$a_1^{i+j}\otimes a_2^{i+j}$}};
        \node at (4.3,0.4) {\scalebox{1}{$\psi_2^j$}};
        \node at (1, -0.8) { };
        \end{tikzpicture}}}
\ ,
\\ & \notag
\hbar\vcenter{\hbox{\begin{tikzpicture}[scale=0.9]
  \draw[thick]  (-0.4,2.7) arc [radius=0.4, start angle=180, end angle= 0];
  \draw[thick] (0, 2) to [out=135,in=270] (-0.4,2.7);
  \draw[thick] (0, 2) to [out=45,in=270] (0.4,2.7);
   \draw[fill=white, thick] (0,2) circle [radius=8pt];
  \node at (0,3.4) {\scalebox{1}{$a_1\otimes a_2$}};
\end{tikzpicture}}}\ ,
\ \
\sum_{l\geqslant 2}\,
\vcenter{\hbox{\begin{tikzpicture}[scale=0.6]
        \draw[fill=white, thick] (0,0) circle [radius=22pt];
        \draw[thick] (0,22pt)--(0,12pt);
        \node at (0,2pt) {\scalebox{1}{$a^l\psi^l$}};
        \end{tikzpicture}}} \ ,
\ \
\text{and}
\ \
\vcenter{\hbox{\begin{tikzpicture}[scale=0.75]
        \draw[fill=white, thick] (0,0) circle [radius=10pt];
        \node at (0,0) {\scalebox{0.8}{$\ell$}};
        \end{tikzpicture}}}\ ,
\end{align}
for all elements $a_1^{k}\otimes a_2^{k} \in Z_\bullet\left(A^{\otimes 2}\right)$ such that
$\left(a_1^{k}\otimes a_2^{k}\right)^{(12)}=(-1)^{k} a_1^{k}\otimes a_2^{k}$, for $k\geqslant 0$,
all elements $a_1\otimes a_2 \in Z_\bullet\left(A^{\otimes 2}\right)^{\Sy_2}$,
all elements $a^l\in Z_\bullet(A)$, for $l\geqslant 0$, and all element $\ell\in L_{MZ}$~.
Notice that the image of the differential of $\aa_{\GGra_A}^\theta$
contains neither one-vertex graphs without tadpoles nor one-vertex graphs with a usual tadpole
labeled by an element not containing $\psi_1$ and $\psi_2$. This implies that such cycles cannot be exact.

\medskip

The second point is a direct consequence of the first one since both maps
$\aa^\theta_{\Gra}\hookrightarrow \aa_{\GGra_A}^\theta$
and
$\b^\vartheta_{\GGra_A^\omega}\hookrightarrow \aa_{\GGra_A}^\theta$
preserve the respective complete dg Lie algebra structures.
\end{proof}

\medskip

Using the Givental--Grothendieck--Teichm\"uller group action, we can now offer a more conceptual  interpretation of the Buryak--Rossi functor of \cref{subsec:Quantum} that maps homotopy CohFTs to homotopy quantum CohFTs.
To that end, we first construct a right inverse to the morphism of shifted dg Lie algebras
$\g_A^\hbar\to\g_A$ corresponding to the evaluation $\hbar=1$~. (Notice the similarity with the abovementioned map $\b^\vartheta_{\GGra_A^\omega}\hookrightarrow \aa_{\GGra_A}^\theta$.)

\begin{lemma}\label{lem:Xi}
The map $\Xi\colon\g_A\to \g_A^\hbar$,  defined on each decorated graph $\gamma$ by setting
 \[
\Xi(\gamma)\coloneq\hbar^{b_1(\gamma) + \sum_{v\in V(\gamma)} g(v)} \gamma~,
 \]
 is a morphism of shifted dg Lie algebras, which is the right inverse to the evaluation at $\hbar=1$~.
\end{lemma}

\begin{proof}
This is a direct computation using the fact that the powers of $\hbar$ in $\g_A^\hbar$ correspond to the total genus.
\end{proof}

Let us introduce an auxiliary parameter $s$ and work over the  ring $\K[[s]]$. We consider the following Manin--Zograf element
\[
\widehat{F} \coloneqq \sum_{j\geqslant 0} (2j)! s^{2j+1} f_{2j+1}  \ .
\]

\begin{proposition}\label{prop:BRasGGTaction}
For any homotopy CohFT $\alpha \in \MC\left(\g_A\right)$, the evaluation at $s=\hbar^{-1}$ of the Givental--Grothendieck--Teichm\"uller action of $\widehat{F}$ on $\Xi(\alpha)$ is equal to the image of $\alpha$ under the Buryak--Rossi functor:
\[\left(\widehat{F}\cdot \Xi(\alpha)\right)|_{s=\hbar^{-1}}=\BR(\alpha)~.\]
\end{proposition}

\begin{proof}
Notice first that $\Xi(\alpha)\in \MC\left(\g_A^\hbar\right)$ is a quantum homotopy CohFT since $\Xi$ is a shifted dg Lie algebra morphism by \cref{lem:Xi}. So the Givental--Grothendieck--Teichm\"uller action $\widehat{F}\cdot \Xi(\alpha)$ applies and renders yet another homotopy quantum CohFT over $\K[[s]]$.
For each decorated graph $\gamma$ in $\alpha$, the standard relation between the Chern classes and Chern characters implies that the action
$\widehat{F}\cdot \Xi(\alpha)$
multiplies the cohomology class at each vertex $v\in V(\gamma)$ by
\[\exp\left(\sum_{j\geqslant 0} (2j)! s^{2j+1} \mathrm{ch}_{2j+1}\right) = \sum_{i \geqslant 0} \lambda_i s^{i}~,\]
under the notations introduced in \cref{subsec:Quantum}.
Because of the presence of the extra factor of $\hbar^{b_1(\gamma) + \sum_{v\in V(\gamma)} g(v)}$ in the image of $\Xi$, the negative powers of $\hbar$ arising after the evaluation at $s=\hbar^{-1}$ will not remain in the final result. Note that for a vertex $v\in V(\gamma)$, the class  $\hbar^{g(v)}\sum_{i \geqslant 0}^\infty \lambda_i\hbar^{-i}$ is equal to $\lambda^\hbar=\sum_{i \geqslant 0} \lambda_i\hbar^{g(v)-i}$, so we indeed obtain the formula we used to define the Buryak--Rossi functor.
\end{proof}

By a slight abuse of notations, let us denote by $r^\hbar, T$, and $\ell$
the images of $r, T$, and $\ell$
 (of \cref{prop:RecoverGivTel}, \cref{prop:translation}, and \cref{prop:ManinZografAction} respectively)
under the chain map
$\b^\vartheta_{\GGra_A^\omega}\hookrightarrow \aa_{\GGra_A}^\theta$; they are depicted in \cref{eq:GiventalElementshbar}.
\cref{thm:MainGGT} ensures that $r^\hbar, T$, and $\ell$ define respectively a \emph{Givental} action, a \emph{translation} action and a \emph{Manin--Zograf} action on homotopy quantum CohFTs.

\begin{corollary}The Buryak--Rossi functor maps the actions of $r, T$, and $\ell$ on homotopy CohFTs to the action of
$r^\hbar, T$, and $\ell$ on quantum homotopy CohFTs.
\end{corollary}

\begin{proof}
This is a direct corollary of the interpretation of the Buryak--Rossi functor given in \cref{prop:BRasGGTaction} in term of the Givental--Grothendieck--Teichm\"uller action of the Manin--Zograf element $\widehat{F}$, since it commutes
with all the other elements in the Lie algebra $\mathfrak{L}$, see \cref{eq:ManinZograf}.
 We also note that the powers of $\hbar$ match due to the extra factor of $\hbar$ in front of the genus $1$ term in the formula of $r^\hbar$~.
\end{proof}

To conclude,
the full computation of the Givental--Grothendieck--Teichm\"uller group and its action on quantum homotopy CohFTs, like its applications in
the Buryak--Rossi theory of quantization of integrable hierarchies of
topological type,  are promising subjects for future research.

%%%%%%%%%   BIBLIOGRAPHY

\bibliographystyle{alpha}
\bibliography{bib}

\end{document}